\documentclass[12pt,reqno]{amsart}
\usepackage{amssymb}
\usepackage[mathscr]{euscript}
\usepackage{geometry}
\usepackage[latin1]{inputenc}
\usepackage[italian,english]{babel}
\usepackage{amsmath, amsfonts, amsthm}
\usepackage{graphicx}
\usepackage{pgfplots}
\usepackage{paralist}
\usepackage{mathtools, mathabx,bigints}
\usepackage{subfig}
\usepackage{comment}
\usepackage{hyperref}
\usepackage{siunitx}
\usepackage{xcolor}
\usepackage{soul}
\usepackage{cleveref}
\numberwithin{equation}{section}
\newtheorem{thm}{Theorem}[section]
\newtheorem{lem}[thm]{Lemma}
\newtheorem{cor}[thm]{Corollary}
\newtheorem{prop}[thm]{Proposition}
\newtheorem{property}{Property}

\newtheorem{defn}[thm]{Definition}
\theoremstyle{definition}
\newtheorem{rem}[thm]{Remark}
\newtheorem{notation}{Notation}
\theoremstyle{remark}

\newcommand{\norm}[1]{\left\Vert#1\right\Vert}

\newcommand{\abs}[1]{\left\vert#1\right\vert}

\newcommand{\nn}{\mathbf{\hat{n}}}

\newcommand{\Hdiv}{H(\textrm{div},D)}

\newcommand{\HS}[2]{H^{#1}(#2)}

\newcommand{\Jo}{\mathbf{J}_{\varnothing}}
\newcommand{\sigo}{\sigma_{\varnothing}}

\setstcolor{red}

\DeclareMathOperator{\esssup}{ess\,sup}
\DeclareMathOperator{\essinf}{ess\,inf}

{\left\{\begin{array}{@{}l@{}}}{\end{array}\right.}
\patchcmd{\abstract}{\scshape\abstractname}{\textbf{\abstractname}}{}{}
\makeatletter 
\def\@makefnmark{} 
\makeatother 
 
\pgfplotsset{compat=1.18}
\begin{document}
\title{The Kernel Method for Electrical Resistance Tomography}
\author[A. Tamburrino, V. Mottola]{
Antonello Tamburrino$^{1,2}$, Vincenzo Mottola$^1$}\footnote{\\$^1$Dipartimento di Ingegneria Elettrica e dell'Informazione \lq\lq M. Scarano\rq\rq, Universit\`a degli Studi di Cassino e del Lazio Meridionale, Via G. Di Biasio n. 43, 03043 Cassino (FR), Italy.\\
Email: antonello.tamburrino@unicas.it {\it (corresponding author)}, vincenzo.mottola@unicas.it.}
\footnote{$^2$ EUT+ Institute of Nanomaterials and Nanotechnologies-EUTINN, European University of Technology, European Union.}

\begin{abstract} This paper treats the inverse problem of retrieving the electrical conductivity of a material starting from boundary measurements in the framework of Electrical Resistance Tomography (ERT). In particular, the focus is on non-iterative reconstruction methods suitable for real-time applications. 

In this work, the Kernel Method, a new non-iterative reconstruction method for Electrical Resistance Tomography, is presented. The imaging algorithm addresses the problem of retrieving one or more anomalies of arbitrary shape, topology, and size embedded in a known background (the inverse obstacle problem).

The foundation of the Kernel Method is based on the idea that if a proper current density applied at the boundary (Neumann data) of the domain exists such that it is able to produce the same measurements with and without the anomaly, then this boundary source produces a power density that vanishes in the region occupied by the anomaly, when applied to the problem involving the background material only.

This new tomographic method has a simple numerical implementation that requires a very low computational cost.

In this paper, the theoretical foundation of the Kernel Method is provided, and an extensive numerical campaign proves the effectiveness of this new imaging method.


\end{abstract}
\maketitle
\markright{THE KERNEL METHOD}

\section{Introduction} 
Electrical Resistance Tomography (ERT) is a widely used tomographic inspection method that aims to retrieve the electrical conductivity $\sigma$ within a physical body from boundary measurements of currents and voltages in steady-state operation. Throughout the years, it has been successfully applied to a wide range of different applications such as the petroleum and chemical process industry~\cite{Ch19, Du15, Ma09}, medicine~\cite{Ho94} (for cranial imaging \cite{Tarassenko1984-im}, lung monitoring \cite{Meier2008-iu,Adler1995ImagingOP}, thermal monitoring of hyperthermia treatment \cite{Conway1987-xj}, paediatric lung disease \cite{Pham2011-po}, breast imaging \cite{Zain2014-sc} and brain imaging \cite{Holder1996-hn}), semiconductor manufacturing~\cite{Kr03}, civil engineering~\cite{Jo01}, carbon nanotube manufacturing~\cite{Ho07}, biological culture analysis~\cite{Li08}, etc.

Different approaches to solving the ERT problem can be found in the literature, leading to two main classes: iterative and non-iterative imaging methods. Most reconstruction methods are based on iterative minimisation of a proper cost function that encodes the \lq\lq distance\rq\rq \ between the measured and synthetic data for the current estimate of $\sigma$, and the proper a priori information. A list of references to this kind of approach can be found in~\cite{Lionheart_2004}. These methods can achieve good quality for the reconstructions. Still, they are (i) affected by the local minima problem, i.e. they are prone to false solutions, and (ii) computationally expensive because of the ill-posedness and the nonlinearity of the underlying inverse problem that calls for a large number of iterations.

The other class is that of non-iterative methods, where a proper indicator function is identified and utilised to provide reconstructions in just one step. The main advantage of this approach is its ability to achieve real-time performance, a feature as rare as it is required in applications. Several non-iterative methods are available in the literature. The first method proposed was the Linear Sampling Method by Colton and Kirsch \cite{LSM}, followed by the Factorization Method by Kirsch \cite{art:Ki98, Br01} and the Enclosure Method \cite{Ik99} proposed by Ikehata. A method based on the MUSIC signal processing algorithm was also introduced by Devaney in~\cite{Devaney2000}, while the Monotonicity Principle Method was proposed by Tamburrino and Rubinacci in 2002~\cite{Ta02,MPM_2025}. 

This paper contributes to enlarging the class of non-iterative methods for ERT. Specifically, a new real-time inversion method is proposed for the inverse obstacle problem, aiming to retrieve the shape, position, and dimensions of one or more anomalies of arbitrary size and topology embedded in a known background.

The method is based on some peculiar properties of the Neumann-to-Dirichlet operator, i.e., the operator  $\Lambda : g \mapsto \varphi|_{\partial \Omega}$ mapping the electrical current injected into the domain $\Omega$ through its boundary $\partial \Omega$ and the corresponding scalar potential $\varphi$ evaluated on $\partial \Omega$.
The foundation of the proposed method is that when an electrical current density $g$ applied to the boundary of the domain $\Omega$ makes the difference between the measured boundary voltages with and without the anomaly negligible, then the electrical current density circulating within the conducting domain $\Omega$ does not interact with the anomalous region $D$. This corresponds to having an ohmic power density $p_{\varnothing}(x)$ for the reference configuration without the anomaly, which disappears in $D$.

The boundary data $g$ is required to solve $\Lambda_D g =  \Lambda_{\varnothing}g$, i.e. $g$ is in the \emph{kernel} of the operator $\Lambda_D - \Lambda_{\varnothing}$, giving the name of the proposed method.
Nevertheless, due to the Unique Continuation Principle, the kernel of $\Lambda_D - \Lambda_{\varnothing}$ consists of the trivial element only. However, it is possible to make the difference $\Lambda_D g -  \Lambda_{\varnothing}g$
arbitrarily small by considering the sequence of the eigenfunctions of $\Lambda_D - \Lambda_{\varnothing}$. The sequence of the eigenfunction is then shown to make the power density $p_{\varnothing}(x)$ disappear faster in the region occupied by the anomaly than in the external regions.

The roadmap to establish an imaging method is based on two main milestones. First, a characterisation of the unknown is given in terms of proper sequences, termed as \lq\lq kernel sequences\rq\rq, that is, proper sequences of boundary data making $\Lambda_D g -  \Lambda_{\varnothing}g$ arbitrarily small. This characterisation forms the theoretical foundation of the proposed method and is presented in \Cref{sec:pbg_eig}. Second, the characterisation of the anomalies found in \Cref{sec:pbg_eig} requires testing the boundedness of a proper ratio for \emph{all} kernel sequences. This is impractical, from an implementation perspective, but in \Cref{sec_ti} it is found that, under the well-known assumptions appearing in the Factorisation Method, the test can be reduced to consider only the sequence $\{g_n \}_{n \in \mathbb{N}}$ of the eigenfunctions of the operator $\Lambda_D g -  \Lambda_{\varnothing}g$. This is a remarkable fact, and a specific Section (\Cref{sec_eig_fun}) is devoted to understanding the reason why the sequence of the eigenfunctions $\{g_n \}_{n \in \mathbb{N}}$ appears to be so special. Specifically, it is found that $\{g_n \}_{n \in \mathbb{N}}$ represents a very large class of all kernel sequences required in the theoretical characterisation of the anomaly of \Cref{sec:pbg_eig}.

The real-time inversion method is then easily derived as (i) evaluate the aforementioned eigenfunctions $g_n$ starting from the measurement of $\Lambda_D$, (ii) compute $p_{\varnothing}^n(x)$, the ohmic power density in $\Omega$ for the reference configuration when a proper $g_n$ is applied to the boundary $\partial \Omega$, and (iii) reconstruct the anomaly as the region in which $p_{\varnothing}^n(x)$ is almost vanishing.

The paper is organised as follows. In Section~\ref{sec:sop} the problem is stated and some important properties of NtD map are recalled; in Section~\ref{sec:pre} the core idea behind the method is presented, together with the role of the Unique Continuation Principle; in Section \ref{sec:pbg_eig} a characterization of the anomaly in terms of proper sequences (Kernel Sequences) is provided, while a discussion of the role played by the eigenfunctions $g_n$ is in \Cref{sec_eig_fun}; in Section \ref{sec_ti} the behaviour of $p_{\varnothing}^n(x)$, when the reference configuration is driven by $g_n$ is investigated; in Section \ref{sec:main} the Kernel Method is presented;
 in Section~\ref{sec:noise}, the treatment of the noise is proposed, together with the final version of the reconstruction algorithm; in Section~\ref{sec:anex}, the method is applied to the first cases of study, where the NtD operator can be analytically evaluated and the data are noise-free. The numerical results are in Section~\ref{sec:num}, while the conclusions are drawn in Section~\ref{sec:conclusions}.

\section{Statement of the problem}
\label{sec:sop}
Throughout this paper, $\Omega\subset\mathbb{R}^n$, with $n\geq2$, is the region occupied by the conducting material. It is assumed that $\Omega$ is a simply connected and bounded domain with a smooth boundary. The normal derivative defined on $\partial\Omega$ is denoted by $\partial_n$.

Hereafter, for the convenience of the reader, the definition of the functional spaces underlying the mathematical formulation of the problem is recalled:
\begin{align*}
H^1_{\diamond}(\Omega)&=\{u\in H^1(\Omega)\,|\,\begingroup\textstyle\int\endgroup_{\partial\Omega}u(x)\,dx=0\} \\
L^2_{\diamond}(\partial\Omega)&=\{u\in L^2(\partial\Omega)\,|\,\begingroup\textstyle\int\endgroup_{\partial\Omega}u(x)\,dx=0\} \\
L^{\infty}_+(\Omega)&=\{u\in L^{\infty}(\Omega)\,|\, \esssup_{x\in\Omega} u(x) >0\}.
\end{align*}

Assuming a linear and isotropic conductive material in $\Omega$, the constitutive relationship relating the electrical current density $\mathbf{J}$ to the electric field $\mathbf{E}$ is $\mathbf{J}(x)=\sigma(x)\mathbf{E}(x)$, where $\sigma\in L_+^{\infty}(\Omega)$ is the electrical conductivity.

In terms of the electric scalar potential $\varphi$, the steady current problem is formulated as
\begin{equation}
\begin{cases}
\nabla\cdot\left(\sigma(x)\nabla \varphi(x)\right)=0 & \text{in $\Omega$}, \\
\sigma(x) \partial_n \varphi(x)=g(x) & \text{on $\partial\Omega$},
\end{cases}
\label{eqn:dirprob}
\end{equation}
where $\varphi\in H^1_{\diamond}(\Omega)$, $g\in L^{2}_{\diamond}(\partial\Omega)$ is the prescribed normal component of the current density entering the domain $\Omega$, and $\mathbf{E}(x)=-\nabla \varphi(x)$.

Problem~\eqref{eqn:dirprob} is meant in the weak form, that is,
\begin{equation}\label{eqn:weakprob}
\int_{\Omega}\sigma(x)\nabla \varphi(x)\cdot\nabla \psi(x)\,dx=\int_{\partial\Omega}g(x)\psi(x)\,dx \quad \forall \psi\in H^1_{\diamond}(\Omega).
\end{equation}

The Neumann-to-Dirichlet (NtD) operator, which plays a key role since it represents the boundary measurements, maps the current flux $g$ applied to the boundary $\partial \Omega$, to the corresponding electric scalar potential evaluated on $\partial \Omega$, i.e., $\varphi\rvert_{\partial\Omega}$
\begin{equation*}
\Lambda_{\sigma}:g\in L^{2}_{\diamond}(\partial\Omega) \to \varphi\rvert_{\partial\Omega}\in L^{2}_{\diamond}(\partial\Omega).
\end{equation*}
The NtD operator has the following key properties.
\begin{property} [see \cite{Br01, Gi90}]
\label{prop:compact}
$\Lambda_{\sigma} :L^2_{\diamond}(\partial\Omega)\to L_{\diamond}^2(\partial\Omega)$ is compact, self-adjoint and positive.
\end{property}

\begin{property} [see \cite{Gi90}]
\label{prop:mono}
    $\Lambda_{\sigma}$ is monotonically decreasing in $\sigma$, i.e,
\begin{displaymath}
\left\langle g,\Lambda_{\sigma_2}(g)\right\rangle\geq\left\langle g,\Lambda_{\sigma_1}(g)\right\rangle\quad\text{for $\sigma_2\leq\sigma_1$},
\end{displaymath}
where $\sigma_2\leq\sigma_1$ means that $\sigma_2(x)\leq\sigma_1(x)\text{ a.e. in $\Omega$}$.
\end{property}

\begin{lem}[see \cite{Ge08, HaUl13, art:ik98}]
\label{lem:estimate}
Let $\sigma_1, \sigma_2 \in L_+^{\infty}(\Omega)$ be two electrical conductivities and let $\varphi_i\in H^1(\Omega)$ be the solution of the Neumann problem
\begin{equation*}
\nabla\cdot\left(\sigma_i(x)\nabla \varphi_i(x)\right)=0\: \text{in $\Omega$} \quad \text{and} \quad \sigma_i(x)\partial_n \varphi_i(x)=g(x)\: \text{on $\partial\Omega$},
\end{equation*}
where $g\in L^{2}_{\diamond}(\partial\Omega)$ is the applied boundary current density,
then
\begin{equation*}
       \int_{\Omega}\left(\sigma_2-\sigma_1\right)\abs{\nabla \varphi_2}^2\,dx\leq\left\langle(\Lambda_{\sigma_1}-\Lambda_{\sigma_2})(g),g\right\rangle\leq \int_{\Omega} \frac{\sigma_2}{\sigma_1}\left(\sigma_2-\sigma_1\right)\abs{\nabla \varphi_2}^2\,dx
    \end{equation*}
\end{lem}

The inverse problem in Electrical Resistance Tomography, as stated by Calder\`on  \cite{art:Cal80}, is to reconstruct the electrical conductivity $\sigma$ from the knowledge of $\Lambda_{\sigma}$. Given this framework, the present contribution focuses on the problem of retrieving the shape of one or more anomalies  $D_1,D_2,\ldots $, embedded in a known background. In other words, let $D=\bigcup_j D_j$, an electrical conductivity $\sigma_D$ of the form
\begin{equation}\label{eqn:sd}
\sigma_D(x)=
\begin{cases}
\sigma_{bg}(x) & x\in\Omega\setminus D, \\
\sigma_a(x) & x\in D,
\end{cases}
\end{equation}
is considered, where $\sigma_{bg}$ is known, whereas the domains $D_1,D_2,\ldots $ are unknown. 

Defining
\begin{equation}\label{eqn_bound_bg}
    \sigma_{bg}^m=\essinf_{x\in\Omega}{\sigma_{bg}(x)}, \quad \sigma_{bg}^M=\esssup_{x\in\Omega}{\sigma_{bg}(x)}
\end{equation}
and
\begin{equation}\label{eqn_bound_a}
    \sigma_a^m=\essinf_{x\in\Omega}{\sigma_{a}(x)}, \quad \sigma_a^M=\esssup_{x\in\Omega}{\sigma_{a}(x)},
\end{equation}
it is required that $\sigma_a$ and $\sigma_{bg}$ are well-separated, in the sense that
\begin{displaymath}
    \sigma_a^M(x)<\sigma_{bg}^m \quad \text{or} \quad \sigma_{a}^m(x)>\sigma_{bg}^M(x).
\end{displaymath}

The target is to reconstruct the shape $D$ of the region occupied by the anomalies. Each $D_i$ is assumed to be connected and open, the $D_i$s are assumed to be mutually disjoint, and $\overline{D}_i\Subset\Omega$.

\begin{notation}\label{not_1}
    Throughout the paper, $\sigma_A$ denotes an electrical conductivity of the form \eqref{eqn:sd}, where the conductivity differs from the background conductivity $\sigma_{bg}$ only within an open subset $\overline{A} \Subset \Omega$. In particular, $\sigma_{\varnothing}$ denotes the conductivity that coincides with $\sigma_{bg}$ throughout the whole domain $\Omega$.

    Accordingly, $\varphi_A$ denotes the solution to problem \eqref{eqn:dirprob} corresponding to the conductivity $\sigma_A$, and $\Lambda_A$ denotes the associated NtD operator.
\end{notation}
 
\section{UCP and Kernel Method}\label{sec:pre}
The starting point for the proposed method, hereafter referred to as the Kernel Method (KM), is inspired by physical arguments.

Consider two different configurations: the actual one in the presence of the anomalies, described by the electrical conductivity $\sigma_D$, and a second configuration with electrical conductivity $\sigma_{\varnothing}$, i.e. $D=\varnothing$. The second configuration is termed the \emph{reference configuration}.

The Kernel Method is based on the following ideas. First, if, for a proper Neumann boundary condition $g$, the electrical current density in $\Omega$, both in the actual and reference configurations, does not flow through the anomalous region $D$, then the voltage measurements on the boundary are the same, i.e., $\Lambda_D g=\Lambda_{\varnothing}g$.
Second, if there exists a boundary data $g$ that gives the same boundary scalar potential with and without the anomaly, i.e. such that $\Lambda_Dg=\Lambda_{\varnothing}g$, then the corresponding electric current density for the reference configuration vanishes in $D$, as proved in the following results.

\begin{lem}\label{lem:mainl}
Let $\sigma_D$, $\sigma_{\varnothing}$, $\Lambda_D$, $\Lambda_{\varnothing}$, and $\varphi_{\varnothing}$ be defined as in \Cref{not_1}. Then
\begin{equation}\label{eqn:ker}
	\langle\left(\Lambda_D-\Lambda_{\varnothing}\right)g,g\rangle=0 \Longleftrightarrow \int_D\sigo(x) \abs{\nabla \varphi_{\varnothing}(x)}^2\,dx=0.
\end{equation}
\end{lem}
\begin{proof}
Without loss of generality, only the case $\sigma_a^M<\sigma_{bg}^m$ is considered. The other case can be treated with similar arguments.

Applying Lemma \ref{lem:estimate} for $\sigma_1=\sigma_D$ and $\sigma_2=\sigo$, and with some straightforward manipulations, it follows
\begin{equation}\label{eqn:kmequiv}
        k_l \int_{D}\sigma_{bg}\abs{\nabla \varphi_{\varnothing}}^2\,dx\leq\left\langle(\Lambda_D-\Lambda_{\varnothing})g,g\right\rangle\leq k_u\int_{D}\sigma_{bg}\abs{\nabla \varphi_{\varnothing}}^2\,dx.
    \end{equation}
with $k_l$ and $k_u$ two constants independent of the applied boundary data $g$ and the shape or position of the anomaly $D$, specifically
\begin{equation*}
k_l=\frac{\sigma_{bg}^m-\sigma_a^M}{\sigma_{bg}^M},\quad k_u=\frac{\sigma_{bg}^M-\sigma_a^m}{\sigma_a^m}.
\end{equation*}
Since both $k_l$ and $k_u$ are positive because of the assumption $\sigma_a^M<\sigma_{bg}^m$, equation \eqref{eqn:kmequiv} implies that
\begin{equation*}
\left\langle(\Lambda_D-\Lambda_{\varnothing})g,g\right\rangle=0 \Longleftrightarrow \int_D \sigma_{bg} \abs{\nabla \varphi_{\varnothing}}^2\,dx=0.
\end{equation*}
\end{proof}

\begin{cor}\label{cor_1}
Let $\sigma_D$, $\sigma_{\varnothing}$, $\Lambda_D$, $\Lambda_{\varnothing}$, and $\varphi_{\varnothing}$ be defined as in \Cref{not_1}. Then
\begin{equation*}
    \Lambda_Dg=\Lambda_{\varnothing}g \Longleftrightarrow \int_D\sigo(x) \abs{\nabla \varphi_{\varnothing}(x)}^2\,dx=0.
\end{equation*}
\end{cor}
\begin{proof}
    Assuming $\sigma_D < \sigma_{bg}$, the Monotonicity Principle (see Property \ref{prop:mono}) gives $\Lambda_D-\Lambda_{\varnothing}\geqslant 0$. Then, from standard linear algebra arguments, it follows that
\begin{equation}\label{eqn:kmequiv3}
	\Lambda_D g =\Lambda_{\varnothing} g \Longleftrightarrow \langle\left(\Lambda_D-\Lambda_{\varnothing}\right)g,g\rangle=0,
\end{equation}
and therefore, by combining Lemma \ref{lem:mainl} and \eqref{eqn:kmequiv3} the claim follows.
\end{proof}

Summing up, the idea underlying the Kernel Method consists of finding a Neumann boundary condition $g$ such that 
\begin{equation}\label{eqn:idcond}
    \left( \Lambda_D - \Lambda_{\varnothing} \right) g=0,
\end{equation}
and then estimate the anomalous region $D$ as the set of points where the electric current density vanishes when evaluated in the reference configuration driven by the boundary condition $g$. The name Kernel Method comes from the key role that the elements of the kernel of the operator $\Lambda_D-\Lambda_{\varnothing}$ play.

Unfortunately, the Unique Continuation Principle (UCP), which has a central role for Elliptic PDEs, implies that the kernel of the operator $\Lambda_D-\Lambda_{\varnothing}$ is empty, as proved in the following Theorem.

\begin{thm}[Weak UCP~\cite{art:ale12},~\cite{art:gar86}]\label{teo:ucp}
Let $\Omega$ be a connected open subset of $\mathbb{R}^n$, with $n = 2,3$, let $B$ a ball contained in $\Omega$, and $u\in H^1(\Omega)$ the solution of
\begin{displaymath}
    \nabla\cdot(\sigma(x)\nabla u(x))=0 \quad \text{in $\Omega$},
\end{displaymath}
with $\sigma\in L^{\infty}_+(\Omega)$ for $n=2$ or $\sigma$ Lipschitz continuous for $n=3$. If $u$ is constant in $B$, then $u$ is constant in $\Omega$.
\end{thm}

The UCP of Theorem~\ref{teo:ucp} prevents one from finding non-constant Dirichlet data, giving a vanishing electrical field and current density in the ball $B$.
As a consequence, the UCP combined with \eqref{eqn:ker}, gives $\ker \{ \Lambda_D-\Lambda_{\varnothing} \} = \{0\}$.
In fact, if a solution $\varphi_{\varnothing}$ corresponding to an element $g$ of $\ker \{ \Lambda_D-\Lambda_{\varnothing} \}$ is constant in a ball $B$ contained in $D$, then it must be constant on $\Omega$. Consequently, $g=\sigma_{bg} \partial_n \varphi_{\varnothing}$ vanishes on $\partial \Omega$.

This result entails a change in perspective in the imaging method. Rather than looking for a proper boundary datum $g$ such that
$\int_D \sigma_{bg}(x)\abs{\nabla \varphi_{\varnothing}(x)}^2\,dx=0,$
a boundary datum $g$ that is \lq\lq close enough\rq\rq \ to $\ker\{\Lambda_D-\Lambda_{\varnothing}\} = \{0\}$, in the sense that it makes the difference $\langle\left(\Lambda_D-\Lambda_{\varnothing}\right)g,g\rangle$ small enough, is searched.
Indeed, the following Lemma holds.
\begin{lem}\label{lem:mainl_2}
Let $\sigma_D$ be the electrical conductivity defined in \eqref{eqn:sd}, $\sigma_{bg}$ be the prescribed background conductivity, and $\{f_n\}_{n\in\mathbb{N}}\subset L^{2}_{\diamond}(\partial\Omega)$ be a sequence of boundary data on the unit ball $(\norm{f_n}=1)$. Then
\begin{equation}\label{eqn:ker_2}
	\lim_{n\to+\infty} \langle\left(\Lambda_D-\Lambda_{\varnothing}\right) f_n,f_n\rangle =  0 \Longleftrightarrow \lim_{n\to+\infty}\int_D\sigo(x) \abs{\nabla \varphi_{\varnothing}^n(x)}^2\,dx=0,
\end{equation}
where $\varphi_{\varnothing}^n$ is the solution of
\begin{equation*}
	\nabla \cdot \left(\sigma_{\varnothing}\nabla \varphi_{\varnothing}^n\right)=0\:\text{in $\Omega$} \quad \text{and} \quad \sigma_{\varnothing}\partial_n \varphi_{\varnothing}^n=f_n \:\text{on $\partial\Omega$}.
\end{equation*}
\end{lem}
\begin{proof}
    The claim follows by passing to the limit in \eqref{eqn:kmequiv}.
\end{proof}
\begin{rem}
The normalisation condition $\lVert f_n \rVert = 1$ has no impact in proving \eqref{eqn:ker_2}; it is introduced only to exclude trivial sequences that converge to zero or sequences that converge in $L^2(\partial\Omega)$. Regarding the latter, it is worth noting that the limit of a convergent sequence, under the assumption 
$\lim_{n\to+\infty} \langle\left(\Lambda_D-\Lambda_{\varnothing}\right) f_n,f_n\rangle =  0$,
violates the UCP.
\end{rem}
A key concept for the following is that of \emph{kernel sequence}.
\begin{defn}[Kernel Sequence]
A sequence $\{f_n\}_{n\in\mathbb{N}} \subset L^2_\diamond(\partial\Omega)$ satisfying $\lVert f_n \rVert = 1$ for all $n \in \mathbb{N}$ and
\begin{equation*}
    \lim_{n\to+\infty} \langle (\Lambda_D-\Lambda_{\varnothing})f_n,f_n\rangle = 0
\end{equation*}
is called a Kernel Sequence.
\end{defn}

\section{Characterization of the anomaly}\label{sec:pbg_eig}
The previous section shows that kernel sequences can make the power density absorbed by the anomalous region negligibly small. They consist of boundary data that approach the elements of the \lq\lq kernel\rq\rq \ of $\Lambda_D - \Lambda_{\varnothing}$, in some sense. A natural question arising from \Cref{lem:mainl_2} is whether the set appearing on the right-hand side of \eqref{eqn:ker_2} is fully characterised by kernel sequences. The following theorem provides lower bounds to the \emph{outer support} $D^*$ of the anomalous region $D$ in terms of kernel sequences.

The outer support $D^*$ of a set $D$ has been introduced in \cite{HaUl13}. Intuitively speaking,  $D^*$ is the union of $D$ and all its cavities that are not path-connected to the boundary (see Figure 2 in \cite{MPM_2025}).

\begin{thm}[Lower bound]
\label{thm_un}
    If an open set $S$, $\overline{S} \Subset \Omega$, fulfils
    \begin{equation}
    \label{eqn_seq_un}
	\forall \text{ kernel sequence } \{ f_n \}_{n\in\mathbb{N}}, \ \exists K\in\mathbb{R} \ : \ \frac{\int_S\sigo \lvert\nabla \varphi_{\varnothing}^n\rvert^2\,dx}{\langle\left(\Lambda_D-\Lambda_{\varnothing}\right) f_n,f_n\rangle}\leq K, \ 
    \end{equation}
    then $S\subseteq D^*$. In \eqref{eqn_seq_un}, $\varphi^n_{\varnothing}$ is the solution of
    \begin{equation*}
	\nabla \cdot \left(\sigma_{\varnothing}\nabla \varphi_{\varnothing}^n\right)=0\:\text{in $\Omega$} \quad \text{and} \quad \sigma_{\varnothing}\partial_n \varphi_{\varnothing}^n=f_n\:\text{on $\partial\Omega$}.
    \end{equation*}
\end{thm}

\begin{proof}
    Let $S$ be a set that (i) satisfies \eqref{eqn_seq_un} and (ii) is not contained in $D^*$, i.e., $S \backslash D^* \ne \varnothing$.
    From assumption (ii), it follows that there exists a ball $B$, with $B \subset S\setminus D^*$, and from \cite{Ge08} that there exists a sequence of localised potentials $\{ f_n \}_{n\in\mathbb{N}}\subset L^2_{\diamond}(\partial\Omega)$ such that the solutions of
    \begin{equation*}
    \nabla\cdot\left(\sigma_{\varnothing} \nabla \varphi_{\varnothing}^n \right)=0\: \text{in $\Omega$} \quad \text{and} \quad \sigma_{\varnothing} \partial_n \varphi_{\varnothing}^n = f_n \: \text{on $\partial\Omega$},
    \end{equation*}
    satisfy \cite{Ge08}
    \begin{equation}
    \label{eqn_norm_fn}
        \lim_{n \to +\infty} \frac{\int_B \sigo\lvert\nabla \varphi_{\varnothing}^n\rvert^2\,dx}{\int_D \sigo\lvert\nabla \varphi_{\varnothing}^n \rvert^2\,dx} =+\infty.
    \end{equation}
    Each element $f_n$ of the sequence can be assumed to be normalised, since it only matters that the ratio at the left-hand side of \eqref{eqn_norm_fn} approaches infinity.

    The numerators in the ratio that appears in the left-hand side of   \eqref{eqn_norm_fn} form a bounded sequence. Indeed, with reference to set $B$, it turns out that \begin{equation*}
        \int_B \sigo(x) \lvert\nabla \varphi_{\varnothing}^n(x)\rvert^2\,dx \leq k_l^{-1} \langle\left(\Lambda_B-\Lambda_{\varnothing}\right) f_n,f_n\rangle 
        \le k_l^{-1} \lambda_1^B \norm{f_n}^2=k_l^{-1}\lambda_1^B,
    \end{equation*}
    where the first inequality comes from \eqref{eqn:kmequiv} but written with reference to $B$, the second inequality comes from the min-max principle for $\lambda_1^B$, the largest eigenvalue of the compact operator $\Lambda_B-\Lambda_{\varnothing}$, and the rightmost equality exploits the normalisation condition $\norm{f_n}=1$. As a consequence, 
    \begin{equation}
    \label{eqn_pwr}
         \lim_{n\to+\infty} \int_D \sigo\lvert\nabla \varphi_{\varnothing}^n\rvert^2\,dx =0,
    \end{equation}
    since the ratio in \eqref{eqn_norm_fn} tends to $+\infty$. Moreover, the sequence of localised potentials $\{f_n\}_{n\in\mathbb{N}}$ forms a kernel sequence, as follows by combining \eqref{eqn_pwr} and the rightmost inequality of \eqref{eqn:kmequiv}:
    \begin{equation*}
        0\leq \lim_{n\to+\infty} \langle\left(\Lambda_D-\Lambda_{\varnothing}\right) f_n,f_n\rangle \leq k_u \lim_{n\to+\infty} \int_D \sigo\lvert\nabla \varphi_{\varnothing}^n\rvert^2\,dx=0.
    \end{equation*}

    On the other hand, the ratio appearing at the left-hand side of \eqref{eqn_seq_un} is unbounded; indeed
    \begin{equation*}
    \begin{split}
        \frac{\int_S\sigo \lvert\nabla \varphi_{\varnothing}^n\rvert^2\,dx}{\langle\left(\Lambda_D-\Lambda_{\varnothing}\right) f_n,f_n\rangle}
        &\geq \frac{\int_B\sigo \lvert\nabla \varphi_{\varnothing}^n\rvert^2\,dx}{\langle\left(\Lambda_D-\Lambda_{\varnothing}\right)
        f_n,f_n\rangle}\\
        &\geq k_u^{-1}\frac{\int_B\sigo \lvert\nabla \varphi_{\varnothing}^n\rvert^2\,dx}{\int_D\sigo \lvert\nabla \varphi_{\varnothing}^n\rvert^2\,dx} \to+\infty, \text{ for } n \to +\infty.
    \end{split}
    \end{equation*}
\end{proof}
\Cref{thm_un} states that the condition $S \backslash D^* \neq \varnothing$ can be revealed by a proper kernel sequence for which the ratio in \eqref{eqn_seq_un} becomes arbitrarily large. 

\begin{rem}\label{rem_d}
 The condition \eqref{eqn_seq_un} is satisfied by the set $D$, for $K=k_l^{-1}$.
\end{rem}

As highlighted in  \Cref{rem_d}, the condition \eqref{eqn_seq_un} is satisfied by the set $D$. Moreover, even the larger set $D^*$ satisfies the test condition \eqref{eqn_seq_un}, as proved in the following.

\begin{lem}\label{lem_cav_c}
    Let $D\Subset\Omega$ be an open set with a Lipschitz-continuous boundary, and containing a finite number of cavities $C_1,\dots, C_j \subset D$. Moreover, let $\partial C_i$ be made by a single connected component, for $1\leq i \leq j$. 
    Let $g$ be the applied boundary condition and let $\varphi_{\varnothing}$ and $\Jo$ be the corresponding electric scalar potential and the electric current density, respectively.  
    Then
    \begin{align*}
        | \left \langle \mu_i, \Jo \cdot \nn \right\rangle_{\partial C_i} | &\le \| \mu_i \|_{\HS{1/2}{\partial C_i}} \sqrt{\frac{\sigma_{bg}^M}{k_l}}\sqrt{\left\langle(\Lambda_D-\Lambda_{\varnothing})g,g\right\rangle}, \quad \forall\,\mu_i \in \HS{1/2}{\partial C_i} \\
        \| \gamma \left( \varphi_{\varnothing} - \overline{\varphi}_{\varnothing} \right) \|_{H^{1/2}(\partial C_i)}&\leq 
        \| \gamma \| \, \sqrt{1+C_{PW}^2} \,  \sqrt{\frac{1}{k_l\sigma_{bg}^m}}\sqrt{\left\langle(\Lambda_D-\Lambda_{\varnothing})g,g\right\rangle},
    \end{align*}
    for $1\leq i \leq j$,
    where $\gamma : u \in H^1\left(D\right) \to u|_{\partial D} \in H^{1/2}(\partial D)$ is the trace operator, $\| \gamma \|$ is its norm, and $\overline{\varphi}_{\varnothing}$ is the average value of $\varphi_{\varnothing}$ in $D$.
\end{lem}
\begin{proof}
The boundary of $D$ is given by $\partial D=(\cup_{i=1}^j\Gamma_i) \cup \Gamma_e$, where $\Gamma_i= \partial{C_i}$ and $\Gamma_e = \partial D \setminus \Gamma_i$ (see Figure \ref{fig:geo_h} for case $j=1$).
\begin{figure}[h]
    \centering
    \includegraphics[width=0.35\linewidth]{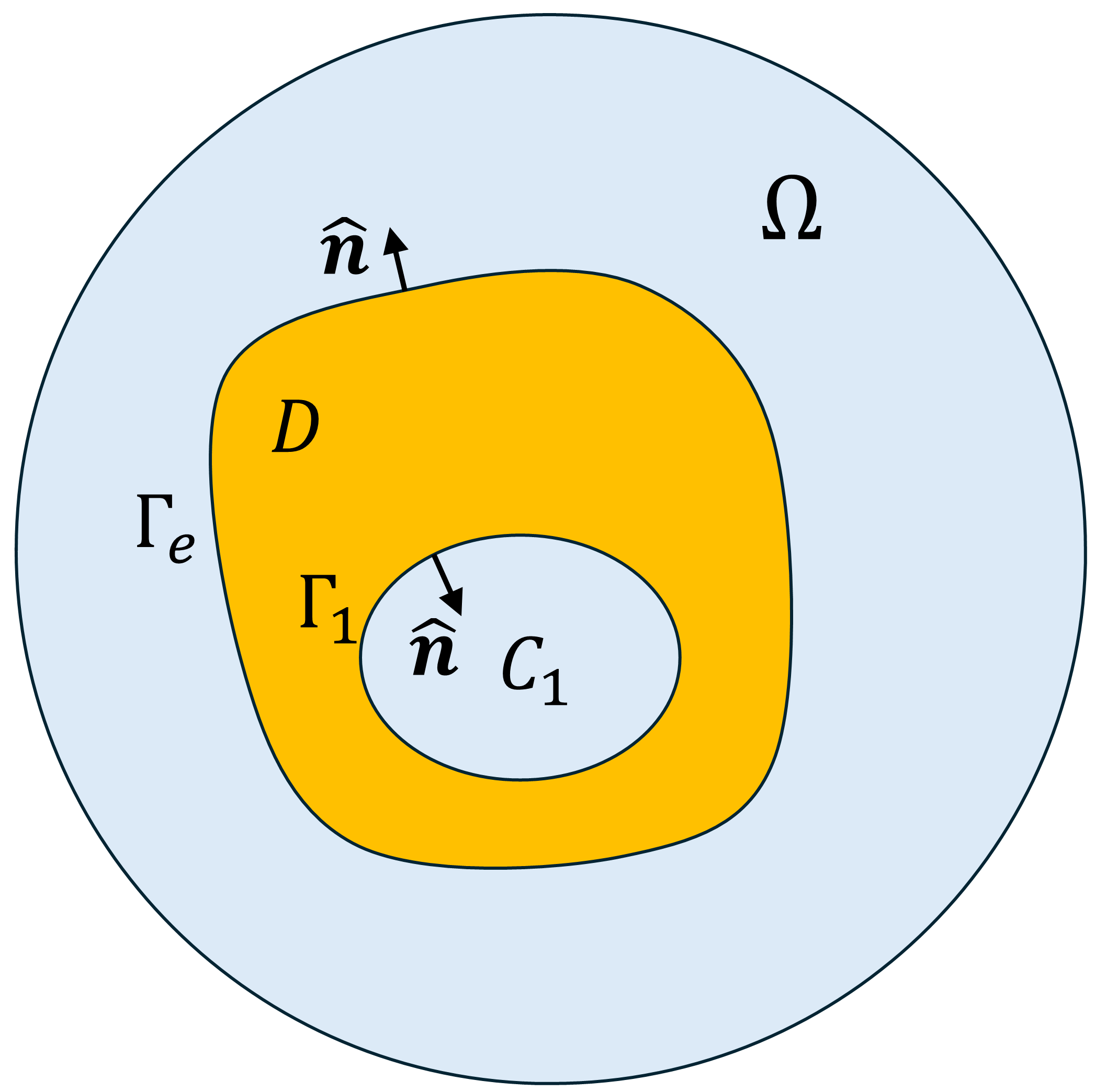}
    \caption{Anomaly $D$ and its cavity $C_1$.}
    \label{fig:geo_h}
\end{figure}

The trace operator $\gamma_n$ that gives the normal component of a div-conforming vector field $\mathbf{v}$ on $\partial D$, i.e., $\gamma_n : \mathbf{v} \in \Hdiv \to \mathbf{v} \cdot \mathbf{\hat{n}} \in H^{-1/2}(\partial D)$ is linear and bounded in $\mathscr{L}(H(div;D);H^{-1/2}(\partial D))$ \cite{girault2012}. Therefore,
\begin{equation}
\label{eq_miq}
\begin{split}
    | \left \langle \mu, \Jo \cdot \nn \right\rangle_{\partial D} | & \le \| \gamma_n \Jo \|_{H^{-1/2}(\partial D)} \| \mu \|_{\HS{1/2}{\partial D}} \\
    & \le \| \gamma_n \| \| \Jo \|_{\Hdiv} \| \mu \|_{\HS{1/2}{\partial D}} \\
    &  \le \| \Jo \|_{\Hdiv} \| \mu \|_{\HS{1/2}{\partial D}},
\end{split}
\end{equation}
where $\| \gamma_n \| \le 1$ is the norm of $\gamma_n$ in $\mathscr{L}(\Hdiv;H^{-1/2}(\partial D))$. By evaluating \eqref{eq_miq} for $\mu$ non-vanishing only on $\Gamma_i$, it follows that
\begin{equation}
\label{eqn_norm_ub}
    | \left \langle \mu_i, \Jo \cdot \nn \right\rangle_{\Gamma_i} | \le \| \mu_i \|_{\HS{1/2}{\Gamma_i}} \| \Jo \|_{\Hdiv},
\end{equation}
where $\mu_i$ is the restriction of $\mu$ to $\Gamma_i$. An upper bound to 
$\| \Jo \|_{\Hdiv}$ in \eqref{eqn_norm_ub} can be easily found by exploiting $\nabla \cdot \Jo =0$ in $D$ and the leftmost inequality in \eqref{eqn:kmequiv}, which gives:
\begin{equation}\label{eqn_norm_div}
        \frac{1}{\sigma_{bg}^M} \int_{D} \abs{\Jo}^2\,dx \le \int_{D}\frac{1}{\sigma_{bg}}\abs{\Jo}^2\,dx\leq k_l^{-1} \left\langle(\Lambda_D-\Lambda_{\varnothing})g,g\right\rangle.
\end{equation}
Therefore, the first claim follows by combining \eqref{eqn_norm_div} and \eqref{eqn_norm_ub}.

Similarly, the operator $\gamma$ that gives the trace of $\varphi_{\varnothing}$ on $\partial D$, i.e., $\gamma : u \in H^1\left(D\right) \to u|_{\partial D} \in H^{1/2}(\partial D)$, is linear and bounded in $\mathscr{L}(H^1\left(D\right);H^{1/2}(\partial D))$ \cite{girault2012}. Therefore, it follows that
\begin{equation}\label{eqn_norm_ub_2}
\begin{split}
    \| \gamma \left( \varphi_{\varnothing} - \overline{\varphi}_{\varnothing} \right) \|_{H^{1/2}(\Gamma_i)} & \le \| \gamma \left( \varphi_{\varnothing} - \overline{\varphi}_{\varnothing} \right) \|_{H^{1/2}(\partial D)}\\
    & \le \| \gamma \| \, \|  \varphi_{\varnothing} - \overline{\varphi}_{\varnothing} \|_{H^{1}(D)} \\
    & \le \| \gamma \| \, \sqrt{1+C_{PW}^2} \,  \|  \nabla \varphi_{\varnothing}  \|_{L^{2}(D)},
\end{split}
\end{equation}
where $\overline{\varphi}_{\varnothing}$ is the average value of $\varphi_{\varnothing}$ in $D$:
\begin{equation*}
    \overline{\varphi}_{\varnothing} = \frac{1}{|D|} \int_D \varphi_{\varnothing} \textrm{d}x,
\end{equation*}
$\| \gamma \|$ is the norm of the trace $\gamma$ in $\mathscr{L}(H^1\left(D\right);H^{1/2}(\partial D))$, the second inequality comes from the definition of the norm of $\gamma$, the third inequality comes from the Poincar\'{e}-Wirtinger inequality, and $C_{PW}$ is the related constant depending only on $D$. An upper bound to $\|  \nabla \varphi_{\varnothing}  \|_{L^{2}(D)}$ can be easily found by exploiting the leftmost inequality of \eqref{eqn:kmequiv} that gives
\begin{equation}\label{eqn_norm_l2}
        \sigma_{bg}^m \int_{D} \lvert\nabla \varphi_{\varnothing} \rvert^2 \,dx \le \int_{D} \sigma_{bg}\lvert\nabla \varphi_{\varnothing} \rvert^2 \,dx \leq  k_l^{-1} \left\langle(\Lambda_D-\Lambda_{\varnothing})g,g\right\rangle.
\end{equation}
Therefore, the second claim follows by combining \eqref{eqn_norm_l2} and \eqref{eqn_norm_ub_2}.
\end{proof}

\begin{rem}
    The result can be generalised with similar arguments to the case where $\partial C_i$ has multiple connected components.
\end{rem}

\begin{thm}
The interior of $D^*$ is the largest (maximal) set that satisfies condition \eqref{eqn_seq_un}.
\end{thm}
\begin{proof}
Let $D$ be strictly included in $D^*$; for example, $D$ has only one cavity that is not in contact with the boundary of $\Omega$. The boundary of $D$ is then given by $\partial D=\Gamma_1 \cup \Gamma_e$, where $\Gamma_1= \partial{C_1}$ and $\Gamma_e = \partial D \setminus \Gamma_1$ (see Figure \ref{fig:geo_h}).

The idea of the proof is the following: first, it is noted that for a kernel sequence for $D$, (i) the current density and the electric field vanish in $D$, therefore, (ii) the normal component of the electric current density vanishes on $\Gamma_i$, and (iii) the scalar potential approaches a constant value. As a consequence, the cavity $C$ is driven by vanishing boundary data, which, in turn, implies that the ohmic power in $C$ vanishes, and so is for the interior of $D^*$, i.e. $D  \cup  C$.

Let $\{f_m\}_{m\in\mathbb{N}}$ be a kernel sequence for $D$, and $\varphi_\varnothing^m$, $\Jo^m$ be the corresponding electric scalar potential and electrical current density, respectively. Furthermore, let $\overline{\varphi}_{\varnothing}$ be the average value of $\varphi_{\varnothing}$ in $D$.  A divergence theorem for the cavity $C$ gives
\begin{equation*}
\begin{split}
        \int_C \sigma_\varnothing \abs{\nabla\varphi_\varnothing^m}^2 dx = \int_C \Jo^m \cdot \nabla\varphi_\varnothing^m dx &=-\int_{\Gamma_1}\varphi_\varnothing^m \Jo^m \cdot \mathbf{\hat{n}}_1 ds \\
        &= - \left \langle \gamma \left( \varphi_{\varnothing}^m - \overline{\varphi}_{\varnothing}^m \right), \Jo^m \cdot \mathbf{\hat{n}}_1\right \rangle_{\Gamma_1},
\end{split}
\end{equation*}
where it has been exploited that $\int_{\Gamma_1} \overline{\varphi}_{\varnothing}^m \Jo^m \cdot \mathbf{\hat{n}}_1 ds=0$ since $\overline{\varphi}_{\varnothing}^m$ is a constant and $\nabla \cdot \Jo^m = 0$ in $\Omega$. Consequently, it turns out that
\begin{equation}
\begin{split}
    \int_C \sigma_\varnothing \abs{\nabla\varphi_\varnothing^m}^2 dx & \le \left| \left \langle \gamma \left( \varphi_{\varnothing}^m - \overline{\varphi}_{\varnothing}^m \right), \Jo^m \cdot \mathbf{\hat{n}}_1\right \rangle_{\Gamma_1} \right| \\
    & \le \| \gamma \left( \varphi_{\varnothing}^m - \overline{\varphi}_{\varnothing}^m \right) \|_{H^{1/2}(\Gamma_1)} \, \| \Jo^m \|_{\Hdiv} \\
    & \le \| \gamma \| \sqrt{1 + C_{PW}^2} \, \|  \nabla \varphi_{\varnothing}^m  \|_{L^{2}(D)} \, \| \Jo^m \|_{\Hdiv},
\end{split}
\end{equation}
and, by \Cref{lem_cav_c},
\begin{equation}
    \begin{split}
            \int_C \sigma_\varnothing \abs{\nabla\varphi_\varnothing^m}^2 dx & \le \frac{\alpha_D}{k_l} \left\langle(\Lambda_D-\Lambda_{\varnothing})f_m,f_m\right\rangle,
    \end{split}
\end{equation}
where
\begin{equation}
    \alpha_D = \| \gamma \| \sqrt{1 + C_{PW}^2} \, \sqrt{ \frac{ \sigma_{bg}^M}{ \sigma_{bg}^m}}.
\end{equation}

The thesis follows by noting that
\begin{equation}
\begin{split}
    \int_{D^*} \sigo(x) \lvert\nabla \varphi_{\varnothing}^m(x) \rvert^2 \, dx & =  \int_{D} \sigo(x) \lvert\nabla \varphi_{\varnothing}^m(x) \rvert^2 \, dx +  \int_{C} \sigo(x) \lvert\nabla \varphi_{\varnothing}^m(x) \rvert^2 \, dx \\
    & \le \frac{1+\alpha_D}{k_l} \left\langle(\Lambda_D-\Lambda_{\varnothing})f_m,f_m\right\rangle.
\end{split}
\end{equation}
\end{proof}

\noindent Consequently, the retrieval of $D$ using the proposed method consists of retrieving the interior of $D^*$, the largest set that satisfies condition \eqref{eqn_seq_un}.

\begin{rem}[Localized Potentials and Kernel Sequences]
    The proof of \Cref{thm_un} shows that \emph{any} sequence of localised potentials providing a vanishing ohmic power in $D$ is a kernel sequence for the same set $D$.
\end{rem}

\begin{rem}\label{rem_large}
    Testing \eqref{eqn_seq_un} is equivalent to verifying if the ratio
    \begin{equation*}
         \frac{\int_S \sigo \abs{\nabla \varphi_{\varnothing}^n}^2\,dx}{\langle (\Lambda_D-\Lambda_{\varnothing})f_n,f_n\rangle}
    \end{equation*}
    stays bounded for large $n$.
\end{rem}

\section{Eigenfunctions and Kernel Sequences}\label{sec_eig_fun}
In this Section is shown that the sequence  $\{g_n\}_{n\in\mathbb{N}}$ of the eigenfunctions of $\Lambda_D-\Lambda_{\varnothing}$ \lq\lq represents\rq\rq \ a large class of kernel sequences, i.e. testing the condition of \Cref{thm_un} on the sequence $\{g_n\}_{n\in\mathbb{N}}$ corresponds to testing the condition on a wide class of kernel sequences.

The sequence $\{g_n\}_{n\in\mathbb{N}}$ is
a natural candidate as a testing sequence for the condition~\eqref{eqn_seq_un}. Indeed, Property~\ref{prop:compact} states that the NtD operator is compact, which in turn implies that $\Lambda_D-\Lambda_{\varnothing}$ is compact as well. By the Spectral Theorem for compact operators~\cite{book:Rud91}, there exists a sequence $\{g_n\}_{n\in\mathbb{N}}\subset L^2_{\diamond}(\partial\Omega)$ of orthonormal eigenfunctions and a corresponding sequence $\{\lambda_n\}_{n\in\mathbb{N}}$ of positive eigenvalues, with $\lambda_n$ monotonically decreasing to zero.
Assuming the following normalisation for the eigenfunctions,
\begin{equation}
\label{eq_norm}
    \lvert g_n \rvert^2=\int_{\partial\Omega} g_n^2(x)\,dx=1,
\end{equation}
it follows that
\begin{equation}\label{eqn:norm}
    \langle\left(\Lambda_D-\Lambda_{\varnothing}\right) g_n,g_n\rangle
    = \left\langle \lambda_n g_n,g_n \right\rangle
    = \lambda_n.
\end{equation}
Therefore, the sequence of normalised eigenfunctions of the operator $\Lambda_D-\Lambda_{\varnothing}$ provides a kernel sequence for testing condition~\eqref{eqn_seq_un}. 

The eigenfunctions $\{g_k\}_{k\in\mathbb{N}}$ form a complete orthonormal basis of $L^2_{\diamond}(\partial\Omega)$. Hence, given a kernel sequence $\{f_n\}_{n\in\mathbb{N}}$, each element can be expanded with respect to this basis as
\begin{equation}\label{eqn_exp_fn}
    f_n = \sum_{k=1}^{+\infty} \langle f_n, g_k \rangle g_k .
\end{equation}
Since the eigenfunctions are orthonormal, the normalisation condition $\lVert f_n \rVert = 1$ implies
\begin{equation}\label{eqn_exp_norm}
    \sum_{k=1}^{+\infty} \langle f_n, g_k \rangle^2 = 1 .
\end{equation}
Moreover, taking into account \eqref{eqn_exp_fn}
\begin{equation}\label{eqn_sc_1}
    \langle (\Lambda_D - \Lambda_{\varnothing}) f_n, f_n \rangle
    = \sum_{k=1}^{+\infty} \lambda_k \langle f_n, g_k \rangle^2, 
\end{equation}
for a kernel sequence, it follows that
\begin{equation}\label{eqn_exp_sc}
    \lim_{n\to+\infty} \sum_{k=1}^{+\infty} \lambda_k \langle f_n, g_k \rangle^2=0.
\end{equation}
Equations \eqref{eqn_exp_norm} and \eqref{eqn_exp_sc} provide insight into the spectral structure of the kernel sequences. Equation \eqref{eqn_exp_norm} guarantees that, for each fixed $n$, the sequence of coefficients $\{\langle f_n, g_k \rangle\}_{k \in \mathbb{N}}$ belongs to
$\ell^2(\mathbb{R})$, implying that the coefficients decay rapidly to zero.
On the other hand, the admissibility condition \eqref{eqn_exp_sc} shows that as $\sum_{k=1}^{+\infty} \lambda_k \langle f_n, g_k \rangle^2$ approaches zero, i.e., $n$ increases, the most relevant coefficients $\langle f_n, g_k\rangle^2$ need to be
progressively shifted toward higher-order eigenfunctions, where the weights $\lambda_k$ of the sum are smaller.
From a formal point of view, the following proposition holds (see \Cref{proof_1}).
\begin{prop}\label{prop_approx}
Let $\{f_n\}_{n\in\mathbb{N}}\subset L_{\diamond}^2(\partial\Omega)$ be a kernel sequence. For a given element $f_n$ of the sequence and for any arbitrary $\varepsilon>0$, there exist scalars $\alpha_n,\beta_n\in\mathbb{N}$, depending on $n$ and $\varepsilon$, such that
\begin{equation*}
   \abs{ f_n - \sum_{k=\alpha_n}^{\beta_n} \langle f_n, g_k \rangle g_k }^2 \leq \varepsilon
\end{equation*}
\end{prop}
\Cref{prop_approx} has direct relevance from an application point of view. Indeed, it shows that each element of a kernel sequence can be accurately approximated by a linear combination of eigenfunctions with indices in the interval
$\{\alpha_n, \ldots,\beta_n \}$. This interval depends on $f_n$, the specific element of the sequence, and shifts to higher indices as $n \to +\infty$.

Moreover, the approximation can be constructed in such a way that the key ratio appearing in \eqref{eqn_seq_un} is only marginally affected. More precisely (see \Cref{proof_1}), the following result holds.

\begin{prop}\label{prop_approx_ratio}
Let $\{f_n\}_{n\in\mathbb{N}} \subset L_{\diamond}^2(\partial\Omega)$ be a kernel sequence such that
\begin{equation*}
    \lim_{n\to +\infty} \frac{\int_S\sigo \lvert\nabla \varphi_{\varnothing}^n\rvert^2\,dx}{\langle\left(\Lambda_D-\Lambda_{\varnothing}\right) f_n,f_n\rangle} = + \infty.
\end{equation*}
Then, there exists an approximating sequence $\{\widetilde{f}_n \}_{n \in \mathbb{N}}$ of $\{f_n \}_{n \in \mathbb{N}}$, such that $\{\widetilde{f}_n \}_{n\in\mathbb{N}}$ is a kernel sequence, and
\begin{equation*}
    \lim_{n\to +\infty} \frac{\int_S\sigo \lvert\nabla \widetilde{\varphi}_{\varnothing}^n\rvert^2\,dx}{\langle\left(\Lambda_D-\Lambda_{\varnothing}\right) \widetilde{f}_n,\widetilde{f}_n\rangle} = + \infty,
\end{equation*}
where $\varphi_{\varnothing}^n$ and $\widetilde{\varphi}_{\varnothing}^n$ denote the scalar potentials corresponding to the boundary data $f_n$ and $\widetilde{f}_n$, respectively.
\end{prop}

At this stage, it is convenient to introduce the \lq\lq length\rq\rq \ of an element of $L^2(\partial\Omega)$ with respect to the basis of the eigenfunctions $\{ g_n \}_{n \in \mathbb{N} }$.

\begin{defn}[Length of $f$]
    The length of $f \in L^2(\partial\Omega)$, denoted as $\ell(f)$, is the number of non-zero coefficients of $f$ along the eigenfunction basis $\{ g_n \}_{n \in \mathbb{N} }$.
\end{defn}

By combining \Cref{prop_approx} and \Cref{prop_approx_ratio}, it turns out that \emph{any} kernel sequence can be approximated by a sequence in which any element is a linear combination of a finite number of eigenfunctions $g_k$. This motivates the analysis of the ratio that appears in \eqref{eqn_seq_un} to functions $f$ made by the linear combination of a finite number of eigenfunctions, i.e., for functions of finite length $\left( \ell(f) < + \infty \right)$. 

\begin{notation}
    In the following, $u_{\varnothing}^n$ denotes the scalar potential corresponding to the electrical conductivity $\sigo$ and the boundary data $g_n$.
\end{notation}

The following lemma holds.
\begin{lem}[see \Cref{proof_2}]
\label{prop_linear}
Let $f_n \in L^2(\partial\Omega)$ with $\norm{f_n}=1$ be the linear combination of a finite number of eigenfunctions
\begin{equation*}
    f_n=\sum_{k\in \mathcal{I}_n} c_{n,k} g_k,
\end{equation*}
where $\mathcal{I}_n$ are the indices of non-zero coefficients depending on $f_n$, with $|\mathcal{I}_n|<+\infty$. 
Assuming that the set $S$ fulfils \eqref{eqn_seq_un} when tested on the eigenfunctions $g_n$, where 
$n \in \mathcal{I}_n$, i.e.
\begin{equation*}
    \frac{\int_S \sigo\lvert\nabla u_{\varnothing}^n\rvert^2\,dx}{\lambda_n}\leq K, \quad  n\in \mathcal{I}_n,
\end{equation*}
then, it follows that
\begin{equation*}
    \frac{\int_S \sigo\abs{\nabla \varphi_{\varnothing}^n}^2\,dx}{\langle (\Lambda_D-\Lambda_{\varnothing}) f_n,f_n \rangle} \leq K|\mathcal{I}_n|,
\end{equation*}  
where $\varphi_{\varnothing}^n$ denotes the scalar potentials corresponding to the boundary data $f_n$.
\end{lem}

\Cref{prop_linear} leads to the following relevant proposition.
\begin{prop}
\label{pr_key}
    Let $\{ f_n \}_{n \in \mathbb{N} } \subset L^2(\partial\Omega)$ with $\lVert f_n \rVert=1$ be a sequence in which each element is made by the linear combination of a bounded number of eigenfunctions, i.e., $\ell(f_n) \le M$. If the set $S$ fulfils \eqref{eqn_seq_un} on the sequence of the eigenfunctions, that is,
\begin{equation*}
    \frac{\int_S \sigo \lvert\nabla u_{\varnothing}^n\rvert^2\,dx}{\lambda_n}\leq K, \quad  \forall n \in \mathbb{N},
\end{equation*}
then, it follows that
\begin{equation*}
    \frac{\int_S \sigo\abs{\nabla \varphi_{\varnothing}^n}^2\,dx}{\langle (\Lambda_D-\Lambda_{\varnothing}) f_n,f_n \rangle} \leq KM.
\end{equation*}  
\end{prop}

\Cref{pr_key} means that the testing condition \eqref{eqn_seq_un} applied to the eigenfunctions $g_n$ of the operator $\Lambda_D - \Lambda_{\varnothing}$ is equivalent to the testing \eqref{eqn_seq_un} applied to \emph{any} kernel sequence where each element $f_n$ is the linear combination of at most $M$ eigenfunctions. This gives the sequence of eigenvectors $\{ g_n \}_{n \in \mathbb{N} }$ a central role in the proposed imaging method. 

Moreover, it is worth noting that
\begin{equation}
\label{eqn_lim_k}
    \lim_{n \to +\infty}
    \frac{\int_{S}\sigo\abs{\nabla u_{\varnothing}^n}^2\,dx}{\lambda_n}
    = \tilde{K} < +\infty,
\end{equation}
implies that    
\begin{equation}
\frac{\int_{S}\sigo\abs{\nabla u_{\varnothing}^n}^2\,dx}{\lambda_n}
    \leq K \quad \forall\, n \in \mathbb{N},
\end{equation}
for some proper finite constant $K$. Therefore, the sequence of eigenfunctions satisfies the testing condition \eqref{eqn_seq_un} if the latter is satisfied in the limit for $n \to +\infty$, i.e. for large $n$ in any practical setting.

In summary, if a set $S$ satisfies the testing condition \eqref{eqn_seq_un} on $g_n$ for a large enough index $n$, then the testing condition \eqref{eqn_seq_un} is automatically satisfied for all kernel sequences where each element is the linear combination of a bounded number of eigenfunctions of the operator $\Lambda_D-\Lambda_{\varnothing}$. The kernel sequences $\{ f_n \}_{n \in \mathbb{N} }$ that are not taken into account by testing \eqref{eqn_seq_un} on eigenfunctions $g_n$, for a sufficiently large index $n$, are only those where the sequence of lengths $\{ \ell(f_n) \}_{n \in \mathbb{N} }$ is not bounded.

\section{Towards the imaging method}\label{sec_ti}
The results of \Cref{sec:pbg_eig} show that an imaging method based on evaluating the power density for a known reference (background) and suitable sequences of boundary data can be established.
The outer support of the anomaly can be characterised as the \emph{largest} set that satisfies the condition \eqref{eqn_seq_un} for every kernel sequence.

However, such a characterisation poses challenges for a practical implementation, as it would require testing the condition \eqref{eqn_seq_un} for infinitely many kernel sequences. A key issue is reducing the number of sequences that must be considered. From this perspective, \Cref{sec_eig_fun} shows that the eigenfunctions of the operator $\Lambda_D-\Lambda_{\varnothing}$ play a central role. Indeed, this specific sequence encodes testing information for all kernel sequences, where each element (of the sequence) is the linear combination of a bounded number of eigenfunctions, as shown in \Cref{pr_key}.

The key role of the eigenfunctions of the operator $\Lambda_D-\Lambda_{\varnothing}$ becomes crystal clear when proper minimal assumptions are introduced, in view of the practical implementation of the method. Specifically, it can be shown that the eigenfunction sequence carries information about sequences of localised potentials capable of localising the ohmic power in \emph{any} subset external to $D^*$. This makes it unnecessary to test the condition \eqref{eqn_seq_un} for any sequence, except the eigenfunction sequence.

Within the framework of Electrical Resistance Tomography, it is known that the eigenvalues $\lambda_k$ of the compact operator $\Lambda_D-\Lambda_{\varnothing}$ decay exponentially (see, for example, \cite{Gi90,fm_3}). Moreover, a similar exponential behaviour is observed for the power density $\sigo|\nabla u_{\varnothing}^k|^2$ corresponding to the reference configuration driven by the eigenfunction $g_k$ of $\Lambda_D-\Lambda_{\varnothing}$ \cite{fm_2,fm_3,fm_5} (see also \Cref{app_C})
\begin{align}
    \ln(\lambda_k) & \sim \ln(c_{\lambda}) +  k \ln q_{\lambda}, \label{asp_lam} \\
    \ln \left(\sigo(x)\abs{\nabla u^k_{\varnothing}(x)}^2 \right) &\sim \ln (c_x)+ k \ln  p_x,  \label{asp_pbg}
\end{align}

Under assumptions \eqref{asp_lam} and \eqref{asp_pbg}, analytical sequences of localised potentials, as proposed by Gebauer in \cite{Ge08} and briefly summarised in \Cref{app_C}, are essential in understanding the key role played by the sequence of eigenfunctions $\{ g_n \}_{n \in \mathbb{N}}$. These sequences of analytically localised potentials, hereafter denoted as ALP, are kernel sequences in \emph{analytic} form that can determine whether a set $B$ is external to $D^*$. Specifically, for any arbitrary $B \subset \Omega \backslash D^*$, there exists an ALP sequence $\{ f^A_n \}_{n \in \mathbb{N}}$ with $\lVert f^A_n \rVert = 1$, and
     \begin{align*}
        \lim_{n \to + \infty} 
        \frac{\int_{B}\sigo\abs{\nabla \varphi_{\varnothing}^{n}(x)}^2\,dx}
        {\int_D \sigo\abs{\nabla \varphi_{\varnothing}^{n}(x)}^2\,dx}
        & = +\infty \\
        \lim_{n \to +\infty} \langle (\Lambda_D-\Lambda_{\varnothing}) f^A_n,f^A_n \rangle & = 0,
    \end{align*}
where $\varphi_{\varnothing}^{n}$ is the scalar potential corresponding to $f^A_n$. The key property of ALP sequences is that if a set $B$ satisfies the testing condition 
\eqref{eqn_seq_un} for the eigenfunction sequence $\{g_n \}_{n \in \mathbb{N} }$, then it satisfies \eqref{eqn_seq_un} for an arbitrary ALP sequence and, therefore, $B \subseteq D^*$. Indeed, the following proposition holds.

\begin{prop}[see Appendix C]
\label{prop_loc_eig}
    Let $B\Subset\Omega$ be such that $B\cap D^*=\varnothing$, and let $\{f_{n}^A\}_{n \in \mathbb{N} }$ be an ALP sequence for the set $B$.
    Then, under assumptions \eqref{asp_lam} and \eqref{asp_pbg}, the condition
    \begin{equation*}
        \frac{\int_B \sigo \lvert\nabla u^n_{\varnothing}\rvert^2\,dx}{\lambda_n}\leq K 
        \quad \forall \, n\in\mathbb{N}
    \end{equation*}
    implies
    \begin{equation*}
        \lim_{n \to +\infty}
        \frac{\int_B \sigo\abs{\nabla \varphi^{n}_{\varnothing}}^2\,dx}
        {\langle (\Lambda_D-\Lambda_{\varnothing}) f_{n}^A,f_{n}^A \rangle}
        \leq \tilde{K}.
    \end{equation*}
\end{prop}

Thanks to \Cref{prop_loc_eig}, it is straightforward to prove the following main result using the same methods as in \Cref{thm_un}.

\begin{thm}[Lower bound, revised]
\label{thm_un_gn}
    If a set $S \Subset \Omega$ satisfies 
    \begin{equation}\label{eqn_seq_gn}
	\frac{\int_S\sigo \lvert\nabla u^n_{\varnothing}\rvert^2\,dx}{\lambda_n}=
    \frac{\int_S\sigo \lvert\nabla u^n_{\varnothing}\rvert^2\,dx}{\langle\left(\Lambda_D-\Lambda_{\varnothing}\right) g_n,g_n\rangle}\leq K, \ \forall n \in \mathbb{N},
    \end{equation}
    then $S \subseteq D^*$.
\end{thm}
This result is very relevant from an application perspective. Indeed, it proves that the test of condition \eqref{eqn_seq_un} solely on the sequence of eigenfunctions yields a lower bound to the anomalous region $D^*$. In other terms, there is no need to test condition \eqref{eqn_seq_un} on \emph{all} kernel sequences.

Furthermore, under assumptions \eqref{asp_lam} and \eqref{asp_pbg}, the testing condition \eqref{eqn_seq_gn} can be interpreted in terms of the decay rates of the eigenvalue $\lambda_n$ and of the power density $p_{\varnothing}^n(x)=\sigma_{\varnothing}(x)\abs{\nabla u^n_{\varnothing}(x)}^2$ due to the eigenfunction $g_n$.  More precisely, the boundedness condition
\begin{equation*}
    \frac{\int_S c_xp_x^{n}\,dx}{c_{\lambda}q_{\lambda}^{n}}
    \sim\frac{\int_S \sigo(x)\lvert\nabla u^n_{\varnothing}(x)\rvert^2\,dx}{\lambda_n} \leq K
\end{equation*}
can only hold if and only if, at any point of $S$, the exponential rate of the numerator is not greater than that of the denominator, i.e.,
\begin{equation*}
    p_x \leq q_{\lambda}, \ \forall x \in S.
\end{equation*}
Therefore, $D^*$, the maximal set that satisfies the revised testing condition \eqref{thm_un_gn} is characterised as the set of points $x \in \Omega$ where $p_x \le q_{\lambda}$, that is,
\begin{equation}\label{eqn_res1}
    D^* = \{ x \in \Omega \mid p_x \le q_{\lambda} \}.
\end{equation}
Moreover, for any $x$ external to $D^*$, it follows that $p_x > q_{\lambda}$ and therefore,
\begin{equation}\label{eqn_res2}
    \lim_{n \to +\infty} \frac{p^n_{\varnothing}(x)}{\lambda_n} = + \infty.
\end{equation}

\section{Proposed method}\label{sec:main}
The analysis carried out in the previous section shows that, in a practical setting, it is sufficient to investigate the asymptotic behaviour of the power density $p_{\varnothing}^n$, for the reference configuration when driven by the eigenfunctions $g_n$. Indeed, as highlighted in \eqref{eqn_res1} and \eqref{eqn_res2}, $p_{\varnothing}^n$ exhibits a qualitatively different behaviour inside and outside $D^*$, which is the maximal set satisfying \Cref{thm_un_gn}. Therefore, it is natural to define the reconstructed region $\widetilde{D}$ as the set $D^*$, i.e., the method aims to reconstruct the outer support of the anomaly.

Moreover, from the characterisation in \eqref{eqn_res1} and \eqref{eqn_res2}, it follows that $\widetilde{D}$ is given by the points for which the ratio $p_{\varnothing}^n/\lambda_n$ stays bounded as $n\to+\infty$. When a given eigenfunction $g_n$ of $\Lambda_D-\Lambda_{bg}$, corresponding to a large but finite index $n$, is considered, a natural translation of the above characterisation is  
\begin{equation}\label{eqn:rec}
    \widetilde{D}_{\alpha}^n=\left\{ x\in \Omega \mid \frac{p_{\varnothing}^n(x)}{\lambda_n} < \alpha \right\},
\end{equation}
where $\alpha>0$ is an appropriately chosen threshold. 

In principle, determining this threshold is non-trivial. Indeed, the power density $p_{\varnothing}^n(x)$ generated by $g_n$ is a continuous function when the background conductivity $\sigma_{\varnothing}(x)$ is continuous, so no sharp gap is expected between the values of $p_{\varnothing}^n(x)/\lambda_n$ when evaluated inside and outside the anomaly. 
\begin{figure}
    \centering
    \includegraphics[width=0.35\linewidth]{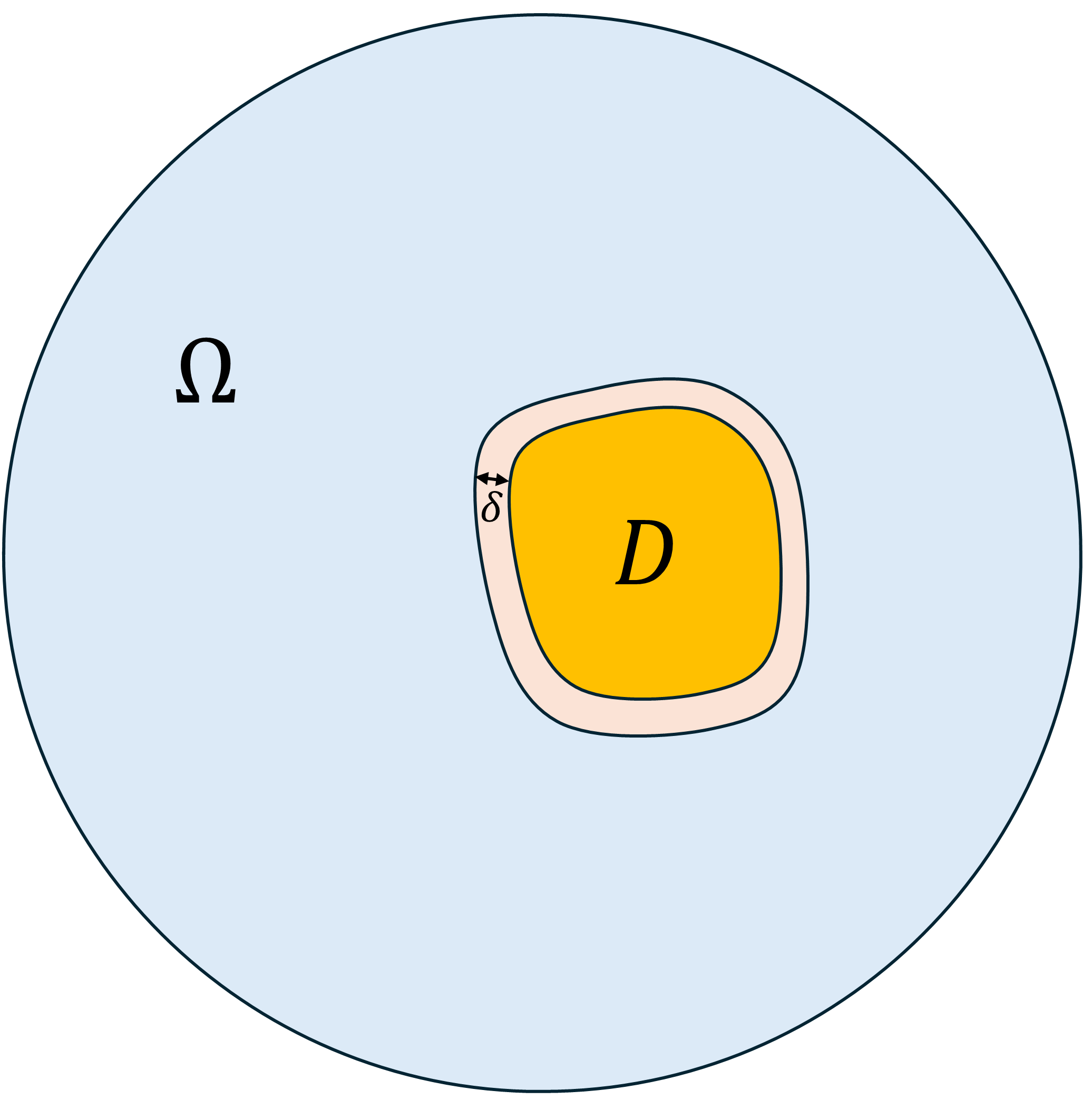}
    \caption{Anomaly $D$ embedded in a uniform background. A tubular neighbourhood of $\partial D$ is considered.}
    \label{fig_delta}
\end{figure}

To address this issue, consider the configuration illustrated in Figure~\ref{fig_delta}. Let $\delta>0$ be an arbitrary (small) positive threshold, let $x$ be a point in $D^*$ and let $y$ be a point at a distance greater than $\delta$ from $D^*$. The point $y$ lies outside the tubular neighbourhood highlighted in Figure~\ref{fig_delta}. Combining \eqref{eqn_res1} and \eqref{eqn_res2}, it follows that the decay rates at $x$ and $y$ satisfy $p_y > p_x$. In other words, any point well separated from the anomaly exhibits a strictly slower exponential decay rate of $p_{\varnothing}^n$ than any point inside the anomaly.

As a consequence, for a sufficiently large $n$, a clear separation emerges between the power density values within the anomaly and those obtained at points that are well separated from it. This asymptotic separation justifies the introduction of a threshold to distinguish the two regions. In this sense, the threshold $\alpha$ should be chosen between the values of the ratio $p_{\varnothing}^n(x)/\lambda_n$ evaluated inside the anomaly and those evaluated at points sufficiently separated from it. Since the width $\delta$ of the tubular neighbourhood can be chosen arbitrarily, it follows that the outer support of the anomaly can be reconstructed with arbitrary precision by introducing a proper threshold.

To complete the reconstruction rule \eqref{eqn:rec}, a proper strategy is needed to select the right threshold $\alpha^*$. For this purpose, it is worth noting that the power density dissipated in $D$ for the background configuration is related to the power difference $\langle (\Lambda_D-\Lambda_{\varnothing})g_n,g_n\rangle$. Indeed, from Equation \eqref{eqn:kmequiv} (see the proof of Lemma \ref{lem:mainl}), it follows that
    \begin{equation}\label{eqn:pro0}
        \frac{\sigma_a^m}{\sigma_{bg}^m-\sigma_{a}^M}\langle (\Lambda_D-\Lambda_{\varnothing}) \, g,g\rangle\leq \int_D p_{\varnothing}(x)\,dx\leq \frac{\sigma_{bg}^M}{\sigma_{bg}^m-\sigma_{a}^M}\langle (\Lambda_D-\Lambda_{\varnothing}) \, g,g\rangle,
    \end{equation}
where $p_{\varnothing}=\sigo \abs{\nabla \varphi_\varnothing}^2$.
Equation \eqref{eqn:pro0} is relevant since it allows one to compute the threshold $\alpha^*$ in \eqref{eqn:rec}. In fact, setting $g=g_n$ and accounting \eqref{eqn:norm}, it turns out that
  \begin{equation*}
        \frac{\sigma_a^m}{\sigma_{bg}^m-\sigma_{a}^M}\leq \frac{\int_D p_{\varnothing}^n(x)\,dx}{\lambda_n}\leq \frac{\sigma_{bg}^M}{\sigma_{bg}^m-\sigma_{a}^M},
    \end{equation*}
where the constants $\sigma_a^m$, $\sigma_a^M$, $\sigma_{bg}^m$, and $\sigma_{bg}^M$ are defined in \eqref{eqn_bound_bg} and \eqref{eqn_bound_a}.
Hence, the threshold ${\alpha^*}$ is found as the solution of
\begin{equation}
\label{eqn:rec_free}
    \frac{\int_{D^n_{\alpha}} p_{\varnothing}^n(x) \,dx}{\lambda_n}=\varepsilon^*,
\end{equation}
where 
\begin{equation}
\label{eqn:thr_free}
    \frac{\sigma_{a}^m}{\Delta\sigma_{G}}\le \varepsilon^* \le \frac{\sigma_{bg}^M}{\Delta\sigma_{G}},
\end{equation}
and $\Delta\sigma_{G}=\sigma_{bg}^m-\sigma_{a}^M$.

\begin{rem}
    It is worth noting that the left-hand side of \eqref{eqn:rec_free} is monotonic w.r.t. $\alpha$. Therefore, equation \eqref{eqn:rec_free}
 admits a unique solution in $\alpha$.
\end{rem}

\begin{rem}
    From an equivalent perspective, the approximation of the anomaly boundary is obtained by selecting the level curve of $p_{\varnothing}^n$ at level $\alpha^* \lambda_n$. for a prescribed $n$. An example is provided in Figure \ref{fig:ordkm}, which shows the level curves of $p_{\varnothing}^n$ for different values of $n$. As $n$ increases, the shape of these curves progressively approaches $\partial D$.
    \begin{figure}
        \centering
        \subfloat[][]
        {\includegraphics[width=.33\textwidth]{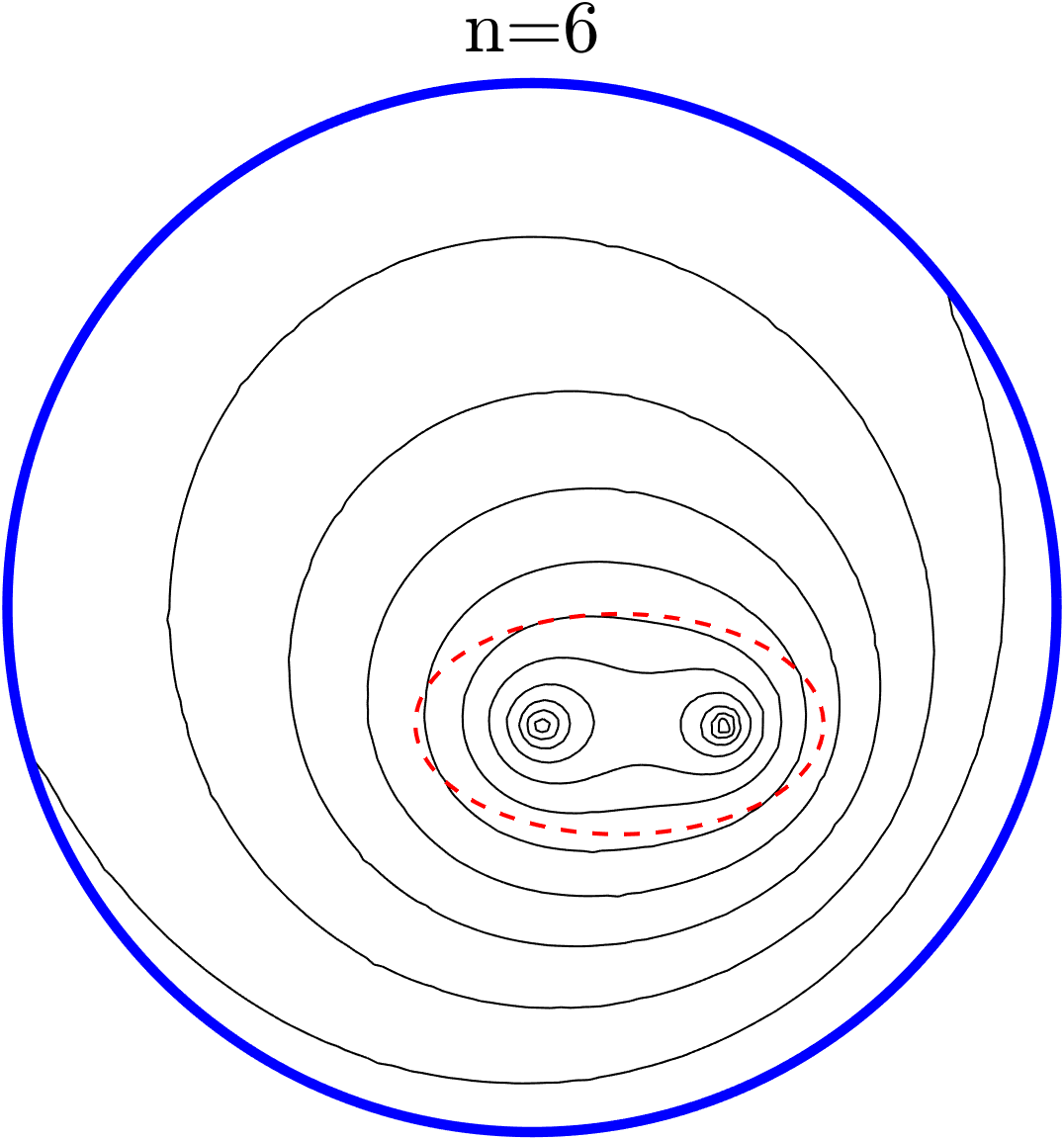}} \quad
        \subfloat[][]
        {\includegraphics[width=.33\textwidth]{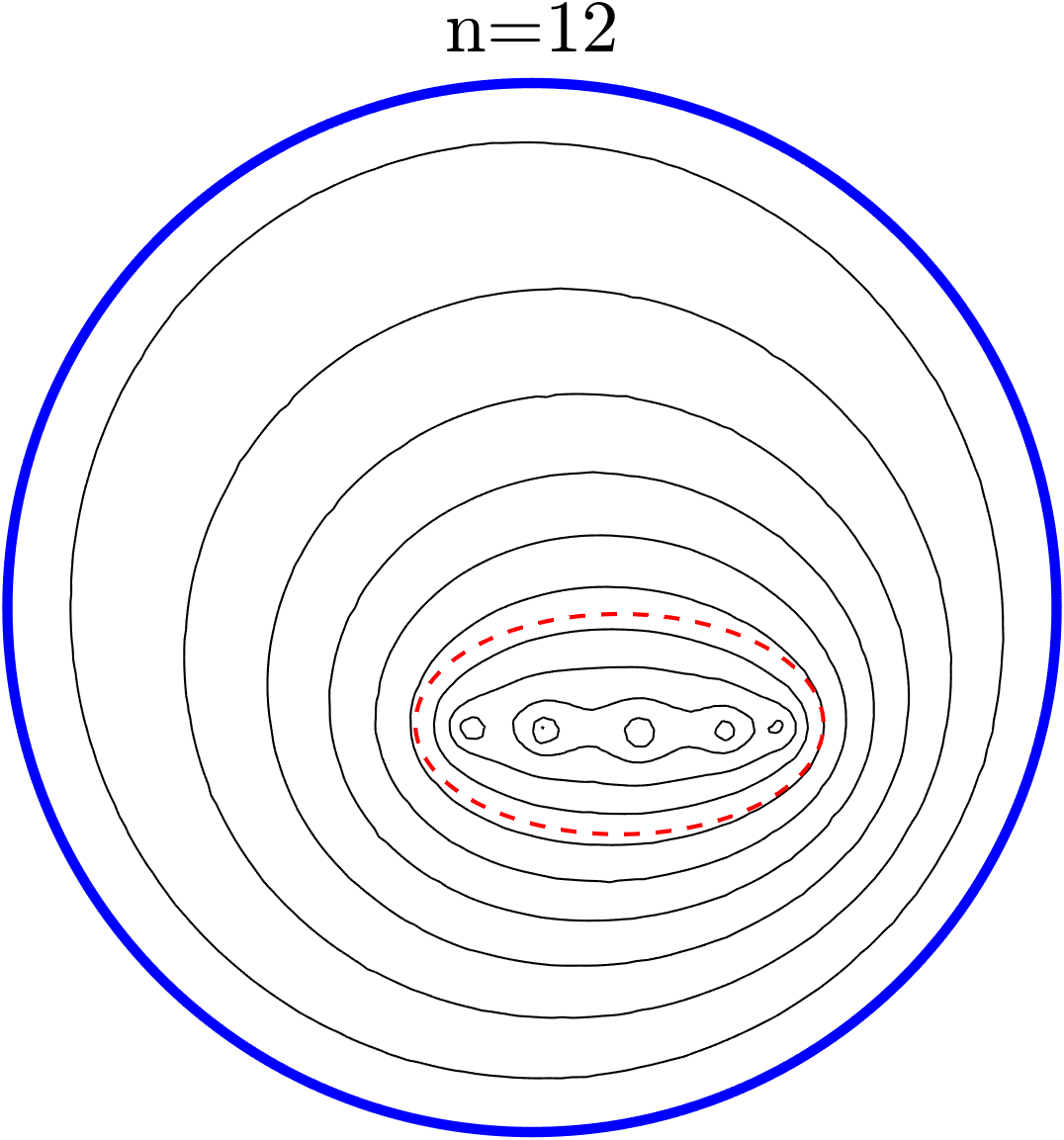}} \\
        \subfloat[][]
        {\includegraphics[width=.33\textwidth]{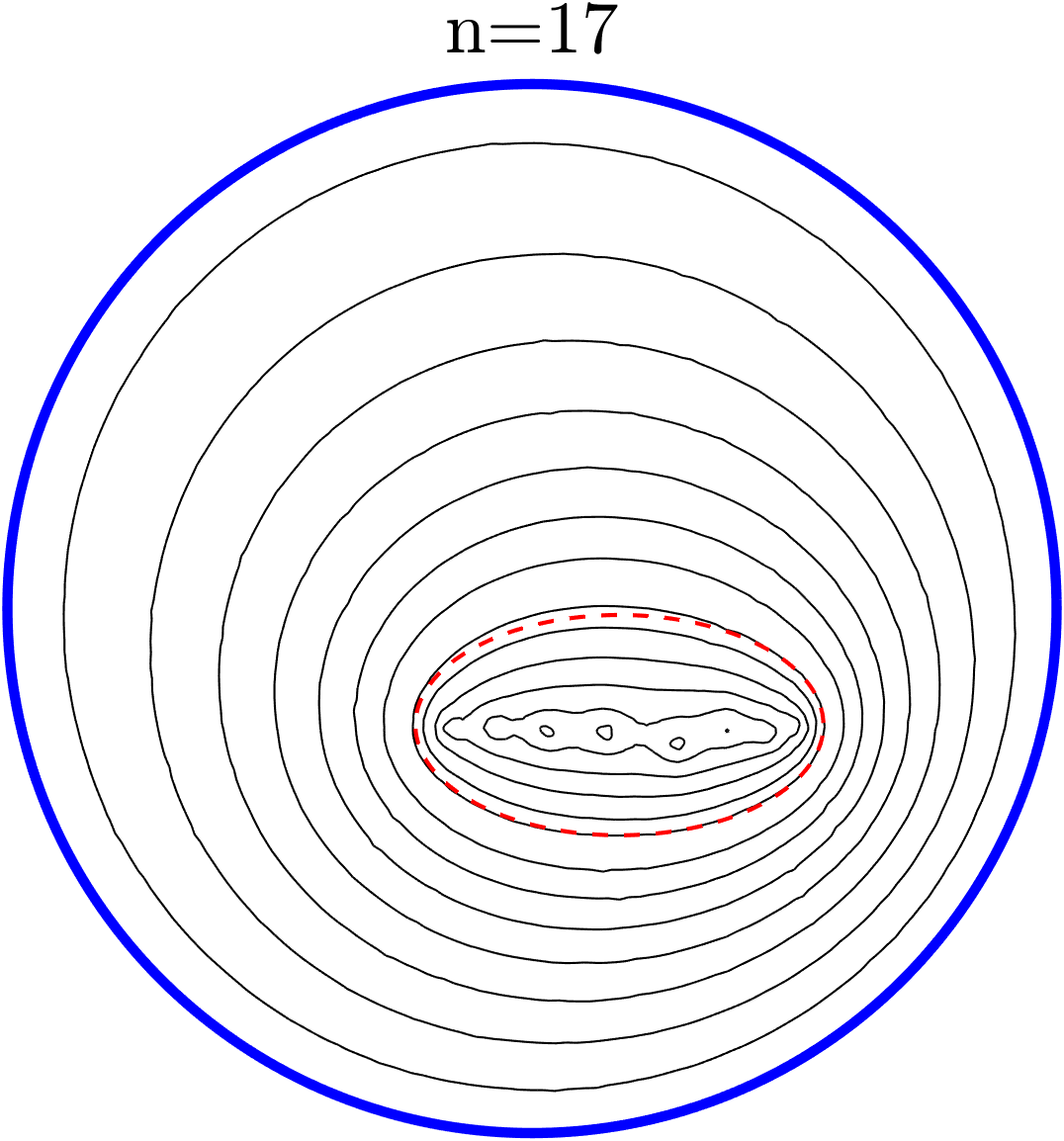}} \quad
        \subfloat[][]
        {\includegraphics[width=.33\textwidth]{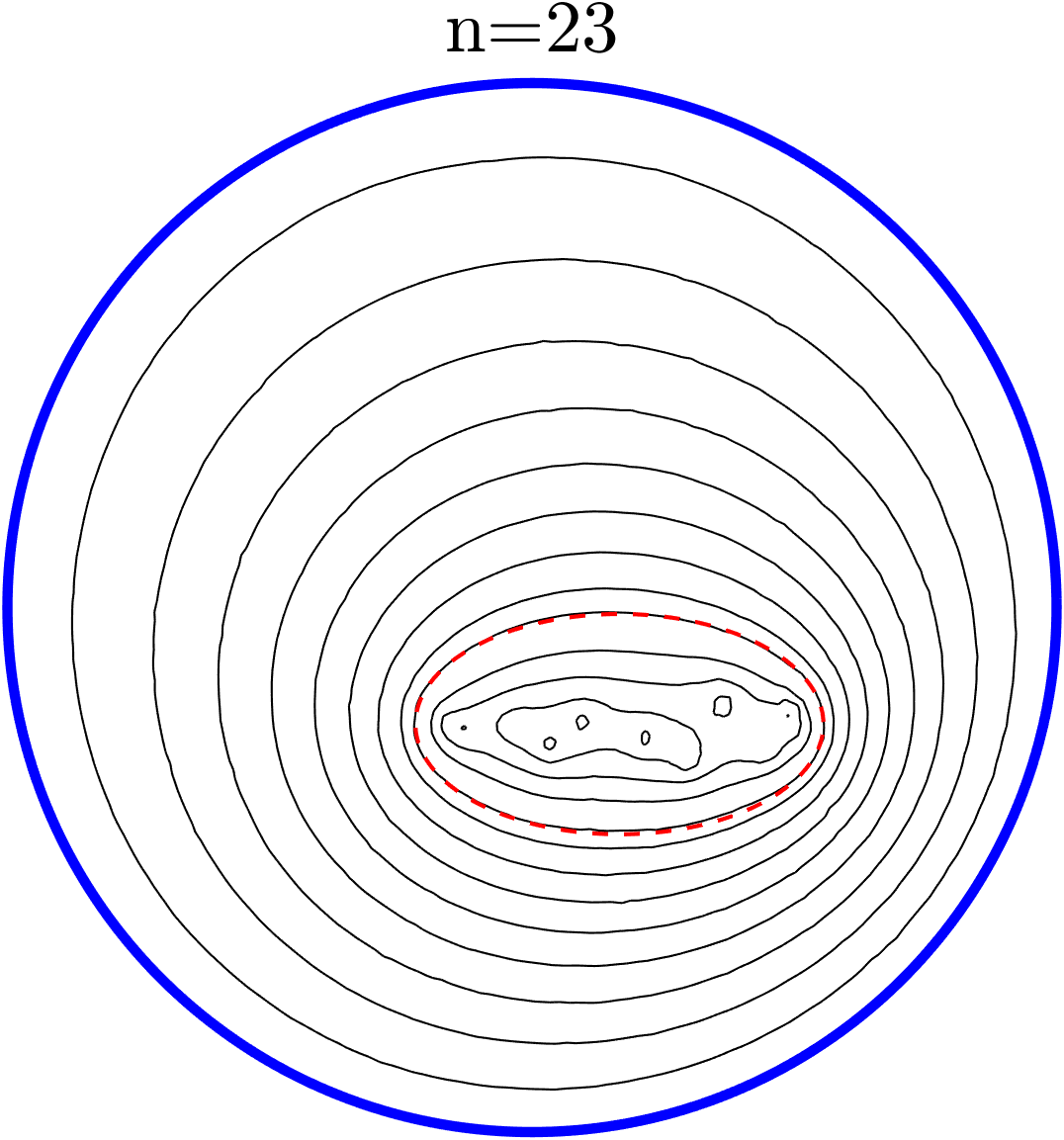}} 
        \caption{Shape of the level curves varying the order of the eigenfunctions.}
        \label{fig:ordkm}
    \end{figure}
\end{rem}

In summary, the proposed imaging method consists of:
\begin{itemize}
\item compute experimentally or numerically $\Lambda_{\varnothing}$;
\item measure the operator $\Lambda_D$;
\item find the eigenvalues $\lambda_n$ and the eigenfunctions $g_n$ of $\Lambda_D-\Lambda_{\varnothing}$, with the normalisation condition of \eqref{eq_norm};
\item solve problem~\eqref{eqn:dirprob} when the conductivity is $\sigma_{\varnothing}$ and the applied current flux at the boundary is $g_n$, for a sufficiently small eigenvalue $\lambda_n$;
\item the reconstructed anomaly is given by $\widetilde{D}^n_{\alpha^*}$, where $\alpha^*$ is the unique solution of \eqref{eqn:rec_free}.

\end{itemize}

\begin{rem}
    The methodology introduced in this section decouples the retrieval of the anomalous region into two different problems: (i) the determination of the shape and (ii) the determination of a proper threshold on the power density. Information about the shape of the anomaly is contained in the level curves of $p_{\varnothing}^n(x)$ (see Figure \ref{fig:ordkm}), while the threshold to be applied to $p_{\varnothing}^n(x)$ depends on the eigenvalues and the bounds to the material properties (see \eqref{eqn:rec_free}).
\end{rem}

\section{Treatment of the noise}\label{sec:noise}
In any practical application, only a noisy version of the NtD map is available. Specifically, noisy measurements are supposed to be modelled as 
\begin{equation}\label{eqn:noise}
\Delta \widetilde{\Lambda}=\Delta\Lambda+N,
\end{equation}
where $\Delta\widetilde{\Lambda}$ is the noisy version of the noise free difference operator $\Delta\Lambda=\Lambda_D-\Lambda_{\varnothing}$, and $N$ is the operator that represents the noise.

Hereafter, it is assumed that $N$ is a bounded operator, that is, $\norm{N} < + \infty$, where:
\begin{equation*}
    \norm{N}=\sup_{f\in L^2_{\diamond}(\partial\Omega)}\frac{\norm{Nf}_2}{\norm{f}_2}
\end{equation*}
Moreover, the operator $N$ can always be considered self-adjoint, taking the symmetric part of $\Delta\widetilde{\Lambda}$. Under these assumptions, the spectrum of $\Delta\widetilde{\Lambda}$ can be related to the spectrum of $\Delta \Lambda$ through the following proposition, which is an application of standard arguments in perturbation theory \cite{Ka95}.
\begin{prop}\label{prop:pert}
    Let $\Delta\widetilde{\Lambda}$ as defined in~\eqref{eqn:noise}, where $N$ is a bounded and self-adjoint operator, and $\delta=\norm{N}$. Let $\lambda_k$, $\widetilde{\lambda}_k$ be the eigenvalues of $\Lambda_D-\Lambda_{\varnothing}$ and $\Delta\widetilde{\Lambda}$, respectively, ordered in decreasing order. Then
    \begin{equation*}
        \lvert \widetilde{\lambda}_k - \lambda_k \rvert \leq \delta
    \end{equation*}
\end{prop}

The Proposition~\ref{prop:pert} allows one to quantitatively evaluate the impact of noise on the eigenvalues of the key operator $\Lambda_D-\Lambda_{\varnothing}$. Specifically, given the noise level $\delta$ as the norm of the noise operator $N$, the eigenvalues of $\Delta\widetilde{\Lambda}$ differ at most by $\delta$ from the corresponding eigenvalues of the noise-free operator $\Lambda_D-\Lambda_{\varnothing}$. Therefore, all eigenvalue/eigenfunction pairs with the eigenvalue smaller than $\delta$ are not reliable, because they are corrupted by noise. 
\begin{figure}[tbhp]
\centering
\includegraphics[width=0.7\textwidth]{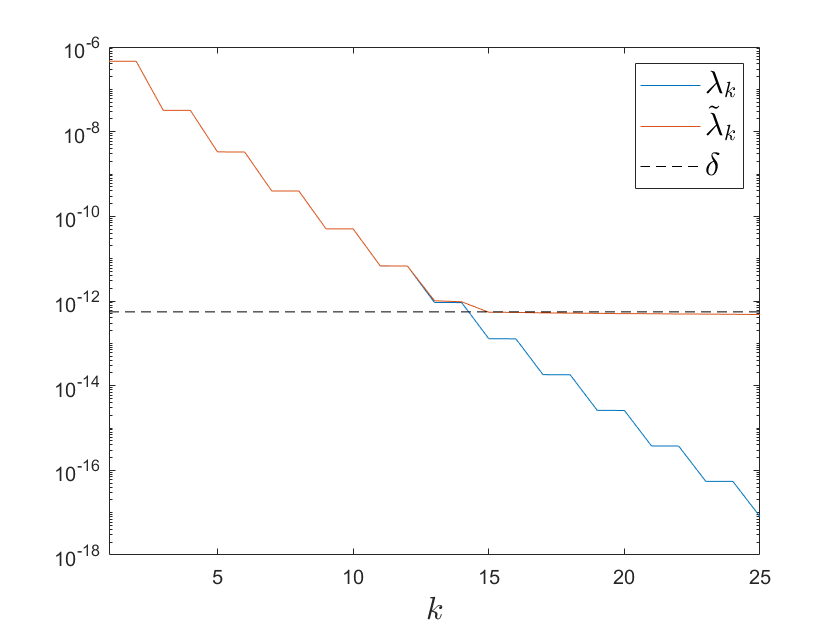}
\caption{Perturbation of the eigenvalues of the key operator $\Delta\Lambda$. The domain $\Omega$ is a circle centred in the origin with radius \qty{10}{\cm}, while the anomaly $D$ is a circle centred in the origin and radius \qty{4}{\cm}. The electrical conductivities are $\sigma_{bg}=\qty{200}{\siemens.\m}$ and $\sigma_a=\SI{1}{\siemens.\m}$. The norm of the operator $N$ is $\delta$.}
\label{fig:eign}
\end{figure}
An example is shown in Figure~\ref{fig:eign}, where the eigenvalues of the perturbed operator plateau when their amplitude is of the order of $\delta$. It is remarkable how they differ from that of the noise-free operator $\Delta \Lambda$.

As highlighted in the previous sections, the eigenfunction $g_n$ is chosen to get $\langle \Delta \Lambda g_n, g_n \rangle = \lambda_n \lVert g_n \rVert^2$ as close to zero as possible, to mimic an element of the (void) kernel of the operator $\Lambda_D-\Lambda_{\varnothing}$. In the presence of noisy data, this results in the eigenfunction $\widetilde{g}_n$ corresponding to the smallest eigenvalue $\widetilde{\lambda}_n$ above the noise level $\delta$. The latter can be easily identified due to the presence of the plateau that appears about $\delta$.

The remainder of this section is devoted to the choice of the threshold $\alpha$ appearing in the reconstruction rule of \eqref{eqn:rec}, in the presence of noisy data. 

Preliminarily, the following proposition is proven, which establishes the relationship between the measured data, i.e., $\Delta\widetilde{\Lambda}$, and the ohmic power absorbed by the unknown anomalous region. 
\begin{prop}\label{prop:sog1}
    Let $\Delta \widetilde{\Lambda}$ be the noisy data defined in \eqref{eqn:noise}, where $\Delta \Lambda=\Lambda_D-\Lambda_{\varnothing}$ is the noise-free operator and $N$ is a bounded noise operator with $\lVert N \rVert \leq \delta$. Let $g$ be the boundary condition applied to the reference configuration, then
    
    \begin{equation}\label{eqn:in1}
        \frac{\sigma_{a}^m}{\Delta \sigma}\left(\langle \Delta \widetilde{\Lambda} \, g, g \rangle - \delta\lVert g\rVert^2 \right)\leq\int_{D}\sigo(x)\abs{\nabla \varphi_{\varnothing}(x)}^2\,dx \leq \frac{\sigma_{bg}^M}{\Delta \sigma}\left(\langle \Delta \widetilde{\Lambda} \, g, g \rangle + \delta\lVert g\rVert^2 \right),
    \end{equation}
    where $D$ is the unknown anomalous region and $\Delta\sigma=\sigma_{bg}^m-\sigma_{a}^M$.
\end{prop}
\begin{proof}
    The noise-free data can be related to the measured data due to \eqref{eqn:noise}, leading to $\langle \Delta \Lambda \, g,g\rangle=\langle \Delta \widetilde{\Lambda} \, g,g\rangle-\langle N \, g,g\rangle$. Moreover, recalling the definition of operator norm \cite[Ch.1 Par. 4.1]{Ka95}
    \begin{equation*}
        -\delta \lVert g \rVert^2 \leq \langle N g,g \rangle \leq \delta \lVert g \rVert^2,
    \end{equation*}
    hence,
    \begin{equation}\label{eqn:pro1}
        \langle \Delta\widetilde{\Lambda} \, g,g\rangle-\delta\lVert g\rVert^2\leq\langle \Delta\Lambda \, g,g\rangle\leq\langle \Delta\widetilde{\Lambda} \, g,g\rangle+\delta\lVert g\rVert^2.
    \end{equation}
    By combining \eqref{eqn:pro0} and \eqref{eqn:pro1}, the claim follows.    
\end{proof}

Proposition \ref{prop:sog1} gives an estimate for the ohmic power absorbed by the anomaly based only on available quantities: the measured data $\Delta \widetilde{\Lambda}$, the noise level $\delta$ and the upper and lower bounds $\sigma_a^m$, $\sigma_a^M$, $\sigma_{bg}^m$ and $\sigma_{bg}^M$ to the electrical conductivities. The proposition \ref{prop:sog1} sets a constraint in the sense that a proper reconstruction method has to provide a reconstruction of the anomalous region $D$ satisfying \eqref{eqn:in1}.

In the Kernel Method, the anomalous region is evaluated by applying a proper eigenfunction $g_n$ of the difference operator $\Delta\Lambda$ as boundary data to the reference configuration. In the presence of noise, only the eigenfunctions $\widetilde{g}_n$ of the noisy operator $\Delta\widetilde{\Lambda}$ can be computed. By specialising Proposition \ref{prop:sog1} for $g=\widetilde{g}_n$, it holds
\begin{equation}\label{eqn:sogp1}
        \frac{\sigma_{a}^m}{\Delta\sigma_{G}}\left(\widetilde{\lambda}_n - \delta\right)\lVert\widetilde{g}_n\rVert^2\leq\int_{D} \widetilde{p}_{\varnothing}^n(x)\,dx \leq \frac{\sigma_{bg}^M}{\Delta\sigma_{G}}\left(\widetilde{\lambda}_n + \delta \right)\lVert\widetilde{g}_n\rVert^2,
\end{equation}
where $ \widetilde{p}_{\varnothing}^n(x)=\sigo(x)\abs{\nabla \widetilde{u}_{\varnothing}^{n}(x)}^2$ and $\widetilde{u}_{\varnothing}^{n}$ is the scalar potential evaluated for the reference configuration and the boundary data $\widetilde{g}_n$.

Equation \eqref{eqn:sogp1} suggests the introduction of a threshold on the ohmic power absorbed by the anomaly. Specifically, let $\widetilde{D}_{\alpha}^n$ be defined as
\begin{equation*}
    \widetilde{D}_{\alpha}^n=\bigg\{x\in\Omega \, | \, \frac{\widetilde{p}_{\varnothing}^n(x)}{\widetilde{\lambda}_n} < \alpha\bigg\},  \ \forall \ \alpha>0,
\end{equation*}
then the reconstructed region is $\widetilde{D}_{\alpha^*}^n$, where $\alpha^*$ is the unique solution of
\begin{equation}\label{eqn:da}
    \int_{\widetilde{D}_{\alpha}^n}\widetilde{p}_{\varnothing}^n(x)\,dx=\varepsilon^* \widetilde{\lambda}_n \lVert \tilde{g}_n \rVert^2,
\end{equation}
with respect to $\alpha$. In \eqref{eqn:da} the parameter $\varepsilon^*$ satisfies
\begin{equation}\label{eqn:thr}
\frac{\sigma_{a}^m}{\Delta\sigma_{G}}\left(1 - \frac{\delta}{\widetilde{\lambda}_n}\right) \le \varepsilon^* \le \frac{\sigma_{bg}^M}{\Delta\sigma_{G}} \left(1 + \frac{\delta}{\widetilde{\lambda}_n}\right).
\end{equation}

In summary, the Kernel Method for noisy data is
\begin{itemize}
\item Measure experimentally or compute numerically $\Lambda_{\varnothing}$;
\item measure the noisy data $\widetilde{\Lambda}_D$;
\item compute numerically the eigenvalues $\{ \widetilde{\lambda}_k \}_k$ and the eigenfunctions $\{ \widetilde{g}_k \}_k$ of $\Delta\widetilde{\Lambda}$;
\item find the smallest eigenvalue $\widetilde{\lambda}_{k^*}$ of $\Delta\widetilde{\Lambda}$ above the plateau and the corresponding eigenfunction $\widetilde{g}_{k^*}$;
\item solve problem~\eqref{eqn:dirprob} when the conductivity is $\sigo$ and the applied boundary data is $\widetilde{g}_{k^*}$;
\item find the value of $\alpha^*$ that solves \eqref{eqn:da}. The reconstruction of the unknown anomaly is, therefore, $\widetilde{D}_{\alpha^*}^n$.
\end{itemize}

\section{Analytical examples}\label{sec:anex}
In this Section, examples regarding canonical geometries are presented to clarify the main features of the Kernel Method and its limits. 

The reference geometry presents radial symmetry and, therefore, the eigenvalues and eigenfunctions of $\Lambda_D-\Lambda_{\varnothing}$ can be evaluated analytically. The data is assumed to be noise-free, i.e., $\delta=0$. 

The first case (see Section \ref{sec:cc}) refers to a circular anomaly in a circular domain. This case exemplifies the main steps of the methods for a 2D case.
The second case (see Section \ref{sec:hc}) highlights as the method reconstructs the outer support of the anomaly in the presence of cavities.

\subsection{Concentric circles}\label{sec:cc}
The first configuration consists of an unknown circular anomaly $D$ embedded in a circular domain $\Omega$, as shown in Figure \ref{fig:kman1}. The anomaly $D$ has radius $r_i$, while the domain $\Omega$ has radius $R$. The anomaly $D$ consists of a homogeneous disc of electrical conductivity $\sigma_a$, while the background consists of a homogeneous material of electrical conductivity $\sigma_{bg}$. 
\begin{figure}[ht]
\centering
\includegraphics[width=0.7\textwidth]{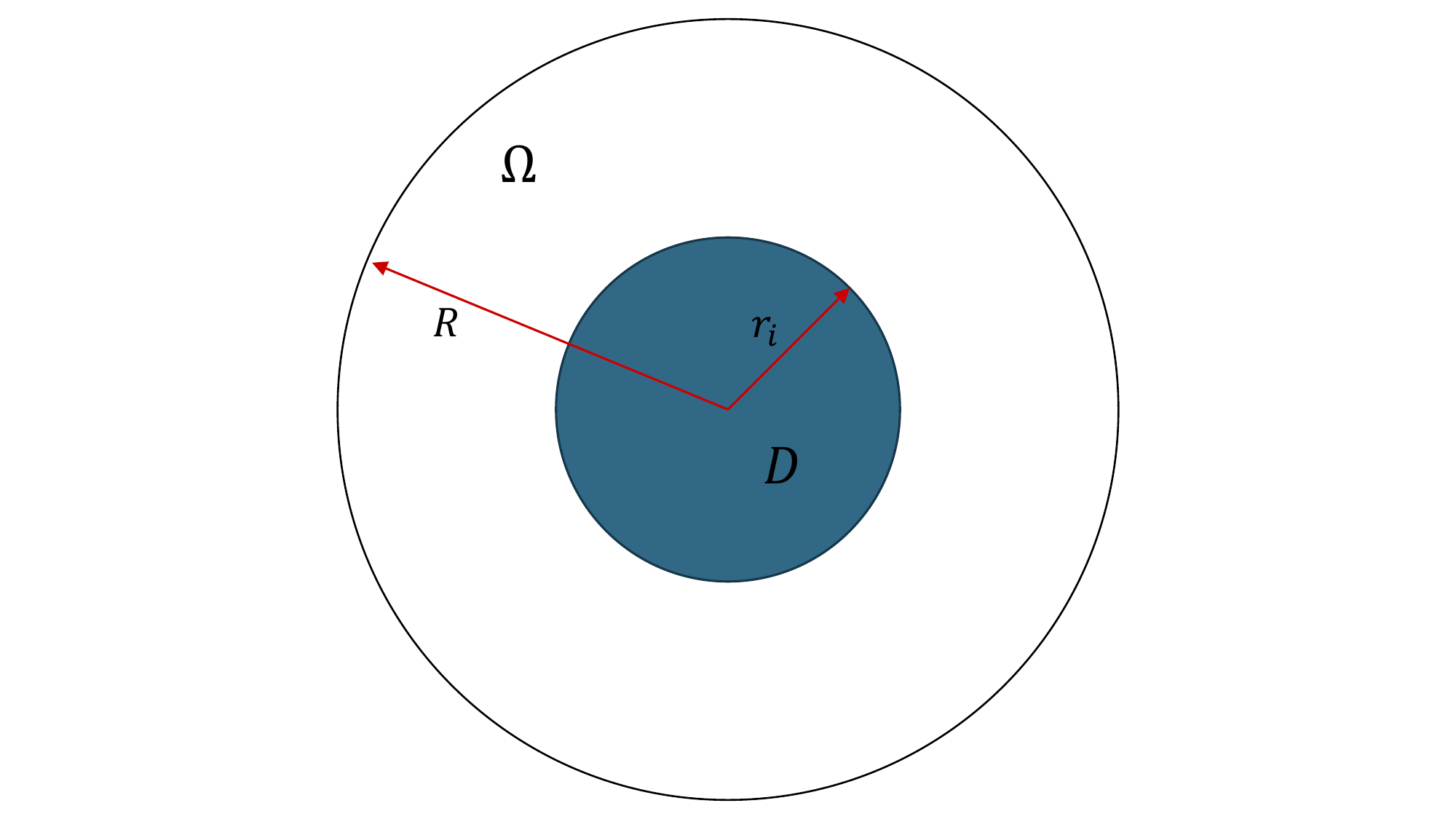}
\caption{The Kernel Method is applied to retrieve the unknown anomaly $D$.}
\label{fig:kman1}
\end{figure}

In this configuration, the eigenfunctions $g_n$ of $\Lambda_D-\Lambda_{\varnothing}$ are given by $g_n(\theta)=\frac{1}{\sqrt{\pi}}\cos(n\theta)$, $\forall n \in \mathbb{N}$, while the corresponding eigenvalues are
\begin{equation}\label{eqn:lam11}
\lambda_n=\frac{R}{\sigma_{bg}\abs{n}}\frac{2\left(\frac{r_i}{R}\right)^{\abs{n}}\left(\sigma_{bg}-\sigma_{a}\right)}{\left(\frac{r_i}{R}\right)^{\abs{n}}\left(\sigma_a-\sigma_{bg}\right)+\left(\frac{R}{r_i}\right)^{\abs{n}}\left(\sigma_a+\sigma_{bg}\right)}.
\end{equation}

For a given eigenfunction $g_n$ applied as Neumann boundary data to the reference configuration $\left( \sigo(x)=\sigma_{bg} \text{ in } \Omega \right)$), the corresponding solution $u^n_{\varnothing}$ of \eqref{eqn:dirprob} is
\begin{equation*}
u^n_{\varnothing}(r,\theta)=\frac{R}{n \sqrt{\pi} \sigma_{bg}}\left(\frac{r}{R}\right)^n\cos(n\theta),
\end{equation*}
with a related power density given by
\begin{equation}\label{eqn:prt}
p_{\varnothing}^n(r,\theta)=\sigma_{\varnothing}\abs{\nabla u_{\varnothing}^n}^2=\frac{1}{\sigma_{bg}\pi}\left(\frac{r}{R}\right)^{2n-2}.
\end{equation}

As already pointed out in Section \ref{sec:main}, the unknown anomaly is reconstructed by solving the equation
\begin{equation}\label{eqn:da_cc}
    \frac{\int_{D_{\alpha}^n} p_{\varnothing}^n(x)\,dx}{\lambda_n}=\varepsilon^*,
\end{equation}
with respect to $\alpha$, where $D_{\alpha}^n$, in this specific case, is given by
\begin{equation*}
    D_{\alpha}^n=\left\{x\in\mathbb{R}^2 \,\bigg|\, \frac{1}{\lambda_n}\frac{1}{\sigma_{bg}\pi}\left(\frac{r}{R}\right)^{2n-2}<\alpha\right\}
\end{equation*}
and $\varepsilon^*$ fulfils the following constraint 
\begin{equation*}
    \frac{\sigma_a}{\sigma_{bg}-\sigma_a}\leq \varepsilon^* \leq \frac{\sigma_{bg}}{\sigma_{bg}-\sigma_a}.
\end{equation*}
Suppose $\sigma_{bg}>2\sigma_a$, a possible choice is given by $\varepsilon^*=1$. 

Recognising that, for every $\alpha>0$, the set $D_{\alpha}^n$ is a circle centred at the origin, equation \eqref{eqn:da_cc} becomes
\begin{equation*}
\int_0^{2\pi}\int_0^{\tilde{r}_i} \frac{1}{\sigma_{bg}\pi}\left(\frac{r}{R}\right)^{2n-2}\,rdr\,d\theta=\lambda_n
\end{equation*}
which gives
\begin{equation}\label{eqn:recri}
\tilde{r}_i(n)=R\left(\frac{n\sigma_{bg}}{R^2}\lambda_n\right)^{\frac{1}{2n}}.
\end{equation}
\begin{figure}[ht]
\centering
\includegraphics[width=0.7\textwidth]{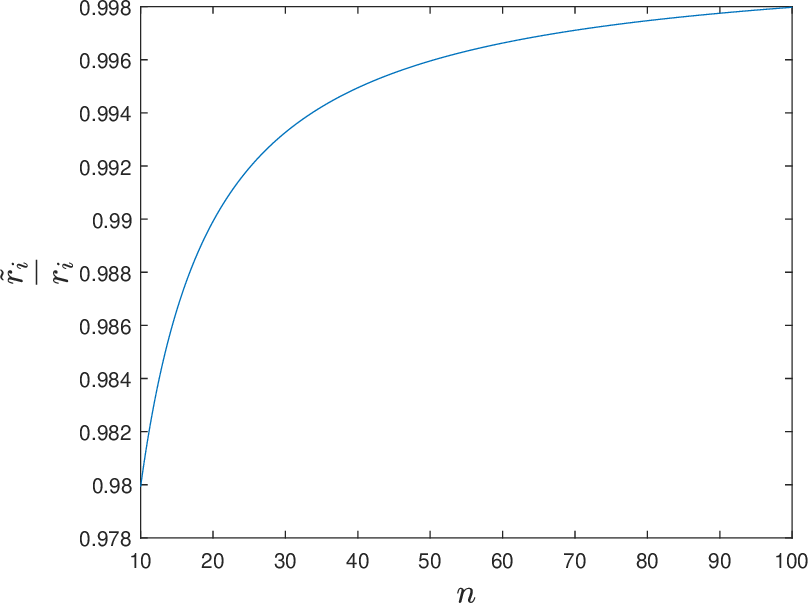}
\caption{Behaviour of the reconstructed internal radius $\tilde{r}_i$, with respect to the order of the selected eigenfunction.}
\label{fig:anex}
\end{figure}
In Figure \ref{fig:anex}, the behaviour of $\tilde{r}_i$ is depicted as a function of the order of the chosen eigenfunction. As can be seen, the reconstructed radius tends to the actual radius for $n\to+\infty$.

\subsection{Holed circle}\label{sec:hc}
This example illustrates the ability of the method to recover the outer support of the anomaly in the presence of cavities in $D$.
\begin{figure}[ht]
\centering
\includegraphics[width=0.7\textwidth]{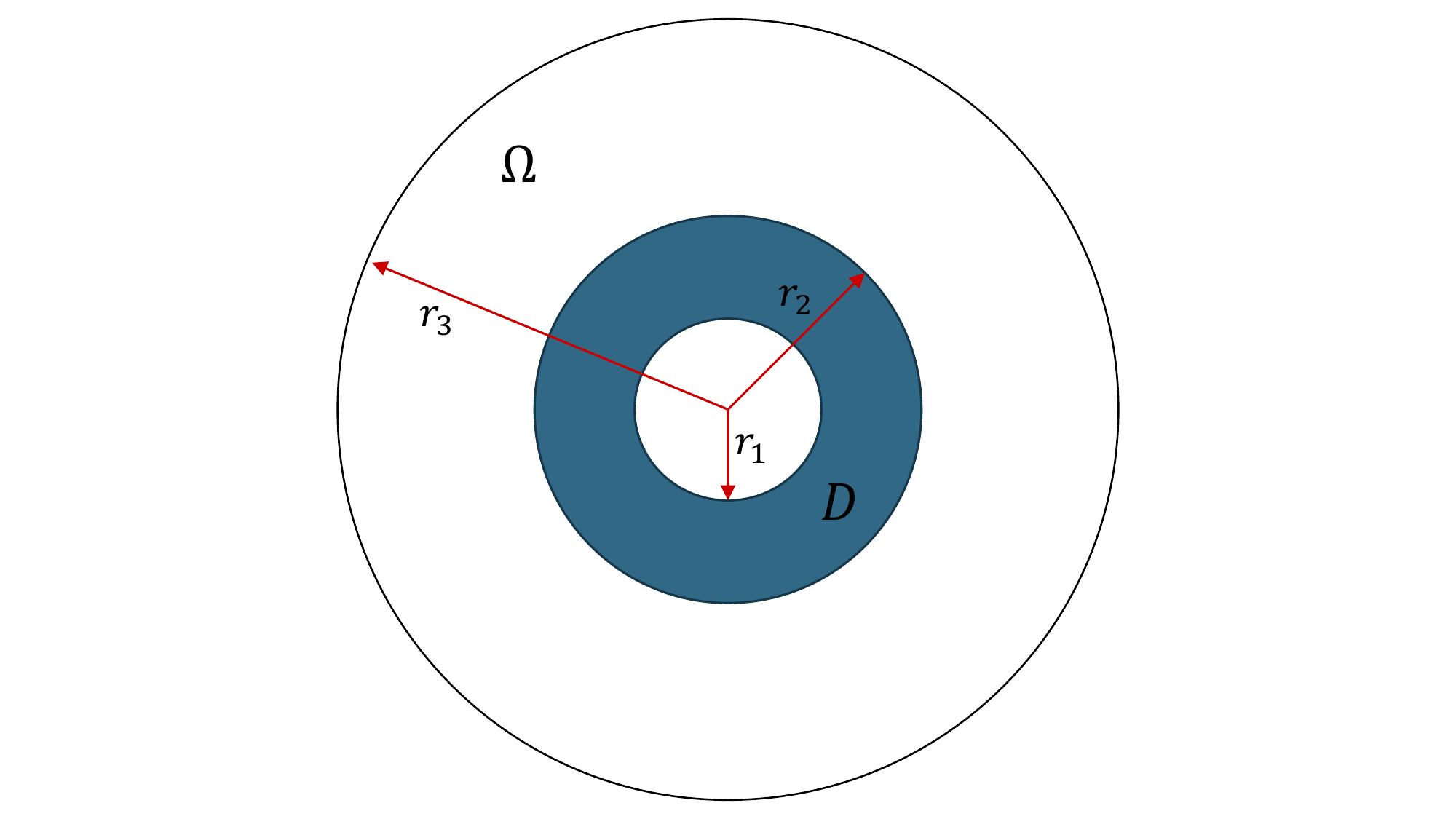}
\caption{The Kernel Method is applied to retrieve the circular crown $r_1\leq r\leq r_2$.}
\label{fig:kman2}
\end{figure}
Consider the configuration depicted in Figure \ref{fig:kman2}, where the objective is to retrieve the anomalous region occupying the circular crown $r_1\leq r \leq r_2$. The external radius of the circle is $r_3$, and the electrical conductivities are homogeneous and equal to $\sigma_a$ in the circular crown and $\sigma_{bg}$ elsewhere. 

A standard application of the separation of variables leads to 
\begin{equation}\label{eqn:lam12}
\begin{aligned}
g_n(\theta)&=\frac{1}{\sqrt{\pi}}\cos(n \theta) \\
\lambda_n&=\frac{r_3}{n\sigma_{bg}}\frac{2\left(\frac{r2}{r3}\right)^{n}\left[\left(\frac{r1}{r2}\right)^{n} -\left(\frac{r2}{r1}\right)^{n}\right](\sigma_{bg}^2-\sigma_a^2 )}{d_n} \\
d_n&=\left(\frac{r_1}{r_2}\right)^{n}\left[
\left(\frac{r_3}{r_2}\right)^{n}(\sigma_a-\sigma_{bg})^2+\left(\frac{r_2}{r_3}\right)^{n}(\sigma_a^2-\sigma_{bg}^2)\right] \\
&-\left(\frac{r_2}{r_1}\right)^{n}\left[
\left(\frac{r_3}{r_2}\right)^{n}(\sigma_a+\sigma_{bg})^2+\left(\frac{r_2}{r_3}\right)^{n}(\sigma_a^2-\sigma_{bg}^2)\right].
\end{aligned}
\end{equation}

As in the previous example, the Kernel Method requires driving the reference configuration with one of the eigenfunctions $g_n$ and then evaluating the power density on that configuration. The expression for the power density is that of \eqref{eqn:prt}, when replacing $R$ with $r_3$. 

Equation \eqref{eqn:prt} suggests that it is not possible to reconstruct the circular crown, due to the fact that the power density is a monotonic function of the radial coordinate $r$. This reflects a physical limit. Indeed, the Kernel Method seeks a proper boundary condition that generates a current density flowing externally of the anomalous region. It is evident that there is no pathway for the current density to circulate in the inner region $r<r_1$, without crossing the circular crown. 

However, the method still correctly recovers the outer boundary of the circular crown, i.e. reconstructs the outer support of the anomaly, as shown in Figure \ref{fig:anex2}.
\begin{figure}[ht]
\centering
\includegraphics[width=0.7\textwidth]{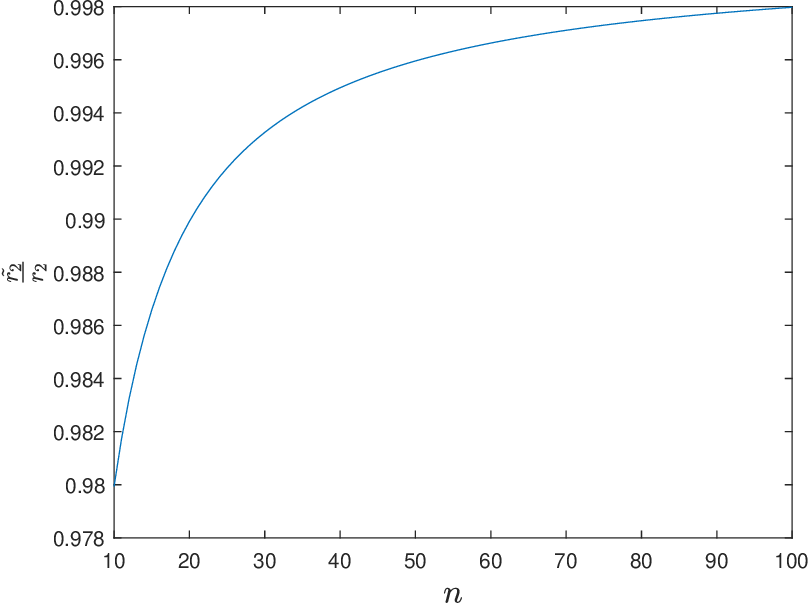}
\caption{Behaviour of the reconstructed outer radius $\tilde{r}_2$ of the circular crown, with respect to the order of the selected eigenfunction.}
\label{fig:anex2}
\end{figure}
The reconstructed radius is obtained from 
\begin{equation}
\tilde{r}_2(n)=R\left(\frac{n\sigma_{bg}}{r_3^2}\lambda_n\right)^{\frac{1}{2n}},
\end{equation}
which tends to $r_2$ as $n\to\infty$.

\section{Numerical Examples}\label{sec:num}
This last section demonstrates the effectiveness of the proposed method through examples of application. Specifically, the aim is to determine the shape, position, and dimensions of one or more anomalies that are less conductive than the background material. This is a typical case in applications where the ERT is used to determine defects in conductive materials.

The data are simulated by an in-house FEM code, based on the Galerkin method. The domain $\Omega$ is a circle with radius $r=\qty{2.5}{\cm}$ and is discretised in \num{16384} triangular elements, while the boundary $\partial\Omega$ is given by $N_b=\num{256}$ elements of the same length. The mesh has $N_p=\num{8321}$ nodes. The electric potential is discretised with a piecewise linear function, as the boundary potential, while the imposed normal component of the current density $g$ is discretised by a piecewise constant function on the boundary elements. In this discrete setting, the NtD operator is a linear map from $\mathbb{R}^{N_b-1}$ to $\mathbb{R}^{N_b-1}$, taking into account the constraint of vanishing integral mean on the boundary for both the applied current density and the measured voltages.

All reconstructions are obtained by adding synthetic noise to the data. Specifically, the following noise model is adopted, which is the discrete counterpart of \eqref{eqn:noise}
\begin{align*}
\Delta\widetilde{\mathbf{R}}&=\mathbf{R}_D-\mathbf{R}_{bg}+\eta \delta_R \mathbf{N} \\
\delta_R&=\max_{i,j}\abs{R_D(i,j)-R_{bg}(i,j)} \\
\mathbf{N}&=\frac{\mathbf{A}+\mathbf{A}^T}{2} \\
A(i,j)&\sim\mathcal{N}(0,1)
\end{align*}
where $\mathbf{R}_D$ and $\mathbf{R}_{bg}$ are the discretization of $\Lambda_D$ and $\Lambda_{\varnothing}$ on a finite subspace of $L_{\diamond}^2(\partial\Omega)$, respectively, $\Delta\widetilde{\mathbf{R}}$ is the noisy version of $\mathbf{R}_D-\mathbf{R}_{bg}$ and $\eta$ is a constant that controls the noise level.

The background material has a uniform electrical conductivity equal to $\sigma_{bg}=\qty{200}{\siemens./\m}$, while the electrical conductivity of the anomaly $D$ is equal to $\sigma_a=\SI{1}{\siemens./\m}$. 

The threshold $\varepsilon^*$ (see equation \ref{eqn:da}) is set to $\widetilde{\lambda}_n$, which is a feasible choice since $\sigma_{bg}>2\sigma_a$ (see also Section \ref{sec:cc}).

The reconstructions are reported in Figures~\ref{fig:ccircles} and~\ref{fig:ccircles2}, where a suitable set of different configurations is proposed to test various aspects of the imaging method. From these Figures, it is possible to infer that, at a realistic noise level, the reconstructions are of high quality. Specifically, the reconstructions demonstrate (i) excellent performance on convex anomalies (A, B, D, F, G, J), (ii) possibility of resolving multiple anomalies (H, N, O), (iii) tendency to convexification (C, E, I, K, L, M), and (iv) impossibility of reconstructing the interior of an anomaly with cavities (I).
\begin{figure}[tbhp]
\centering
\subfloat[]{\includegraphics[width=0.25\textwidth]{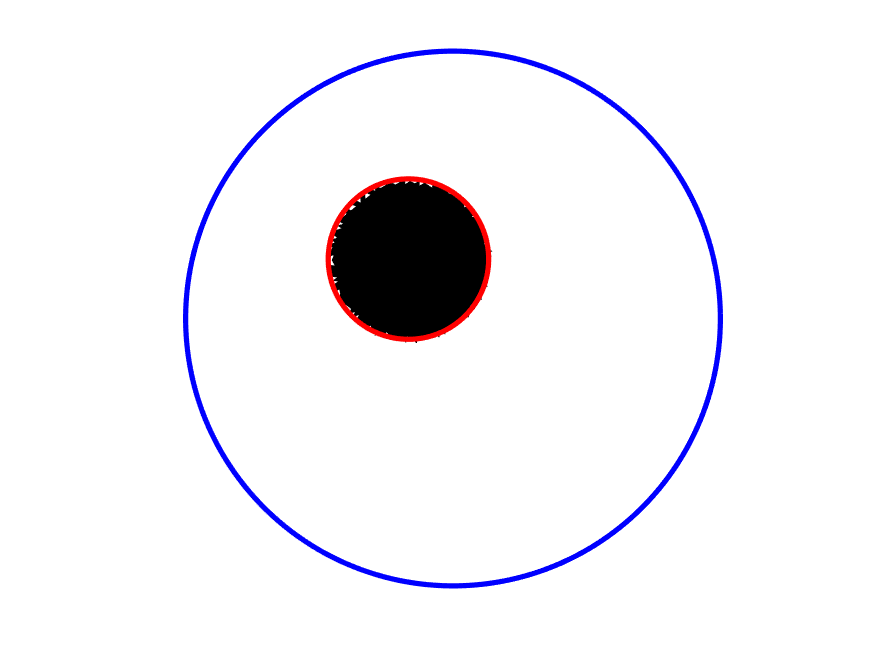}} \quad
\subfloat[]{\includegraphics[width=0.25\textwidth]{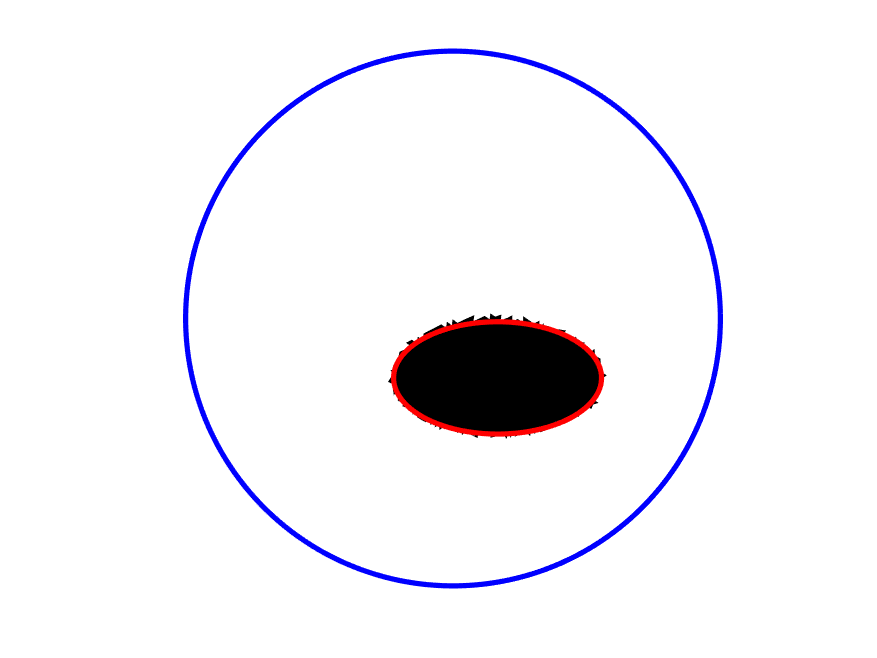}} \quad
\subfloat[]{\includegraphics[width=0.25\textwidth]{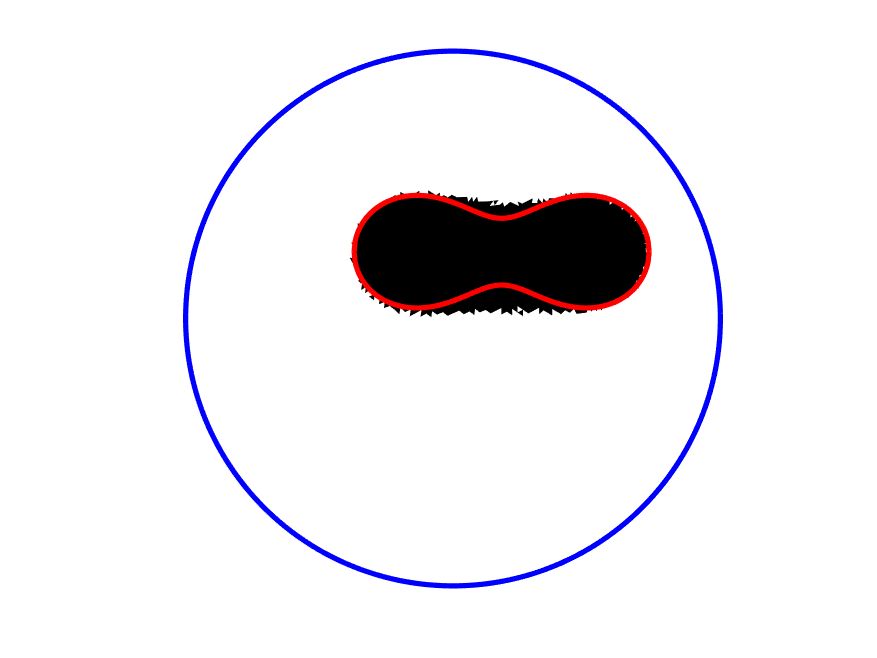}} \\
\subfloat[]{\includegraphics[width=0.25\textwidth]{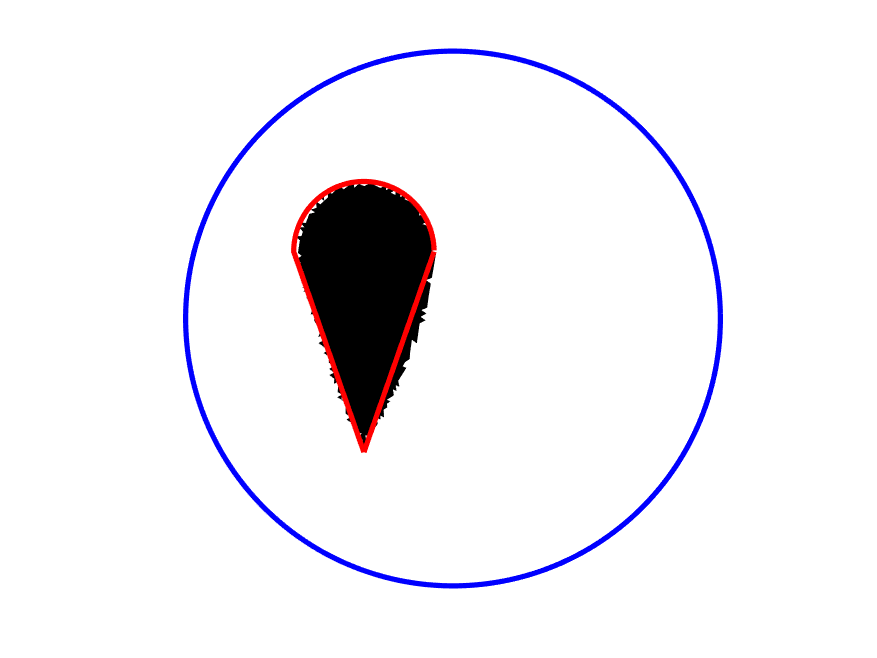}} \quad
\subfloat[]{\includegraphics[width=0.25\textwidth]{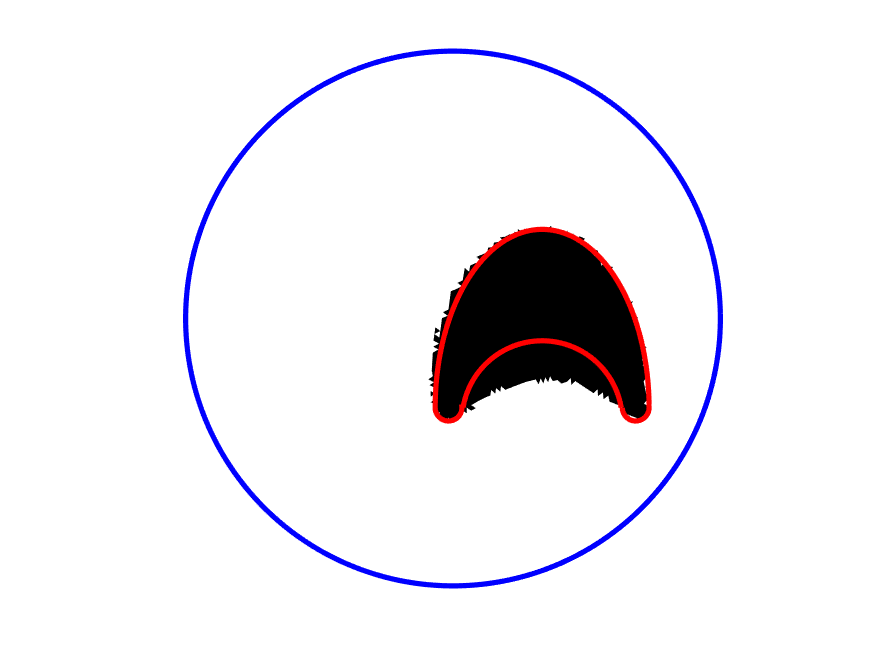}} \quad
\subfloat[]{\includegraphics[width=0.25\textwidth]{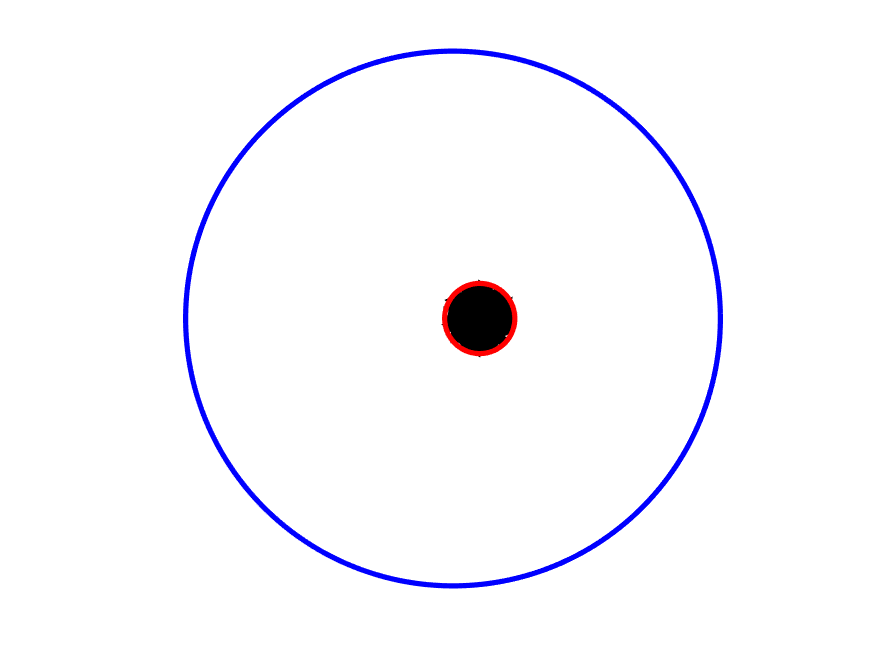}} \\
\subfloat[]{\includegraphics[width=0.25\textwidth]{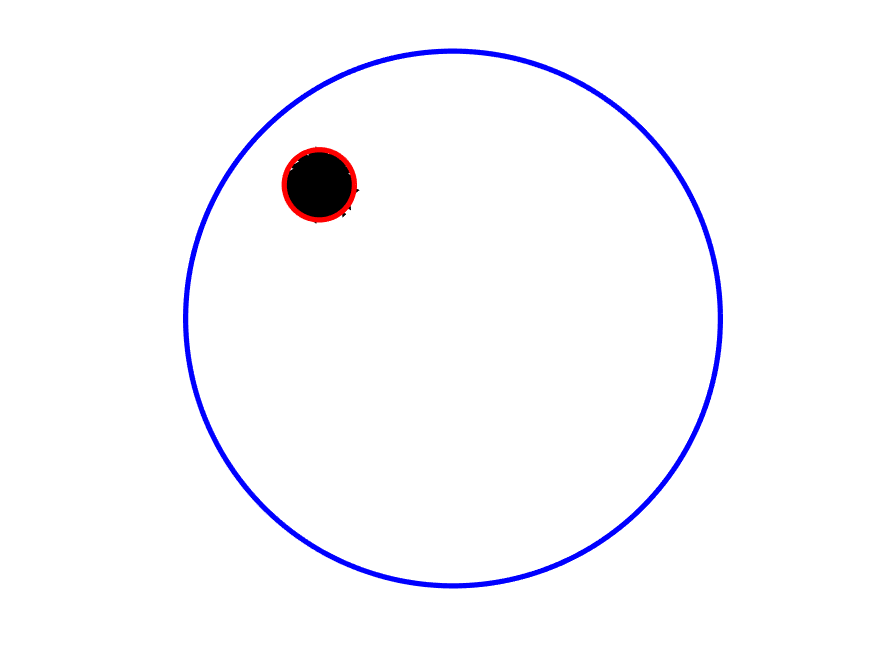}} \quad
\subfloat[]{\includegraphics[width=0.25\textwidth]{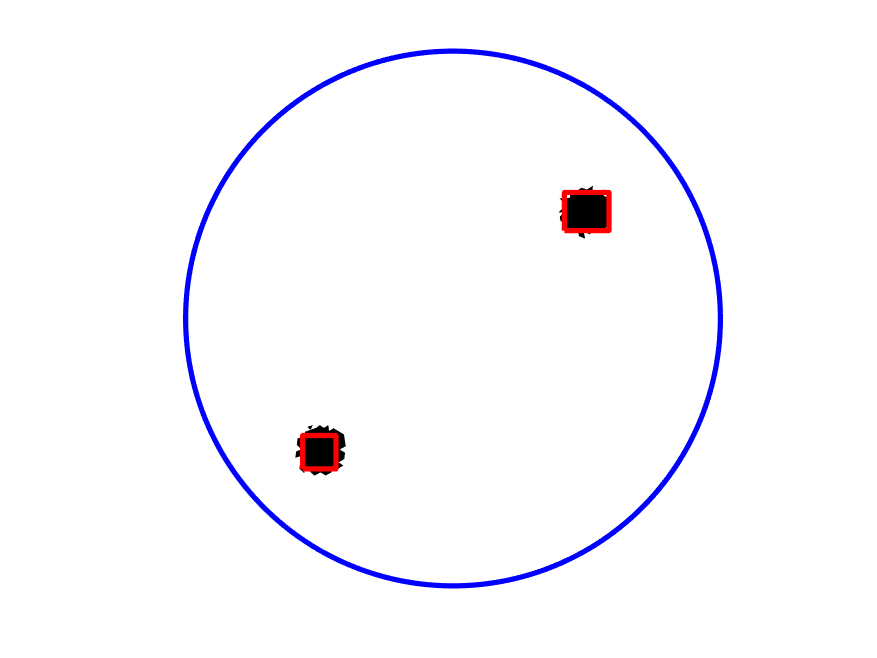}} \quad
\subfloat[]{\includegraphics[width=0.25\textwidth]{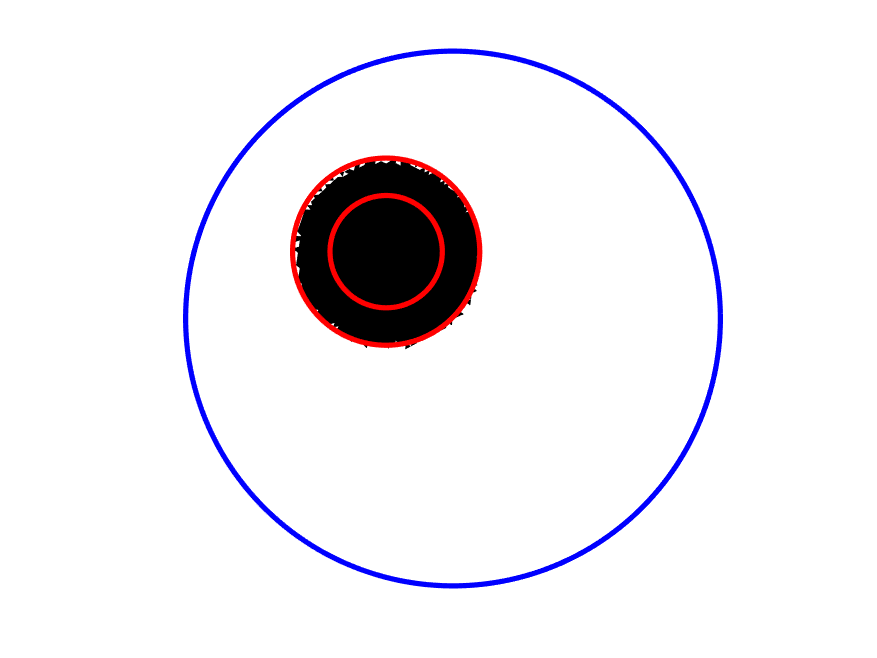}} \\
\subfloat[]{\includegraphics[width=0.25\textwidth]{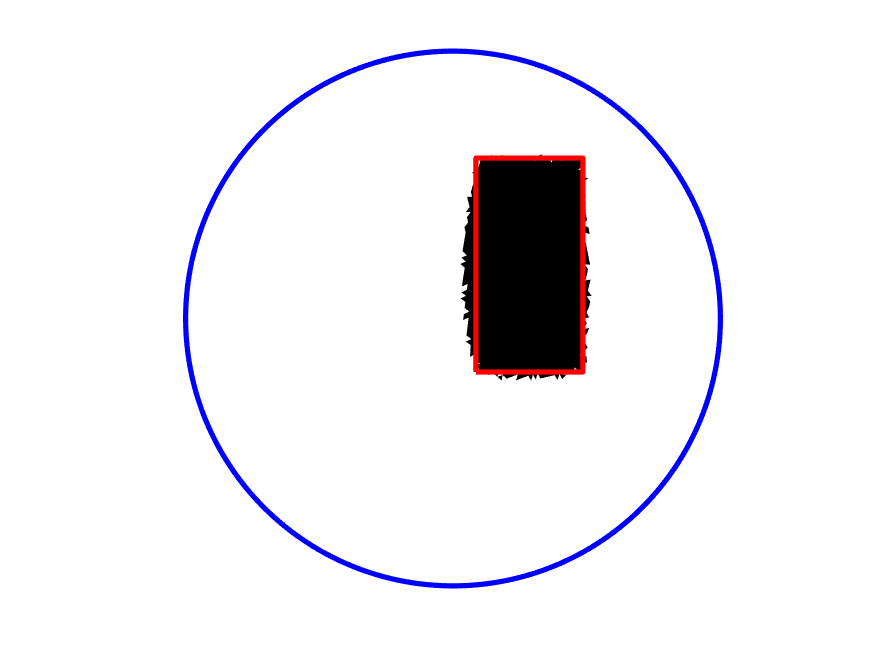}} \quad
\subfloat[]{\includegraphics[width=0.25\textwidth]{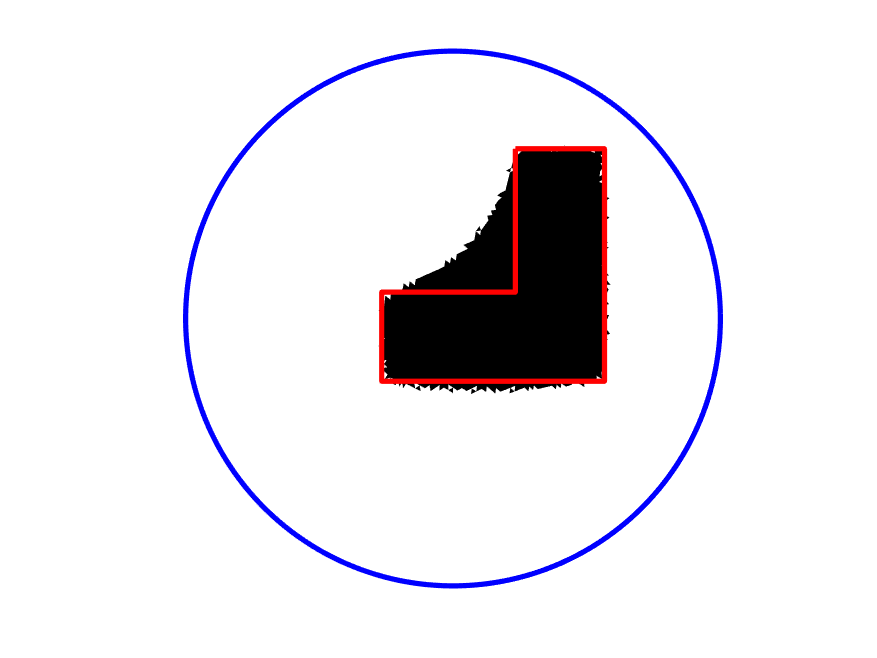}} \quad
\subfloat[]{\includegraphics[width=0.25\textwidth]{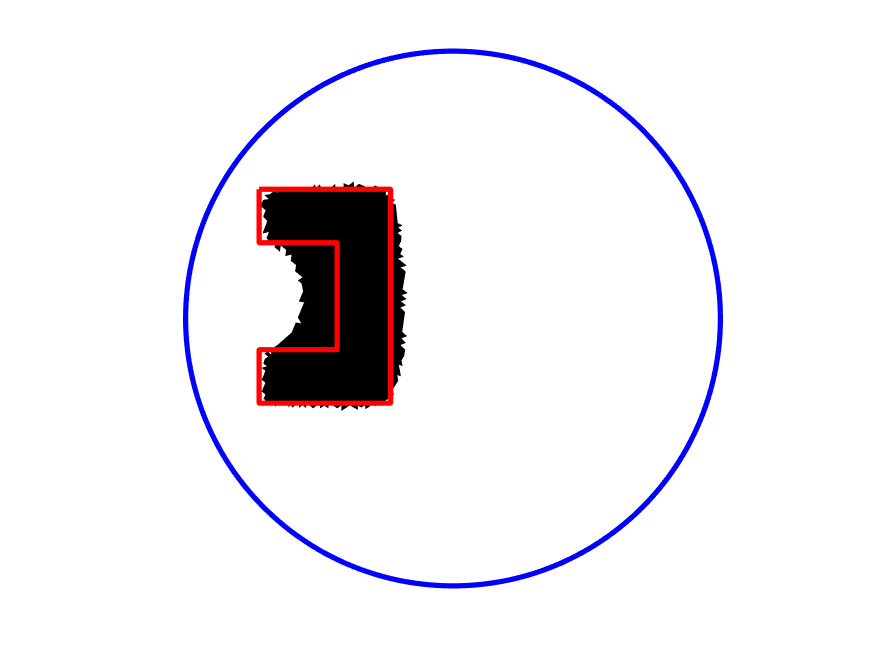}} \\
\subfloat[]{\includegraphics[width=0.25\textwidth]{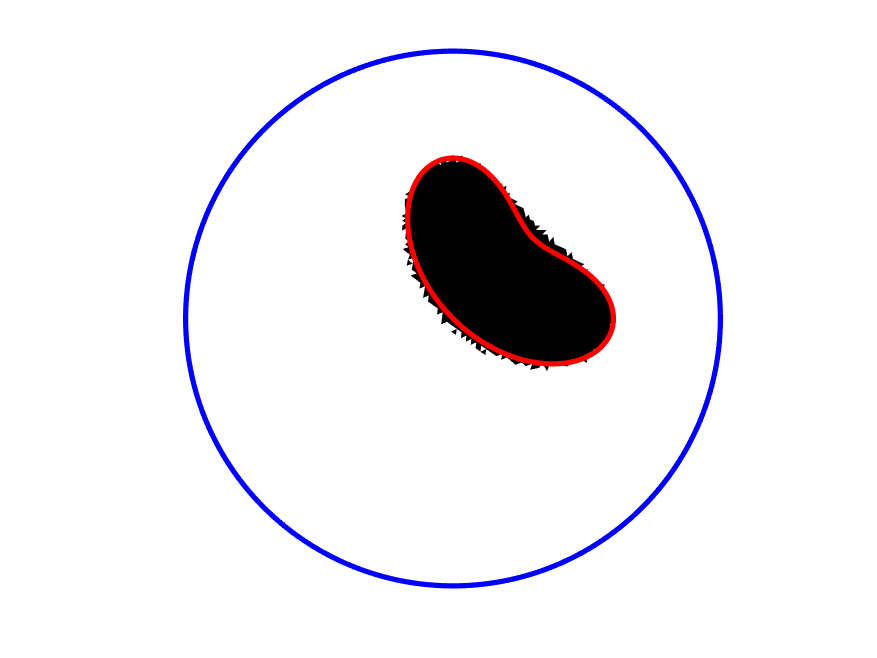}} \quad
\subfloat[]{\includegraphics[width=0.25\textwidth]{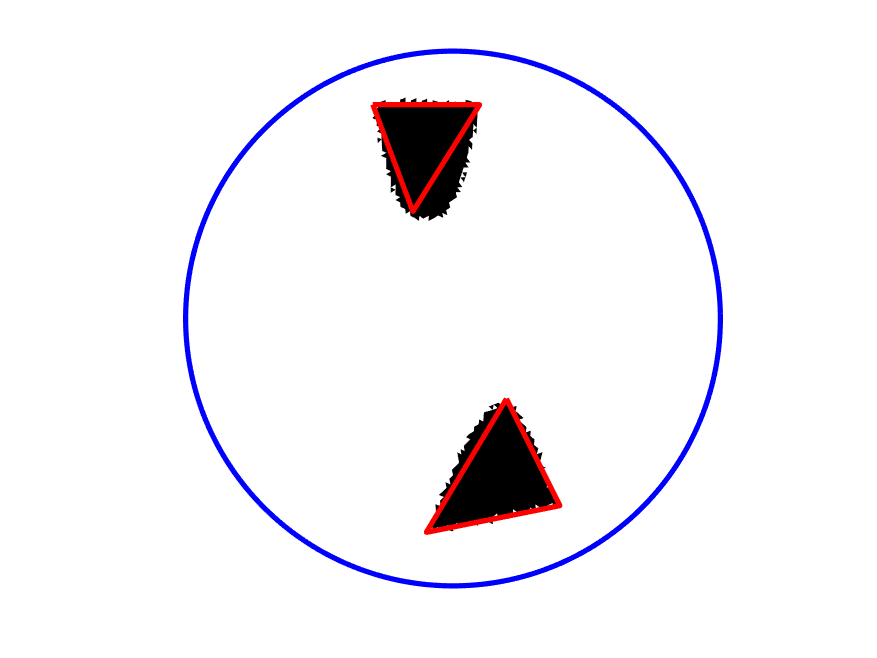}} \quad
\subfloat[]{\includegraphics[width=0.25\textwidth]{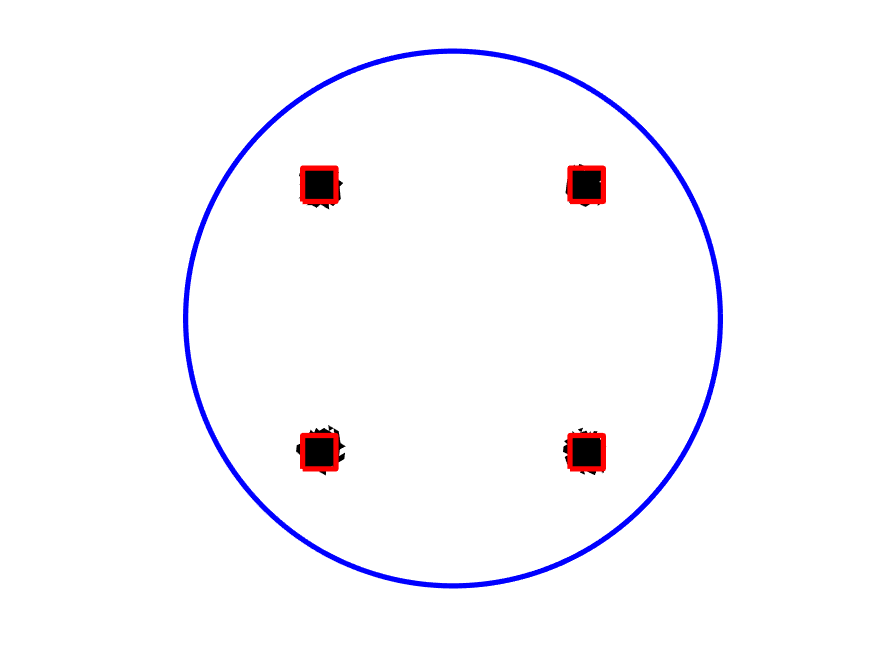}} \\
\caption{Reconstructions for $\eta=0$ (within the machine precision). In black the reconstructed anomaly $\tilde{D}$, while in red the boundary of the actual anomaly $D$.}
\label{fig:ccircles}
\end{figure}

\begin{figure}[tbhp]
\centering
\subfloat[]{\includegraphics[width=0.25\textwidth]{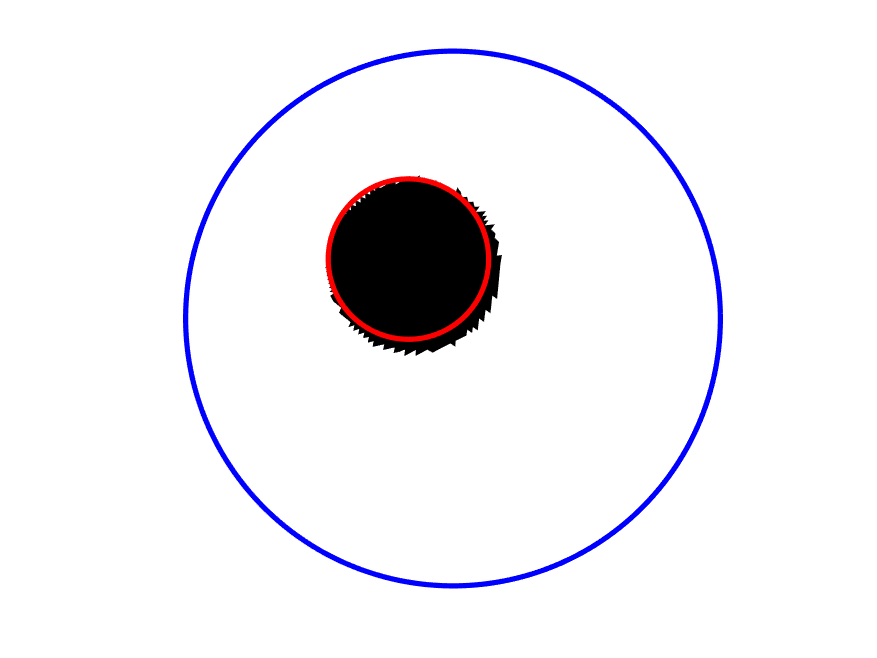}} \quad
\subfloat[]{\includegraphics[width=0.25\textwidth]{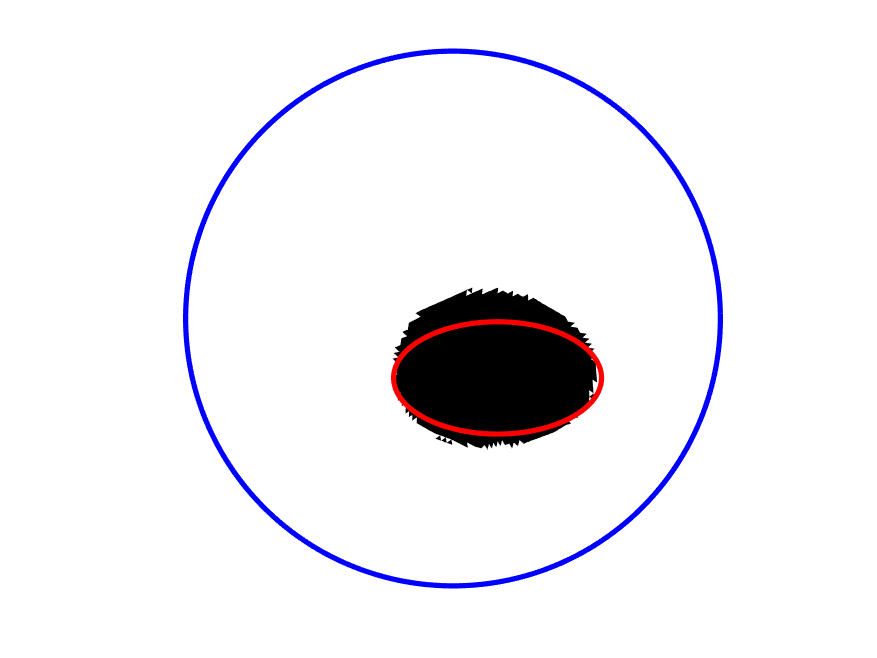}} \quad
\subfloat[]{\includegraphics[width=0.25\textwidth]{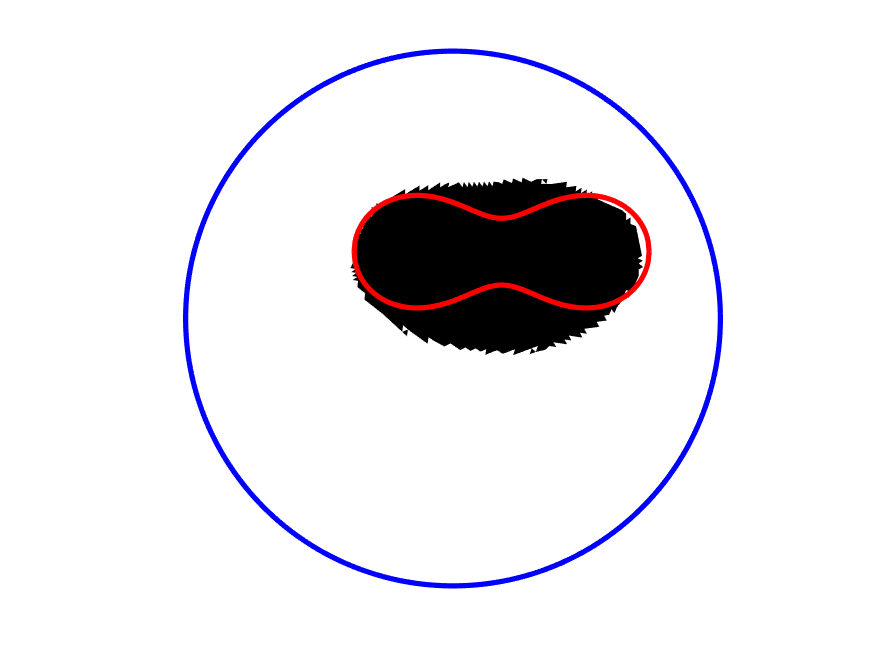}} \\
\subfloat[]{\includegraphics[width=0.25\textwidth]{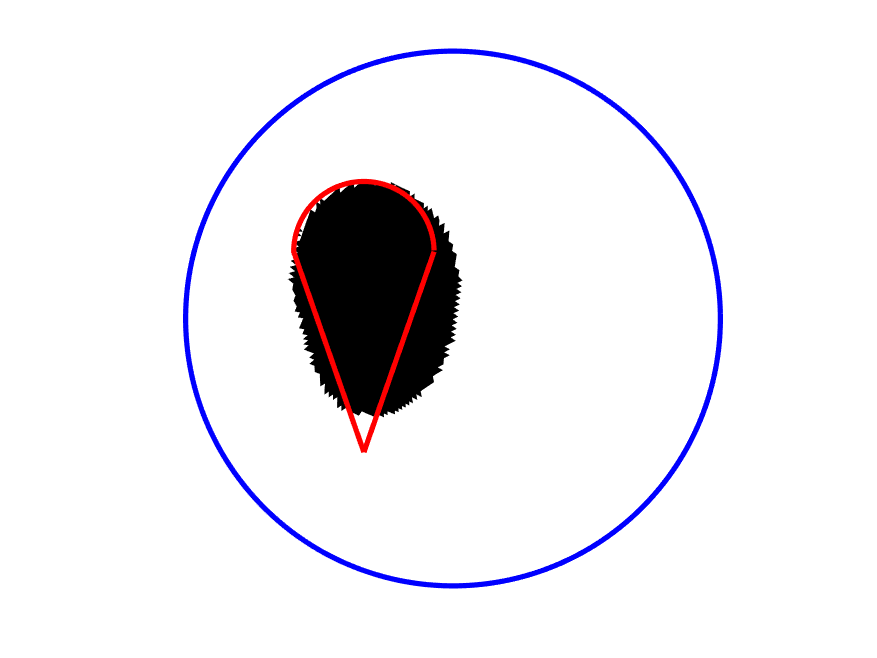}} \quad
\subfloat[]{\includegraphics[width=0.25\textwidth]{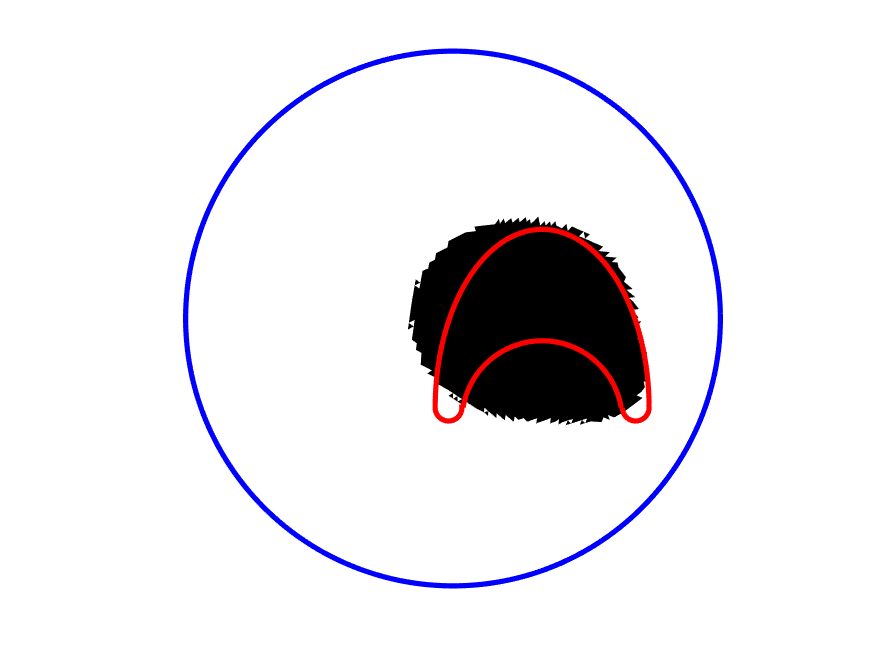}} \quad
\subfloat[]{\includegraphics[width=0.25\textwidth]{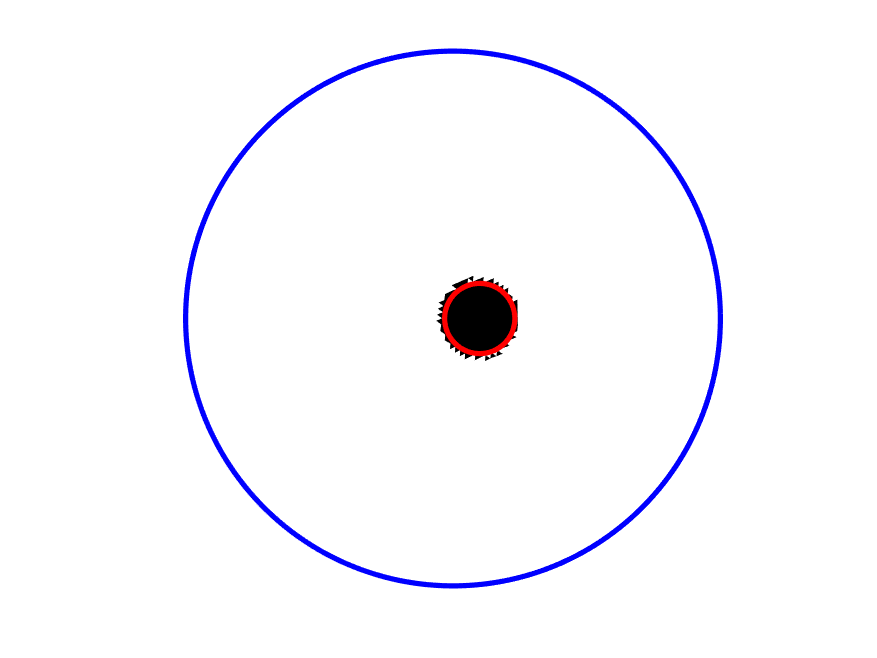}} \\
\subfloat[]{\includegraphics[width=0.25\textwidth]{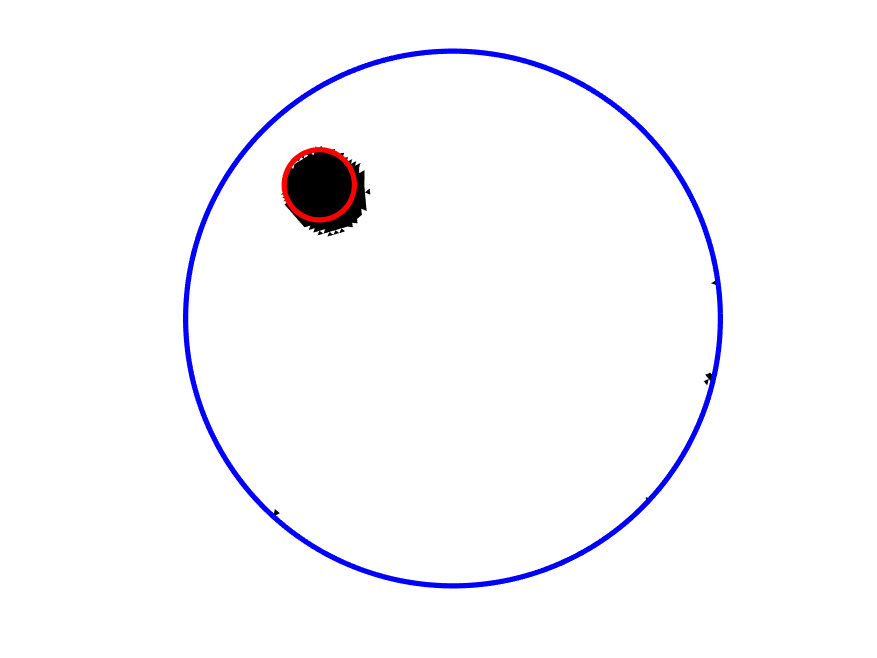}} \quad
\subfloat[]{\includegraphics[width=0.25\textwidth]{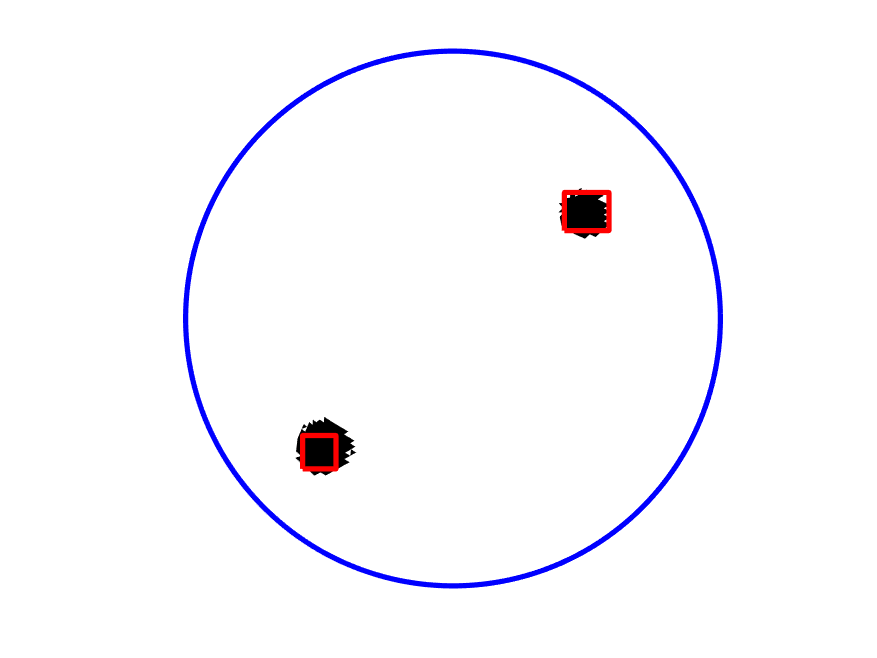}} \quad
\subfloat[]{\includegraphics[width=0.25\textwidth]{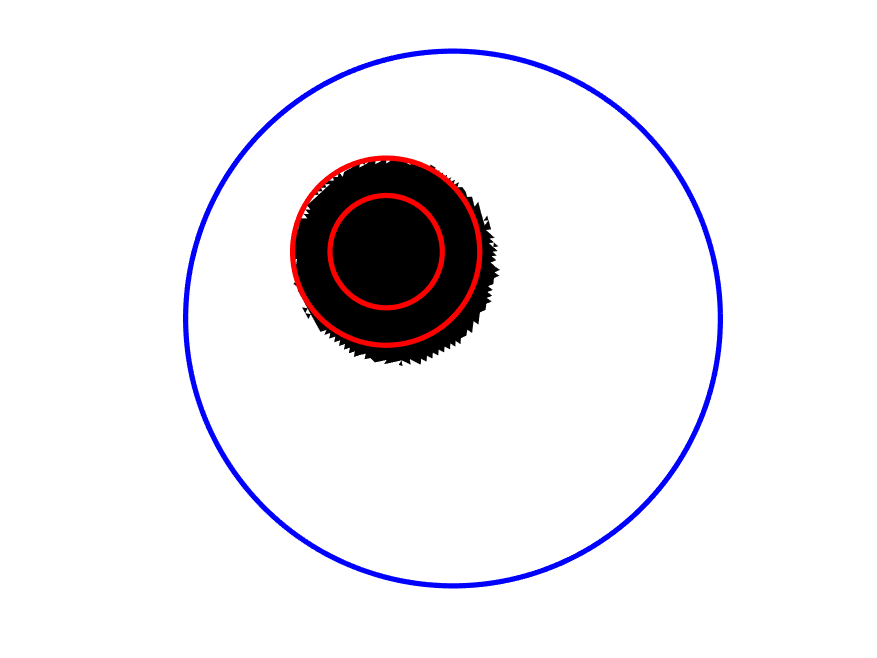}} \\
\subfloat[]{\includegraphics[width=0.25\textwidth]{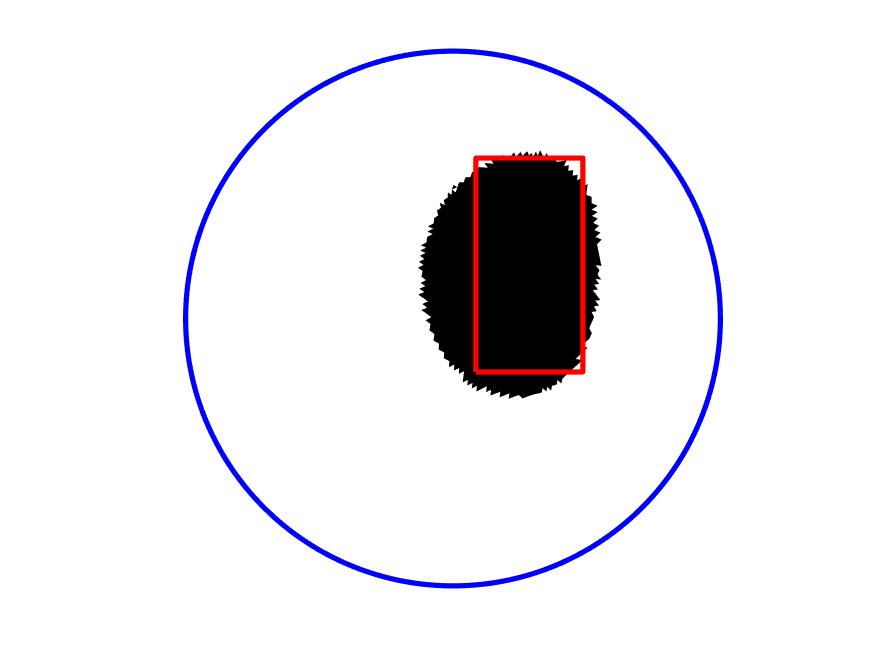}} \quad
\subfloat[]{\includegraphics[width=0.25\textwidth]{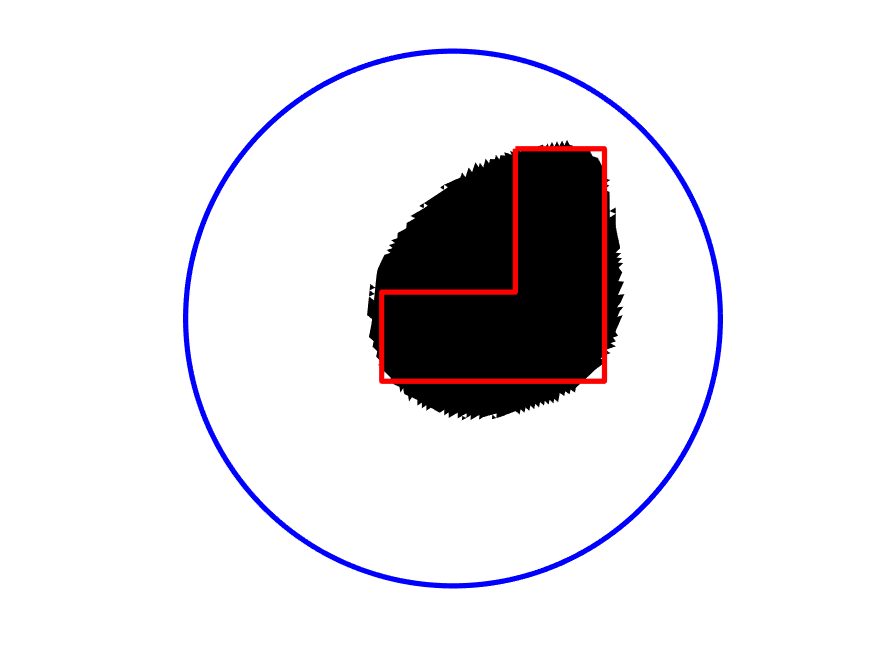}} \quad
\subfloat[]{\includegraphics[width=0.25\textwidth]{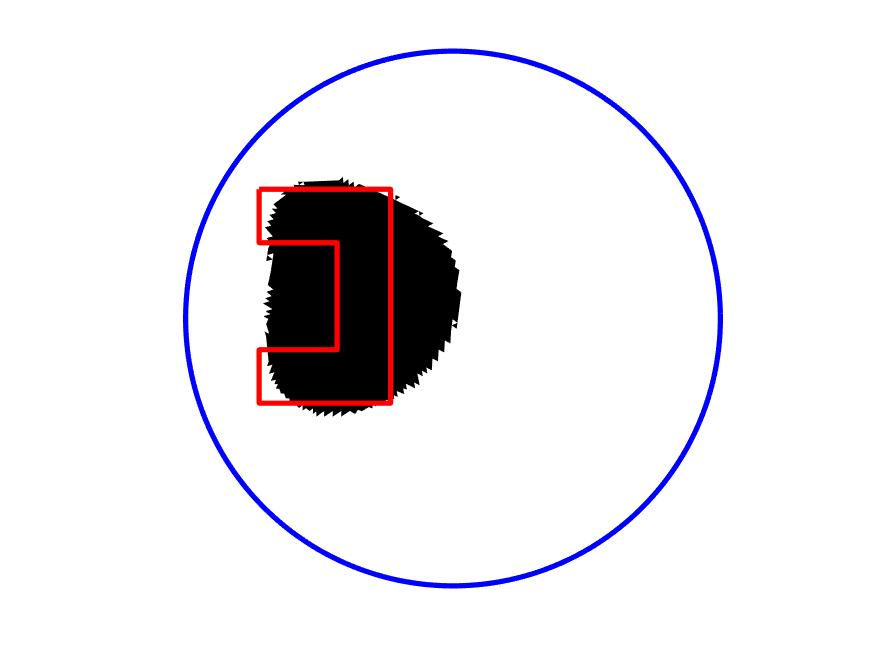}} \\
\subfloat[]{\includegraphics[width=0.25\textwidth]{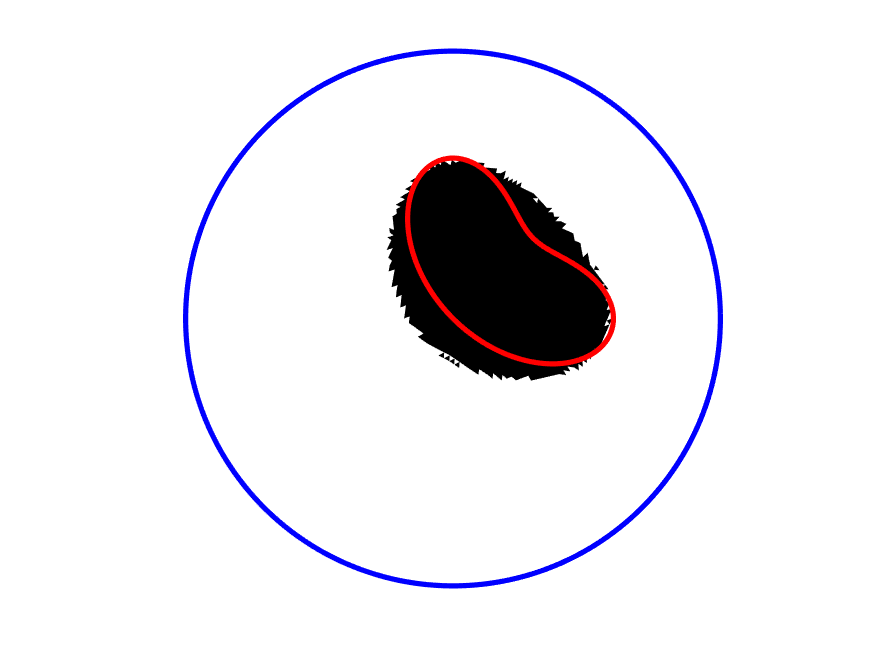}} \quad
\subfloat[]{\includegraphics[width=0.25\textwidth]{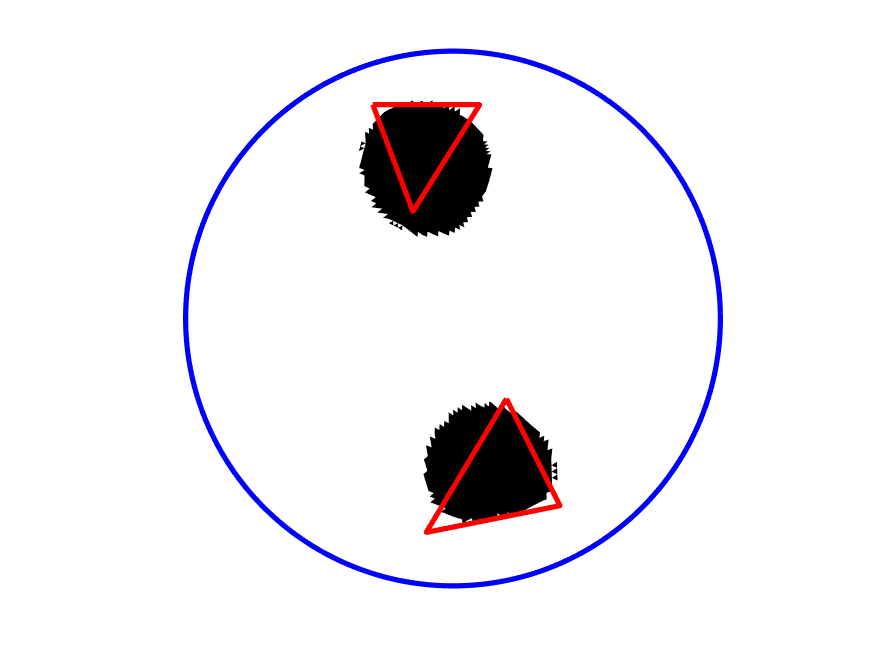}} \quad
\subfloat[]{\includegraphics[width=0.25\textwidth]{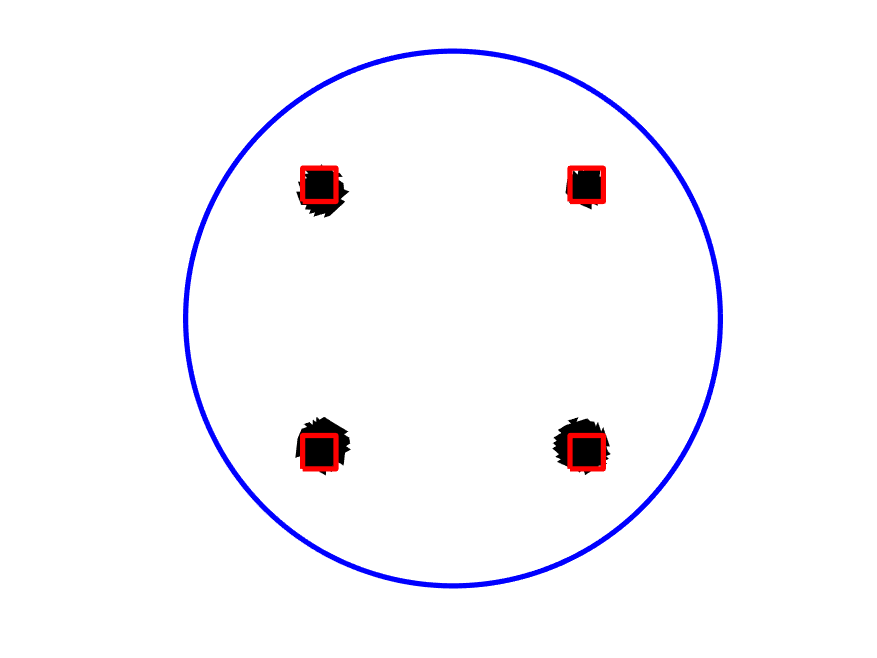}} \\
\caption{Reconstructions for $\eta=10^{-3}$. In black the reconstructed anomaly $\tilde{D}$, while in red the boundary of the actual anomaly $D$.}
\label{fig:ccircles2}
\end{figure}

\section{Conclusions}\label{sec:conclusions}
In this work, a new non-iterative imaging method for Electrical Resistance Tomography is presented. Specifically, the method is designed for the inverse obstacle problem, where the goal is to recover an unknown anomaly embedded in a known background. 

The method relies on the possibility of finding a boundary condition $g$ that produces a current density that is almost vanishing in the regions occupied by the anomaly, when applied to the reference configuration without the anomaly. This \lq\lq special\rq\rq \ boundary data $g$ is proven to be given by the higher order eigenfunctions of the operator $\Lambda_D-\Lambda_{\varnothing}$, which is the difference between the Neumann-to-Dirichlet map measured in the presence of an unknown anomaly $D$ and the Neumann-to-Dirichlet map computed or measured for the reference configuration.

The imaging method, characterised by a very simple and stable numerical implementation even in the presence of noise, provides a straightforward way to evaluate the outer support of the unknown anomaly $D$ at minimal computational cost.

Numerical examples demonstrate that the proposed algorithm achieves excellent performance.

\section*{Acknowledgment}
This work was supported by the Italian Ministry of University and Research under the PRIN-2022, Grant Number 2022Y53F3X \lq\lq Inverse Design of High-Performance Large-Scale Metalenses \rq\rq.

\section*{Authorship contribution statement}

{\bf A. Tamburrino}: Conceptualization, Methodology, Formal analysis, Writing, Supervision.

{\bf V. Mottola}: Conceptualization, Methodology, Formal analysis, Writing. 

\appendix
\section{Approximation of kernel sequences}\label{proof_1}

\subsection{Proof of \Cref{prop_approx}}
Since sequences with finitely non-zero terms are dense in $\ell^2(\mathbb{R})$, for each $f_n$, and for any $\varepsilon>0$, there exists $\beta_n<\infty$ such that
\begin{equation}\label{eqn_ub_sc}
    \norm{ \sum_{k=\beta_n+1}^{+\infty}\langle f_n,g_k\rangle g_k }^2 < \varepsilon.
\end{equation}
This yields the claim for $\alpha_n=1$, indeed
\begin{equation*}
    \norm{f_n - \sum_{k=1}^{\beta_n} \langle f_n , g_k \rangle g_k}^2=\sum_{k=\beta_n+1}^{+\infty} \langle f_n , g_k \rangle^2 < \varepsilon.
\end{equation*}

In the following, it is shown that $\alpha_n$ shifts towards higher indices as $n$ increases. For a given $K>0$, 
\begin{equation*}
     \langle (\Lambda_D - \Lambda_{\varnothing}) f_n, f_n \rangle = \sum_{k=1}^{+\infty}\lambda_k |\langle f_n,g_k\rangle|^2  \geq \sum_{k=1}^{K}\lambda_k |\langle f_n,g_k\rangle|^2
    \geq \lambda_K \sum_{k=1}^{K}|\langle f_n,g_k\rangle|^2,
\end{equation*}
where the first equality follows from Equation \eqref{eqn_sc_1}, and the leftmost inequality follows from $\{\lambda_k\}$ being a monotonic decreasing sequence. Furthermore, since $\{f_n\}_n$ is a kernel sequence, it follows that
\begin{equation*}
    \lim_{n\to+\infty}\sum_{k=1}^{K}|\langle f_n,g_k\rangle|^2
    \leq \lim_{n\to+\infty} \frac{\langle (\Lambda_D - \Lambda_{\varnothing}) f_n, f_n \rangle}{\lambda_K}
    = 0. 
\end{equation*}
Hence, for any $\varepsilon>0$ and for any fixed $K$, there exists $n_0$ such that, for all $n\ge n_0$,
\begin{equation*}
    \norm{f_n- \sum_{k=K+1}^{+\infty}\langle f_n,g_k\rangle g_k }^2= \sum_{k=1}^{K}|\langle f_n,g_k\rangle|^2
    < \varepsilon.
\end{equation*}
This shows that the projection of $f_n$ onto any fixed finite-dimensional eigenspace becomes negligible as $n$ increases. 

Therefore, for a given element $f_n$, there exists $\alpha_n$ such that
\begin{equation}\label{eqn_lb_sc}
    \norm{ f_n-\sum_{k=\alpha_n}^{+\infty}\langle f_n,g_k\rangle g_k }^2 < \varepsilon.
\end{equation}
Combining \eqref{eqn_ub_sc} and \eqref{eqn_lb_sc}, yields
\begin{equation*}
    \norm{f_n- \sum_{k=\alpha_n}^{\beta_n}\langle f_n,g_k\rangle g_k}^2 = \norm{\sum_{k=1}^{\alpha_n-1}\langle f_n,g_k\rangle g_k + \sum_{k=\beta_n+1}^{+\infty}\langle f_n,g_k\rangle g_k}^2<2\varepsilon.
\end{equation*}

\subsection{Approximation of the net power}
\Cref{prop_approx} shows that any element of a kernel sequence can be approximated, with arbitrary precision, by an element of $L^2_{\diamond}(\partial\Omega)$ given by a finite linear combination of eigenfunctions. In this section, sufficient conditions are discussed under which such an approximation leaves the ohmic power in a prescribed domain $S$ almost unchanged.

Given an element $f_n$ of a kernel sequence, let $\widetilde{f}_n$ be the projection of $f_n$ onto the vector space spanned by the first eigenfunctions $g_{1},\ldots,g_{N}$ of $\Lambda_D-\Lambda_{bg}$, and let $r_n$ be the residual part, that is, the projection of $f_n$ onto the vector space spanned by the remaining eigenfunctions $g_{N+1},\ldots $. The element $f_n$ can be written as%
\begin{equation*}
    f_n=\widetilde{f}_n+r_n.    
\end{equation*}

Let $\varphi_n$ be the scalar potential due to $f_n$, and let $\widetilde{\varphi}_n$ and $\psi_n$ be the scalar potentials due to $\widetilde{f}_n$ and $r_n$, respectively. It turns out that
\begin{equation*}
    \int_{S}\sigma _{\varnothing }\left\vert \nabla \varphi_n \right\vert ^{2}%
    \text{d}x=\int_{S}\sigma _{\varnothing }\left\vert \nabla \widetilde{\varphi 
    }_n\right\vert ^{2}\text{d}x+\int_{S}\sigma _{\varnothing }\left\vert \nabla
    \psi_n \right\vert ^{2}\text{d}x+2\int_{S}\sigma _{\varnothing }\nabla 
    \widetilde{\varphi }_n\cdot \nabla \psi_n \text{d}x.
\end{equation*}
It is worth noting that
\begin{equation}
\label{eq_inp}
    \int_{S}\sigma _{\varnothing }\left\vert \nabla \psi_n \right\vert ^{2}\text{d}%
    x\leq \int_{\Omega }\sigma _{\varnothing }\left\vert \nabla \psi_n \right\vert
    ^{2}\text{d}x=\left\langle r_n,\Lambda _{\varnothing }r_n\right\rangle \leq
    \lambda _{1}^{\varnothing }\left\Vert r_n\right\Vert^{2},
\end{equation}
where $\lambda _{1}^{\varnothing}$ is the largest eigenvalue of $\Lambda_{\varnothing}$. In addition, it follows that
\begin{eqnarray*}
\left\vert \int_{S}\sigma _{\varnothing }\nabla \widetilde{\varphi }_n\cdot
\nabla \psi \text{d}x\right\vert  \leq \sqrt{\int_{S}\sigma _{\varnothing }\left\vert \nabla \widetilde{%
\varphi }_n\right\vert ^{2}\text{d}x}\times \sqrt{\lambda _{1}^{\varnothing }}%
\left\Vert r_n\right\Vert,
\end{eqnarray*}
as follows by applying the Cauchy-Schwarz inequality and \eqref{eq_inp}.
Therefore, the error in evaluating the ohmic power in $S$ when replacing $f_n$ with $\widetilde{f}_{n}$ is%
\begin{eqnarray*}
\left\vert \int_{S}\sigma _{\varnothing }\left\vert \nabla \varphi_n
\right\vert ^{2}\text{d}x-\int_{S}\sigma _{\varnothing }\left\vert \nabla 
\widetilde{\varphi }_n\right\vert ^{2}\text{d}x\right\vert  
&\leq
&\int_{S}\sigma _{\varnothing }\left\vert \nabla \psi_n \right\vert ^{2}\text{d%
}x+2\left\vert \int_{S}\sigma _{\varnothing }\nabla \widetilde{\varphi }_n%
\cdot \nabla \psi_n \text{d}x\right\vert  \\
&\leq &\lambda _{1}^{\varnothing }\left\Vert r_n\right\Vert^{2}+2\sqrt{%
\lambda _{1}^{\varnothing }}\left\Vert r_n\right\Vert\times \sqrt{%
\int_{S}\sigma _{\varnothing }\left\vert \nabla \widetilde{\varphi }_n%
\right\vert ^{2}\text{d}x}.
\end{eqnarray*}%
By expanding the absolute value at the left-hand side, it follows that
\begin{align*}
   \int_{S}\sigma _{\varnothing }\left\vert \nabla \widetilde{\varphi }_n\right\vert ^{2}\text{d}x 
    -2\sqrt{%
    \lambda _{1}^{\varnothing }}\left\Vert r_n\right\Vert\times \sqrt{%
    \int_{S}\sigma _{\varnothing }\left\vert \nabla \widetilde{\varphi }_n%
    \right\vert ^{2}\text{d}x}
    -\int_{S}\sigma _{\varnothing }\left\vert \nabla \varphi_n\right\vert ^{2}\text{d}x
    -\lambda _{1}^{\varnothing }\left\Vert r_n\right\Vert^{2}
    &\leq 0
     \\
    \int_{S}\sigma _{\varnothing }\left\vert \nabla \widetilde{\varphi }_n\right\vert ^{2}\text{d}x 
    + 2\sqrt{%
    \lambda _{1}^{\varnothing }}\left\Vert r_n\right\Vert\times \sqrt{%
    \int_{S}\sigma _{\varnothing }\left\vert \nabla \widetilde{\varphi }_n%
    \right\vert ^{2}\text{d}x}
    -\int_{S}\sigma _{\varnothing }\left\vert \nabla \varphi_n\right\vert ^{2}\text{d}x
    +\lambda _{1}^{\varnothing }\left\Vert r_n\right\Vert^{2}
    &\geq 0    
\end{align*}
The solution of the two above inequalities with respect to $\sqrt{\int_{S}\sigma _{\varnothing }\left\vert \nabla \widetilde{%
\varphi }_n\right\vert ^{2}\text{d}x}$ yields the following estimate
\begin{eqnarray*}
\sqrt{\int_{S}\sigma _{\varnothing }\left\vert \nabla \varphi_n \right\vert
^{2}\text{d}x}-\sqrt{\lambda _{1}^{\varnothing }}\left\Vert r_n\right\Vert
 &\leq &\sqrt{\int_{S}\sigma _{\varnothing }\left\vert \nabla \widetilde{%
\varphi }_n\right\vert ^{2}\text{d}x} \\
&\leq &\sqrt{\int_{S}\sigma _{\varnothing }\left\vert \nabla \varphi_n
\right\vert ^{2}\text{d}x+2\lambda _{1}^{\varnothing }\left\Vert
r_n\right\Vert^{2}}+\sqrt{\lambda _{1}^{\varnothing }}\left\Vert
r_n\right\Vert
\end{eqnarray*}%
or, to the first order for small residual $\left\Vert r_n\right\Vert _{2}$,
i.e., for $\lambda _{1}^{\varnothing }\left\Vert r_n\right\Vert^{2}\ll
\int_{S}\sigma _{\varnothing }\left\vert \nabla \varphi_n \right\vert ^{2}$d$x,
$ it turns out that
\begin{equation}\label{eqn_bound_pow}
\left\vert \sqrt{\int_{S}\sigma _{\varnothing }\left\vert \nabla \widetilde{%
\varphi }_n\right\vert ^{2}\text{d}x}-\sqrt{\int_{S}\sigma _{\varnothing }\left\vert
\nabla \varphi_n \right\vert ^{2}\text{d}x}\right\vert \leq \sqrt{\lambda
_{1}^{\varnothing }}\left\Vert r_n\right\Vert.  
\end{equation}%

Since $\left\Vert r_n\right\Vert$ can be made arbitrarily small, the
difference on the left-hand side of \eqref{eqn_bound_pow} can also be made arbitrarily small.

\subsection{Approximation of the key ratio}
In this section, the impact of approximating the elements of a kernel sequence, as a linear combination of a finite number of eigenfunctions, on the key ratio appearing in \eqref{eqn_seq_un} is investigated. Specifically, the approximation error is explicitly computed for a kernel sequence $\{ f_n \}_{n\in\mathbb{N}}$ localising the power in a set $S \subseteq \Omega\setminus D^*$, i.e. a kernel sequence such that
\begin{equation*}
    \lim_{n \to +\infty} \frac{\int_S \sigma_{\varnothing}\abs{\nabla\varphi_n}^2\text{d}x}{\langle (\Lambda_D - \Lambda_{\varnothing}) f_n, f_n \rangle}=+\infty.
\end{equation*}

First, it is observed that the inequality \eqref{eqn_bound_pow} implies
\begin{equation}\label{eqn_bound_pow_2}
    \frac{\left\vert\int_{S}\sigma _{\varnothing }\left\vert \nabla \widetilde{\varphi }_n\right\vert ^{2}\text{d}x
    -\int_{S}\sigma _{\varnothing }\left\vert \nabla \varphi_n \right\vert ^{2}\text{d}x\right\vert}
    {\sqrt{\int_{S}\sigma _{\varnothing }\left\vert \nabla \widetilde{%
\varphi }_n\right\vert ^{2}\text{d}x}+\sqrt{\int_{S}\sigma _{\varnothing }\left\vert
\nabla \varphi_n \right\vert ^{2}\text{d}x}} \leq \sqrt{\lambda_{\varnothing}^1}\norm{r_n}.
\end{equation}
Furthermore, considering again \eqref{eqn_bound_pow}
\begin{align*}
     \sqrt{\int_{S}\sigma _{\varnothing }\left\vert \nabla \widetilde{%
\varphi }_n\right\vert ^{2}\text{d}x}+\sqrt{\int_{S}\sigma _{\varnothing }\left\vert
\nabla \varphi_n \right\vert ^{2}\text{d}x} &\leq \sqrt{\lambda
_{1}^{\varnothing }}\left\Vert r_n\right\Vert +2\sqrt{\int_{S}\sigma _{\varnothing }\left\vert
\nabla \varphi_n \right\vert ^{2}\text{d}x} \\ 
     \sqrt{\int_{S}\sigma _{\varnothing }\left\vert \nabla \widetilde{%
\varphi }_n\right\vert ^{2}\text{d}x}+\sqrt{\int_{S}\sigma _{\varnothing }\left\vert
\nabla \varphi_n \right\vert ^{2}\text{d}x} &\geq -\sqrt{\lambda
_{1}^{\varnothing }}\left\Vert r_n\right\Vert +2\sqrt{\int_{S}\sigma _{\varnothing }\left\vert
\nabla \varphi_n \right\vert ^{2}\text{d}x}. 
\end{align*}
Hence, to the first order for small residual $\left\Vert r_n\right\Vert$,
i.e., for $\lambda _{1}^{\varnothing }\left\Vert r_n\right\Vert^{2}\ll
\int_{S}\sigma _{\varnothing }\left\vert \nabla \varphi_n \right\vert ^{2}$d$x$, it turns out that
\begin{equation}\label{eqn_approx_1}
    \sqrt{\int_{S}\sigma _{\varnothing }\left\vert \nabla \widetilde{%
\varphi }_n\right\vert ^{2}\text{d}x}+\sqrt{\int_{S}\sigma _{\varnothing }\left\vert
\nabla \varphi_n \right\vert ^{2}\text{d}x} \approx 2\sqrt{\int_{S}\sigma _{\varnothing }\left\vert
\nabla \varphi_n \right\vert ^{2}\text{d}x}.
\end{equation}
Combining \eqref{eqn_bound_pow_2} and \eqref{eqn_approx_1} yields the following result
\begin{equation}\label{eqn_bound_pow_3}
    \left\vert\int_{S}\sigma _{\varnothing }\left\vert \nabla \widetilde{\varphi }_n\right\vert ^{2}\text{d}x
    -\int_{S}\sigma _{\varnothing }\left\vert \nabla \varphi_n \right\vert ^{2}\text{d}x\right\vert \leq 
    2\sqrt{\lambda_{\varnothing}^1}\norm{r_n} \sqrt{\int_{S}\sigma _{\varnothing }\left\vert \nabla \varphi_n \right\vert ^{2}\text{d}x}.
\end{equation}

Finally, for the key ratio appearing in \eqref{eqn_seq_un}, it follows that
\begin{equation*}
    \begin{split}
        \frac{\int_S \sigma_{\varnothing}\abs{\nabla\widetilde{\varphi}_n}^2\text{d}x}{\langle (\Lambda_D - \Lambda_{\varnothing}) \widetilde{f}_n, \widetilde{f}_n \rangle} & \geq \frac{1}{k_u}\frac{\int_S \sigma_{\varnothing}\abs{\nabla\widetilde{\varphi}_n}^2\text{d}x}
        {\int_D \sigma_{\varnothing}\abs{\nabla\widetilde{\varphi}_n}^2\text{d}x} \\
        & =\frac{1}{k_u}\frac{\int_S \sigma_{\varnothing}\abs{\nabla\varphi_n}^2\text{d}x+\int_S \sigma_{\varnothing}\abs{\nabla\widetilde{\varphi}_n}^2\text{d}x-\int_S \sigma_{\varnothing}\abs{\nabla\varphi_n}^2\text{d}x}
        {\int_D \sigma_{\varnothing}\abs{\nabla\varphi_n}^2\text{d}x+\int_D \sigma_{\varnothing}\abs{\nabla\widetilde{\varphi}_n}^2\text{d}x-\int_D \sigma_{\varnothing}\abs{\nabla\varphi_n}^2\text{d}x} \\
        & \geq \frac{1}{k_u}\frac{\int_S \sigma_{\varnothing}\abs{\nabla\varphi_n}^2\text{d}x}{\int_D \sigma_{\varnothing}\abs{\nabla\varphi_n}^2\text{d}x}
        \frac{1-a_n \norm{r_n}}{1+b_n \norm{r_n}} \\
        & \geq \frac{k_l}{k_u}\frac{\int_S \sigma_{\varnothing}\abs{\nabla\varphi_n}^2\text{d}x}{\langle (\Lambda_D - \Lambda_{\varnothing}) f_n, f_n \rangle}
        \frac{1-a_n \norm{r_n}}{1+b_n \norm{r_n}},
    \end{split}
\end{equation*}

where the first line follows from \eqref{eqn:kmequiv}, the third line follows from \eqref{eqn_bound_pow_3} and
\begin{equation*}
    a_n=\frac{2\sqrt{\lambda_{\varnothing}^1}}{\sqrt{\int_S \sigma_{\varnothing}\abs{\nabla\varphi_n}^2\text{d}x}}, \quad 
    b_n=\frac{2\sqrt{\lambda_{\varnothing}^1}}{\sqrt{\int_D \sigma_{\varnothing}\abs{\nabla\varphi_n}^2\text{d}x}}.
\end{equation*}

Since $\norm{r_n}$ can be made arbitrarily small for any $n$, by choosing $a_n \norm{r_n} < \alpha$ and $b_n \norm{r_n} < \alpha$, it follows that
\begin{equation*}
    \frac{\int_S \sigma_{\varnothing}\abs{\nabla\widetilde{\varphi}_n}^2\text{d}x}{\langle (\Lambda_D - \Lambda_{\varnothing}) \widetilde{f}_n, \widetilde{f}_n \rangle} \ge \frac{k_l}{k_u}\frac{\int_S \sigma_{\varnothing}\abs{\nabla\varphi_n}^2\text{d}x}{\langle (\Lambda_D - \Lambda_{\varnothing}) f_n, f_n \rangle}
        \frac{1-\alpha}{1+\alpha},
\end{equation*}
and, therefore, the claim of \Cref{prop_approx_ratio} follows.

\section{Proof of \Cref{prop_linear}}
\label{proof_2}
In this Appendix, the proof of \Cref{prop_linear} is given.

It follows that
{\small
\begin{equation*}
\begin{split}
    \int_S \sigma_{\varnothing}\abs{\nabla \varphi^n_{\varnothing}}^2\,dx &= \int_S \sigma_{\varnothing} \left(\sum_{k\in \mathcal{I}_n} c_{n,k} \nabla u^k_{\varnothing}\right)\cdot \left(\sum_{h\in \mathcal{I}_n} c_{n,h} \nabla u^h_{\varnothing}\right)\,dx \\
    &=\sum_{k\in \mathcal{I}_n}\sum_{h\in \mathcal{I}_n}c_{n,k} \,c_{n,h} \int_S \sigma_{\varnothing} \nabla u^k_{\varnothing}\cdot\nabla u^h_{\varnothing}\,dx \\
    &\leq \sum_{k\in \mathcal{I}_n}\sum_{h\in \mathcal{I}_n}\abs{c_{n,k}} \abs{c_{n,h}} \sqrt{\int_S \sigma_{\varnothing} \lvert\nabla u^k_{\varnothing}\rvert^2\,dx}\sqrt{\int_S \sigma_{\varnothing} \lvert\nabla u^h_{\varnothing}\rvert^2\,dx}\\
    &=\left(\sum_{k\in \mathcal{I}_n} \abs{c_{n,k}}\sqrt{\int_S \sigma_{\varnothing} \lvert\nabla u^k_{\varnothing}\rvert^2\,dx}\right)^2\\
    &\leq K \left(\sum_{k\in \mathcal{I}_n} \abs{c_{n,k}}\sqrt{\lambda_k}\right)^2 \\
    &\leq K |\mathcal{I}_n| \sum_{k\in \mathcal{I}_n} \abs{c_{n,k}}^2 \lambda_k,
\end{split}
\end{equation*}
}%
where the first line follows from the linearity of the problem \eqref{eqn:dirprob}, the third line follows from the Cauchy-Schwarz inequality, the fifth line follows from
\begin{equation*}
    \int_S \sigo(x) \lvert\nabla u^n_{\varnothing}(x)\rvert^2\,dx\leq K \lambda_n, \quad  n\in |\mathcal{I}_n|,
\end{equation*}
and the sixth line follows from applying the Cauchy-Schwarz inequality again. 

\section{Localised potentials and eigenfunctions}
\label{app_C}
As discussed in \cite[Section 3.2]{Ge08}, a localised potential $f$ able to (i) make the power absorbed by the domain $D$ negligible and (ii) ensure the circulation of an electrical current density in a prescribed point $z\not\in D^*$, can be obtained by solving a regularised version of 
\begin{equation*}
    (\Lambda_D-\Lambda_{\varnothing})f=v_{z,d},
\end{equation*}
where $v_{z,d}= \left. \varphi_{z,d} \right|_{\partial \Omega}$ is the trace, evaluated on $\partial \Omega$, of the electrical scalar potential generated by a dipole current source placed in $z$ and directed along $\mathbf{d}$, with $\abs{\mathbf{d}}=1$, i.e., $\varphi_{z,d}$ is the solution of
\begin{equation*}
    \begin{cases}
        \nabla\cdot(\sigo(x)\nabla \varphi_{z,d}(x))=-\zeta \mathbf{d}\cdot\nabla \delta(x-z) \\
        \sigo(x)\partial_n \varphi_{z,d}(x)=0,
    \end{cases}
\end{equation*}
where $\zeta=\qty{1}{\ampere \metre}$.

Specifically, by means of the Tikhonov regularisation method, a sequence of scalar potentials making the net electrical power in $D$ negligible, while providing a non-vanishing current density in $z$, is given by  (see \cite{Ge08})
\begin{equation*}
    f_{\alpha}=\left[ (\Lambda_D-\Lambda_{\varnothing})^* (\Lambda_D-\Lambda_{\varnothing})+\alpha I\right] ^ {-1} (\Lambda_D-\Lambda_{\varnothing})^* v_{z,d},
\end{equation*}
where the superscript $^*$ denotes the adjoint operator. The sequence of localised potentials $f_{\alpha}$ can be expanded on the basis of the eigenfunctions of $\Lambda_D-\Lambda_{\varnothing}$ as
\begin{equation}\label{eqn_fa}
    f_{\alpha}=\sum_{k=1}^{+\infty}\frac{\lambda_k^2}{\lambda_k^2+\alpha}\frac{\langle v_{z,d}, g_k \rangle}{\lambda_k}g_k.
\end{equation}
The following result holds for $f_{\alpha}$.
\begin{thm}[Theorem 3.6, \cite{Ge08}]
    For every $z\notin D^*$, the solutions $\varphi^{\alpha}_{\varnothing}$ of
    \begin{equation*}
        \nabla\cdot(\sigo\nabla \varphi^{\alpha}_{\varnothing})=0 \: \text{in }\Omega \quad \sigo\partial_n\varphi^{\alpha}_{\varnothing}=f_{\alpha} \: \text{on }\partial\Omega
    \end{equation*}
    fulfils
    \begin{equation*}
        \lim_{\alpha \to 0^+} \frac{\abs{\nabla \varphi^{\alpha}_{\varnothing}(z)}}{\langle (\Lambda_D-\Lambda_{\varnothing})f_{\alpha},f_{\alpha}\rangle^{3/4}} = +\infty, \quad 
        \lim_{\alpha \to 0^+} \frac{\int_D \abs{\nabla\varphi^{\alpha}_{\varnothing}(x)}^2\,dx}{\langle (\Lambda_D-\Lambda_{\varnothing})f_{\alpha},f_{\alpha}\rangle^{3/2}}=0.
    \end{equation*}
\end{thm}
\begin{cor}
    If $\sigma(x)$ is constant in an open neighbourhood $S$ of $z\not\in D^*$, then
    \begin{equation*}
        \lim_{\alpha \to 0^+} \frac{\int_{S}\sigo(x)\abs{\nabla \varphi^{\alpha}_{\varnothing}(x)}^2\,dx}{\langle (\Lambda_D-\Lambda_{\varnothing})f_{\alpha},f_{\alpha}\rangle^{3/2}} = +\infty, \quad \lim_{\alpha \to 0^+} \frac{\int_D \sigo(x)\abs{\nabla \varphi^{\alpha}_{\varnothing}(x)}^2\,dx}{\langle (\Lambda_D-\Lambda_{\varnothing})f_{\alpha},f_{\alpha}\rangle^{3/2}}=0.
    \end{equation*}
\end{cor}
\begin{proof}
    Let $B(z,r)$ be an open ball centred in $z$ of radius $r$, such that $B\subset S$. If $\sigma(x)$ is constant in $B$, then $\mathbf{d}\cdot\nabla \varphi^{\alpha}_{\varnothing}(x)$ is a harmonic function and therefore $\abs{\mathbf{d}\cdot\nabla \varphi^{\alpha}_{\varnothing}(x)}^2$ is sub-harmonic. 
    From the mean value property for sub-harmonic functions \cite{Gilbarg2001}, it follows 
    \begin{equation*}
        \frac{1}{\abs{B}}\int_B \abs{\mathbf{d}\cdot\nabla \varphi^{\alpha}_{\varnothing}(x)}^2\,dx \geq \abs{\mathbf{d}\cdot\nabla \varphi^{\alpha}_{\varnothing}(z)}^2,
    \end{equation*}
    where $\abs{B}$ is the measure of the ball $B$.

    Therefore, it is concluded that
    \begin{equation*}
         \frac{\int_{S}\abs{\nabla\varphi^{\alpha}_{\varnothing}(x)}^2\,dx}{\langle (\Lambda_D-\Lambda_{\varnothing})f_{\alpha},f_{\alpha}\rangle^{3/2}}\geq
         \frac{\int_{B}\abs{\mathbf{d}\cdot\nabla \varphi^{\alpha}_{\varnothing}(x)}^2\,dx}{\langle (\Lambda_D-\Lambda_{\varnothing})f_{\alpha},f_{\alpha}\rangle^{3/2}}\geq
         \abs{B}\frac{\abs{\mathbf{d}\cdot\nabla \varphi^{\alpha}_{\varnothing}(z)}^2}{\langle (\Lambda_D-\Lambda_{\varnothing})f_{\alpha},f_{\alpha}\rangle^{3/2}},
    \end{equation*}
    where the first inequality comes from $\abs{\nabla \varphi_{\varnothing}^{\alpha}}^2\geq \abs{\mathbf{d}\cdot\nabla \varphi^{\alpha}_{\varnothing}}^2$, and $B\subset S$.
    Hence,
    \begin{equation*}
         \lim_{\alpha\to 0^+}\frac{\int_{S} \sigo(x)\abs{\nabla \varphi^{\alpha}_{\varnothing}(x)}^2\,dx}{\langle (\Lambda_D-\Lambda_{\varnothing})f_{\alpha},f_{\alpha}\rangle^{3/2}} \geq
         \sigma_{bg}^m \abs{B} \lim_{\alpha\to 0^+} \frac{\abs{\mathbf{d}\cdot\nabla \varphi^{\alpha}_{\varnothing}(z)}^2}{\langle (\Lambda_D-\Lambda_{\varnothing})f_{\alpha},f_{\alpha}\rangle^{3/2}} = +\infty,
    \end{equation*}
    and
    \begin{equation*}
        \lim_{\alpha \to 0^+} \frac{\int_D \sigo(x)\abs{\nabla\varphi^{\alpha}_{\varnothing}(x)}^2\,dx}{\langle (\Lambda_D-\Lambda_{\varnothing})f_{\alpha},f_{\alpha}\rangle^{3/2}}
        \leq \sigma_{bg}^M\lim_{\alpha \to 0^+} \frac{\int_D \abs{\nabla\varphi^{\alpha}_{\varnothing}(x)}^2\,dx}{\langle (\Lambda_D-\Lambda_{\varnothing})f_{\alpha},f_{\alpha}\rangle^{3/2}}=0.
    \end{equation*}
    
\end{proof}
\begin{rem}\label{rem_ga}
    The sequence 
    \begin{equation*}
        w_{\alpha}=\frac{f_{\alpha}}{\norm{f_{\alpha}}}
    \end{equation*}
    ensures
    \begin{equation*}
        \lim_{\alpha \to 0^+} \frac{\int_{S}\sigo(x)\abs{\nabla \varphi^{\alpha}_{\varnothing}(x)}^2\,dx}{\int_{D}\sigo(x)\abs{\nabla \varphi^{\alpha}_{\varnothing}(x)}^2\,dx} =+\infty,
    \end{equation*}
    where $\varphi^{\alpha}_{\varnothing}$ solves
    \begin{equation*}
        \nabla\cdot(\sigo\nabla \varphi^{\alpha}_{\varnothing})=0 \: \text{in }\Omega \quad \text{and}\quad \sigo\partial_n\varphi^{\alpha}_{\varnothing}=w_{\alpha} \: \text{on }\partial\Omega.
    \end{equation*}

    Furthermore, the sequence ${w_{\alpha}}$ is a kernel sequence, as follows using the arguments in the proof of \Cref{thm_un}.
\end{rem}
\begin{rem}\label{app_alp}
    The sequence $w_n=w_{\alpha_n}$, where $\alpha_n \to 0$ as $n \to +\infty$, constitutes an ALP sequence, as defined in \Cref{sec_ti}.
\end{rem}

\subsection{Assumptions}
Common assumptions for the well-known Factorization Method are (see \cite{fm_2}, for instance):
\begin{align}
    \ln(\lambda_k)&\approx \ln (c_{\lambda})+k\ln(q_{\lambda}) \label{asp_app_1} \\
    \ln(\langle v_{z,d},g_k \rangle^2)& \approx \ln( c_{z})+k\ln(p_z) \label{asp_app_2}.
\end{align}

The assumptions introduced in \Cref{sec_ti} are directly derived from \eqref{asp_app_1} and \eqref{asp_app_2}. Indeed, from \cite{Ge08}
\begin{equation*}
    \langle v_{z,d},g_k \rangle =\mathbf{d}\cdot\nabla u^k_{\varnothing}(z),
\end{equation*}
and, hence, in a 2D setting,
\begin{equation*}
    \abs{\nabla u^k_{\varnothing}(z)}^2=
    \left(\hat{\mathbf{i}}_x\cdot  \nabla u^k_{\varnothing}(z) \right)^2
    +\left(\hat{\mathbf{i}}_y\cdot \nabla u^k_{\varnothing}(z) \right)^2
    =\langle v_{z,i_x},g_k \rangle^2+\langle v_{z,i_y},g_k \rangle^2,
\end{equation*}
therefore, it follows that
\begin{equation}\label{eqn_log_1}
    \ln\left(\sigo(z) \abs{\nabla u^k_{\varnothing}(z)} \right)=\ln\left(\sigo(z)\right)+ \ln\left(\langle v_{z,i_x},g_k \rangle^2+\langle v_{z,i_y},g_k \rangle^2\right).
\end{equation}
From \eqref{eqn_aps_2}
\begin{align*}
    \ln( \langle v_{z,i_x},g_k \rangle^2 ) & \approx \ln( c_{z,x})+k\ln(p_{z,x}), \\
    \ln( \langle v_{z,i_y},g_k \rangle^2 ) & \approx \ln( c_{z,y})+k\ln(p_{z,y}). \\
\end{align*}
If $p_{z,x}=p_{z,y}$, assumption \eqref{asp_pbg} follows directly from \eqref{eqn_log_1}. For $p_{z,x}>p_{z,y}$
\begin{equation*}
    \ln\left(\langle v_{z,i_x},g_k \rangle^2+\langle v_{z,i_y},g_k \rangle^2\right)=\ln\left(\langle v_{z,i_x},g_k \rangle^2\right)+\ln\left(1+\frac{\langle v_{z,i_y},g_k \rangle^2}{\langle v_{z,i_x},g_k \rangle^2}\right),
\end{equation*}
that, for large $k$, reduces to
\begin{equation*}
    \ln(\sigo(z) \abs{\nabla u^k_{\varnothing}(z)}^2)\sim \ln (\sigo(z)) + \ln\left(\langle v_{z,i_x},g_k \rangle^2\right) \approx \ln( c_z)+k\ln(p_{z,x}),
\end{equation*}
and analogously for $p_{z,x}<p_{z,y}$.

As a final remark, it is worth noticing that in \Cref{sec_ti}, the exponential decays are assumed to hold asymptotically, in coherence with the above derivation.

\subsection{Proof of \Cref{prop_loc_eig}}
Let $S\Subset\Omega$ be a set that fulfils \eqref{eqn_seq_un} for the sequence of eigenfunctions, i.e.,
\begin{equation*}
    \int_S \sigo(x)\abs{\nabla u^k_{\varnothing}(x)}^2\, dx \leq K\lambda_k,\quad \forall\,k\in\mathbb{N}.
\end{equation*}
Let $w_n$ be the sequence of localised potentials defined in \Cref{app_alp}, for an arbitrary but prescribed direction $\mathbf{d}$. This sequence can be expanded on the eigenfunctions $g_k$ as follows:
\begin{equation*}
    w_n=\frac{1}{\sqrt{\sum_{k=1}^{+\infty} \abs{c_{n,k}}^2}}\sum_{k=1}^{+\infty} c_{n,k}g_k,
\end{equation*}
where (see Equation \eqref{eqn_fa})
\begin{equation}\label{eqn_c_alpha}
    c_{n,k}=\frac{\lambda_k^2}{\lambda_k^2+\alpha_n}\frac{\langle v_{z,d},g_k \rangle}{\lambda_k}.
\end{equation}
Preliminarily, it is observed that
\begin{equation*}
    \frac{\int_S \sigo(x)\abs{\nabla \varphi^{n}_{\varnothing}(x)}^2\, dx}{\langle (\Lambda_D-\Lambda_{\varnothing})w_n,w_n\rangle}\leq
    K \frac{\left(\sum_{k=1}^{+\infty} \abs{c_{n,k}} \sqrt{\lambda_k}\right)^2}{\sum_{k=1}^{+\infty} \abs{c_{n,k}}^2 \lambda_k},
\end{equation*}
where the upper bound for the numerator is obtained with the same technique as \Cref{proof_2}. Therefore, the claim of \Cref{prop_loc_eig} is achieved by proving that the ratio
\begin{equation}\label{eqn_ratio}
    \frac{\left(\sum_{k=1}^{+\infty} \abs{c_{n,k}} \sqrt{\lambda_k}\right)^2}{\sum_{k=1}^{+\infty} \abs{c_{n,k}}^2 \lambda_k}
\end{equation}
is bounded for an ALP sequence. 

Substituting \eqref{eqn_c_alpha} in the round brackets of \eqref{eqn_ratio} gives
\begin{equation}
\label{eqn_aps_1}
     I_n \coloneqq \sum_{k=1}^{+\infty} \frac{\lambda_k^2}{\lambda_k^2+\alpha_n}\frac{\abs{\langle v_{z,d},g_k\rangle}}{\sqrt{\lambda_k}}.
\end{equation}
An upper bound to the series of \eqref{eqn_aps_1}, for a prescribed $n$, can be obtained by splitting the sum into two parts: the sum of the terms for which $\lambda_k^2\leq \alpha_n/\gamma$, where $\gamma \gg 1$, and the sum made up of the remaining terms. It is worth noting that
\begin{equation*}
    \lambda_k^2\leq \alpha_n / \gamma \iff k \geq k_n +\beta+\delta ,
\end{equation*}
where 
\begin{equation*}
    k_n= \frac{\ln \alpha_n}{2\ln q_{\lambda}},\quad 
    \beta=-\frac{\ln c_{\lambda}}{\ln q_{\lambda}},\quad
    \delta=-\frac{\ln \gamma}{2\ln q_{\lambda}}.
\end{equation*}


The sum in \eqref{eqn_aps_1} is split as
\begin{equation}
\label{eqn_aps_2}
    I_n = \sum_{k=1}^{\lceil k_n+\beta+\delta \rceil-1} \frac{\lambda_k^2}{\lambda_k^2+\alpha_n} \frac{\abs{\langle v_{z,d},g_k\rangle}}{\sqrt{\lambda_k}}
    \ +\sum_{k=\lceil k_n+ \beta + \delta \rceil}^{+\infty} \frac{\lambda_k^2}{\lambda_k^2+\alpha_n}\frac{\abs{\langle v_{z,d},g_k\rangle}}{\sqrt{\lambda_k}},
\end{equation}
where $\lceil x \rceil=\min \{n\in\mathbb{N} \mid n\geq x\}$.

The first term in \eqref{eqn_aps_2} can be upper-bounded as follows
\begin{equation*}
\begin{split}
    \sum_{k=1}^{\lceil k_n+ \beta + \delta \rceil -1} \frac{\lambda_k^2}{\lambda_k^2+\alpha_n}\frac{\abs{\langle v_{z,d},g_k\rangle}}{\sqrt{\lambda_k}} & 
    \leq \sum_{k=1}^{\lceil k_n+ \beta + \delta \rceil -1} \frac{\abs{\langle v_{z,d},g_k\rangle}}{\sqrt{\lambda_k}}\\
    & \approx 
    \sqrt{\frac{c_z}{c_{\lambda}}}\sum_{k=1}^{\lceil k_n + \beta + \delta \rceil -1} \left(\frac{p_z}{q_{\lambda}}\right)^{k/2} \\
    &= \sqrt{\frac{c_z}{c_{\lambda}}} r\frac{1-r^{\lceil k_n+ \beta + \delta \rceil -1}}{1-r},
\end{split}
\end{equation*}
where $r=\sqrt{p_z/q_{\lambda}}$.
In the second term of \eqref{eqn_aps_2} $\lambda_k^2$ is negligible with respect to $\alpha_n$, hence
\begin{equation*}
\begin{split}
    \sum_{k=\lceil k_n+ \beta + \delta \rceil}^{+\infty} \frac{\lambda_k^2}{\lambda_k^2+\alpha_n}\frac{\abs{\langle v_{z,d},g_k\rangle}}{\sqrt{\lambda_k}}
    &\approx 
    \frac{1}{\alpha_n} \sum_{k=\lceil k_n+ \beta + \delta \rceil}^{+\infty} \lambda_k^{3/2}\abs{\langle v_{z,d},g_k\rangle}\\
    &\approx
    \frac{c_{\lambda}^{3/2}c_z^{1/2}}{\alpha_n} \sum_{k=\lceil k_n+ \beta + \delta \rceil}^{+\infty} s^k\\
    &= \frac{c_{\lambda}^{3/2}c_z^{1/2}}{\alpha_n}\frac{s^{\lceil k_n+ \beta + \delta \rceil}}{1-s},
\end{split}
\end{equation*}
where $s=q_{\lambda}^{3/2}p_z^{1/2}$, and it has been exploited that $s<1$ (see \Cref{rem_s1}). 

In summary, an upper bound for $I_n$ is
\begin{equation}
\label{eqn_ub_n}
    I_n \le 
    \sqrt{\frac{c_z}{c_{\lambda}}}r\frac{1-r^{\lceil k_n+ \beta + \delta \rceil-1}}{1-r}
    +\frac{c_{\lambda}^{3/2}c_z^{1/2}}{\alpha_n}\frac{s^{\lceil k_n+ \beta + \delta \rceil}}{1-s}.
\end{equation}

The upper bound appearing in the right-hand side of \eqref{eqn_ub_n} depends on $n$ through the three terms $r^{\lceil k_n + \beta + \delta \rceil}$, $s^{\lceil k_n+ \beta + \delta \rceil}$ and $\alpha_n$. This expression is involved, but it can be cast as the product of a bounded function multiplied by $\alpha_n^{\frac{\ln r}{2 \ln q_{\lambda}}}$. To obtain this form for the upper bound, $\lceil k_n+ \beta + \delta \rceil$ is conveniently expressed as
\begin{equation}
    \lceil k_n+ \beta + \delta \rceil=k_n+ \beta + \delta +\varepsilon_n,
\end{equation}
where $0\leq \varepsilon_n < 1$. Consequently, it follows that
\begin{equation}\label{eqn_r1}
    r^{\lceil k_n+ \beta + \delta \rceil}
    = r^{\beta +\delta + \varepsilon_n} \  \alpha_n^{\frac{\ln r}{2 \ln q_{\lambda}}}, \quad         \frac{s^{\lceil k_n+ \beta + \delta \rceil}}{\alpha_n}=s^{\beta +\delta +\varepsilon_n} \  \alpha_n^\frac{\ln r}{2\ln q_{\lambda}}.
\end{equation}
where it has been exploited that $r^{\frac{\ln \alpha_n}{2 \ln q_\lambda}}=\alpha_n^{\frac{\ln r}{2 \ln q_\lambda}}$. If $r<1$, then $\ln r / \ln q_{\lambda} >0$ and the bound at the right-hand side of \eqref{eqn_ub_n} for large $n$ is
\begin{equation}
    \sqrt{\frac{c_z}{c_{\lambda}}} \frac{r}{1-r} + \mathcal{O} \left( \alpha_n^\frac{\ln r}{2\ln q_{\lambda}} \right), \ n \to +\infty.
\end{equation}
For $r>1$, $r^{\lceil k_n+ \beta + \delta \rceil-1} \gg 1$ for large $n$ and, therefore, \eqref{eqn_ub_n} reduces to
\begin{equation*}
   I_n \le b_{1,n} \alpha_n^\frac{\ln r}{2\ln q_{\lambda}} \quad \text{as } n \to +\infty,
\end{equation*}
where
\begin{equation*}
    b_{1,n}=\sqrt{\frac{c_z}{c_{\lambda}}}\frac{r^{\delta+\beta +\varepsilon_n}}{r-1}
    +\frac{s^{\delta + \beta +\varepsilon_n}c_{\lambda}^{3/2}c_z^{1/2} }{1-s}.
\end{equation*}
In addition, $b_{1,n}$ is bounded, i.e. $b_{1,n} \leq b_1^{max} < + \infty$, $\forall\, n\in\mathbb{N}$.

Regarding the denominator of \eqref{eqn_ratio}, a similar argument is applied. Specifically, the terms of the series are divided into two groups: (i) those for which $\lambda_k^2 \leq \alpha_n/\gamma$, and (ii) the remaining terms. Here, $\gamma \gg 1$ is the same constant introduced to evaluate the upper bound to the numerator of \eqref{eqn_ratio}.
For $k\geq k_n+ \beta +\delta$, it follows that $\lambda_k^2 \le \alpha_n / \gamma \ll \alpha_n$, and
\begin{equation*}
\begin{split}
     \sum_{k=1}^{+\infty} \left(\frac{\lambda_k^2}{\lambda_k^2+\alpha_n}\right)^2\frac{\abs{\langle v_{z,d},g_k\rangle}^2}{\lambda_k}
    & \ge
    \sum_{k=\lceil k_n+ \beta + \delta \rceil}^{+\infty} \left(\frac{\lambda_k^2}{\lambda_k^2+\alpha_n}\right)^2\frac{\abs{\langle v_{z,d},g_k\rangle}^2}{\lambda_k}\\
    & \approx \frac{1}{\alpha_n^2}\sum_{k=\lceil k_n+ \beta + \delta \rceil }^{+\infty} \lambda_k^3\abs{\langle v_{z,d},g_k\rangle}^2\\
    &\approx \frac{c_{\lambda}^3 c_z}{1-s^2} (s^2)^{\beta + \delta + \varepsilon_n} \alpha_n^{\frac{\ln r}{\ln q_{\lambda}}}\\
    & \ge b_2^{min} \alpha_n^{\frac{\ln r}{\ln q_{\lambda}}},
\end{split}
\end{equation*}
where
\begin{equation*}
    \begin{split}
        b_2^{min} & = \frac{c_{\lambda}^3 c_z}{1-s^2}(s^2)^{\beta + \delta+1}.
\end{split}
\end{equation*}
Finally,
\begin{equation*}
    \frac{\left(\sum_{k=1}^{+\infty} \abs{c_{\alpha,k}} \sqrt{\lambda_k}\right)^2}{\sum_{k=1}^{+\infty} \abs{c_{\alpha,k}}^2 \lambda_k} 
    \leq \frac{(b_1^{max})^2\alpha_n^{\frac{\ln r}{\ln q_{\lambda}}}}{b_2^{min}\alpha_n^{\frac{\ln r}{\ln q_{\lambda}}}}=\frac{(b_1^{max})^2}{b_2^{min}}.
\end{equation*}

\begin{rem}\label{rem_s1}
 The condition $s=q_{\lambda}^{3/2}p_z^{1/2}<1$ follows directly from the fact that the sequences $\{\lambda_k\}_k$ and $\{\langle v_{z,d},g_k\rangle^2\}_k$ both converge to zero. Indeed, convergence to zero implies that $q_{\lambda}<1$, $p_z<1$, which immediately guarantees $s<1$.
\end{rem}

\clearpage

\bibliographystyle
{iopart-num}
\bibliography{biblioCFPPT}
\end{document}